\documentclass[times,sort&compress,3p]{elsarticle} 
\journal{Journal of Multivariate Analysis}

\usepackage{amsmath}
\usepackage{enumerate}
\usepackage{natbib}
\usepackage{url} 

\usepackage{amsmath}
\usepackage{amssymb}
\usepackage{amssymb,color}
\usepackage{graphicx}
\usepackage{soul}
\usepackage{hyperref}
\usepackage{natbib} 
\usepackage{graphicx}
\usepackage{subcaption}
\usepackage{setspace}
\usepackage{algorithm,algcompatible}
\usepackage{algpseudocode}
\usepackage{tikz}
\usepackage{bbm}
\usepackage{comment}
\usepackage{rotating}
\usepackage{pgfplots}
\usetikzlibrary{positioning,calc}
\usepackage{epstopdf} 
\usepackage{thmtools} 
\usepackage{tikz}
\usepackage{pgfplots}

\usepackage[standard]{ntheorem}

\usepackage{epsfig}

\usepackage{natbib}

\usepackage{tikz}
 \tikzstyle{Noteimportant} =[fill=white, very thick, draw=red, rectangle, rounded corners,
 		inner xsep=5pt, inner ysep=5pt]
 \tikzstyle{Note} =[fill=white, very thick, draw=green, rectangle, rounded corners,
 		inner xsep=5pt, inner ysep=5pt]

\algnewcommand\algorithmicto{\textbf{to}}

\DeclareMathOperator*{\argmin}{arg\,min}

\usepackage{dsfont}
\newcommand{\Ind}{\ensuremath{{\mathds{1}}}} 

\newcommand{\Var}{\ensuremath{{\mathbb V}}} 
\newcommand{\E}{\ensuremath{{\mathbb E}}} 
\newcommand{\Prob}{\ensuremath{{\mathbb P}}} 
\newcommand{\N}{\ensuremath{{\mathbb N}}}
\newcommand{\R}{\ensuremath{{\mathbb R}}}
\newcommand{\F}{\ensuremath{{\mathcal F}}}

\newcommand{\A}{\ensuremath{{\mathcal A}}}

\newcommand{\X}{\ensuremath{{\mathcal X}}}

\newcommand{\TV}{\ensuremath{{\mbox TV}}}
\newcommand{\sq}{\ensuremath{{\mbox sq}}}

\begin{document}
\begin{frontmatter}
\title{High Probability Lower Bounds for the Total Variation Distance}

\author[ETH]{Jeffrey N{\"a}f\corref{mycorrespondingauthor}}
\author[ETH]{Loris Michel}
\author[ETH]{Nicolai Meinshausen\corref{thanks}}

\address[ETH]{Seminar for Statistics, ETH Z\"{u}rich, Switzerland}

\cortext[thanks]{The authors would like to thank Sebastian Sippel for providing and maintaining the NCEP Reanalysis 2 data.}
\cortext[mycorrespondingauthor]{Corresponding author. E-mail address: \url{jeffrey.naef@stat.math.ethz.ch}}

\begin{abstract}
The statistics and machine learning communities have recently seen a growing interest in classification-based approaches to two-sample testing. The outcome of a classification-based two-sample test remains a rejection decision, which is not always informative since the null hypothesis is seldom strictly true. Therefore, when a test rejects, it would be beneficial to provide an additional quantity serving as a refined measure of distributional difference. In this work, we introduce a framework for the construction of high-probability lower bounds on the total variation distance. These bounds are based on a one-dimensional projection, such as a classification or regression method, and can be interpreted as the minimal fraction of samples pointing towards a distributional difference. We further derive asymptotic power and detection rates of two proposed estimators and discuss potential uses through an application to a reanalysis climate dataset.
\end{abstract}

\begin{keyword}
two-sample testing, distributional difference,  classification, higher-criticism 
\end{keyword}

\end{frontmatter}

\section{Introduction} \label{sec:intro}

Two-sample testing is a classical statistical task recurring in various scientific fields.
Based on two samples $X_i$, $i=1,\ldots,m$  and $Y_j$, $j=1,\ldots,n$ drawn respectively from probability measures $P$ and $Q$, the goal is to test the hypothesis $H_0: P=Q$, against potentially any alternative. The trend in the last two decades towards the analysis of more complex and large-scale data has seen the emergence of classification-based approaches to testing. Indeed, the idea of using classification for two-sample testing traces back to the work of \citet{Friedman2004OnMG}. Recently this use of classification has seen a resurgence of interest from the statistics and machine learning communities with empirical and theoretical work (\cite{DBLP}; \cite{Rosen2016};  \cite{facebookguys}; \cite{hediger2020use}; \cite{BORJI2019}; \cite{gagnon-bartsch2019}; \cite{kim2019}; \cite{Cal2020}) motivated by broader applied scientific work as well-explained in \cite{DBLP}.  

However, as already pointed out in \citet{Friedman2004OnMG}, it is  practically very unlikely that two samples come from the exact same distribution. It means that with enough data and using a  ``universal learning machine'' for classification, as Friedman called it, the null will be rejected no matter how small the difference between $P$ and $Q$ is. Therefore, in many situations when a classification-based two-sample test rejects, it would be beneficial to have an additional measure quantifying the actual distributional difference supported by the data. 

Practically, one can observe that with two finite samples some fraction of observations will tend to illuminate a distributional difference more than others. At a population level, this translates to the fraction of probability mass one would need to change from $P$ to see no difference with $Q$. It is well known that this is an equivalent characterization of the total variation distance between $P$ and $Q$, see e.g. \cite{YuvalMixing}. We recall that for two probability distributions $P$ and $Q$ on measurable space $(\mathcal{X}, \A)$ the \emph{total variation (TV) distance}  is defined as 

$$\TV(P,Q) = \sup_{A \in \A}\left| P(A)-Q(A) \right |.$$

\noindent Therefore, based on finite samples of $P$ and $Q$, a finer question than ``is $P$ different from $Q$ ?'' could be stated as ``What is a probabilistic lower bound on the fraction of observations actually supporting a difference in distribution between $P$ and $Q$ ?''. This would formally translate into the construction of an estimate $\hat{\lambda}$ satisfying 
$$\mathbb{P}(\hat{\lambda} > \TV(P,Q)) \leq \alpha,$$ 
for $\alpha \in (0,1)$. We call such an estimate $\hat{\lambda}$ a high-probability lower bound (HPLB) for $\TV(P,Q)$. An observation underlying our methodology is that uni-dimensional projections of distributions act monotonically on the total variation distance. Namely for a given (measurable) projection $\rho: \mathcal{X} \rightarrow I \subseteq \mathbb{R}$,

\begin{equation}\label{gap}
\TV\left(\rho_{\#}P,\rho_{\#}Q\right) \leq \TV \left(P,Q \right),
\end{equation}

\noindent where $\rho_{\#}P$ is the push-forward measure of $P$, defined as $\rho_{\#}P(A):=P(\rho^{-1}(A))$ for a measurable set $A$. The construction of an HPLB for a given projection $\rho$ is the focus of our work. They are used as a proxies for $\TV \left(P,Q \right)$ through (\ref{gap}). The gap   depends on the informativeness of the selected projection $\rho$ about the distributional difference between $P$ and $Q$. This naturally established a link with classification and provides insights on how to look for ``good projections''. Nevertheless, the focus of the present paper is not to derive conditions on how to construct optimal projections $\rho$, but rather an analysis of the construction and properties of HPLBs for fixed projections. 
As a by-product, we address an issue that seems to have gone largely unnoticed in the literature on classification and two-sample testing. Namely, given a function $\rho: \X \to [0,1]$ estimating the probability of belonging to the first sample say, what is the ``cutoff'' $t^*$ allowing for the best possible detection of distributional difference for the binary classifier
\begin{equation}\label{binarybase}
    \rho_{t}(Z):=\Ind\{\rho(Z) > t \}.
\end{equation}
In line with the Bayes classifier, $t^*=1/2$ is often used in classification tasks. However, we show that for the detection of distributional difference, this is not always the best choice. The next section illustrates this issue through a toy example.

\subsection{Toy motivating example}\label{toyexamplesec}

As an illustrative example highlighting the importance of the choice of cutoff for a binary classifier, let us consider two probability distributions $P$ and $Q$ on $\mathbb{R}^{12}$ with mutually independent margins defined as follows:  $P=\mathcal{N}_{12}(\mu, \Sigma)$ where $\Sigma=I$ is the $12 \times 12$ identity matrix and $\mu=( 0, \ldots, 0 )$ and $Q=(1-\varepsilon) P + \varepsilon C$, where $C=\mathcal{N}_{12}(\mu_C, \Sigma)$, with $\mu_C=( 3, 3, 0, \ldots, 0 )$ and $\varepsilon = 10^{-2}$. We assume to observe an iid sample $X_1,\ldots, X_n$ from $P$ and $Y_1,\ldots,Y_n$ from $Q$. Figure \ref{fig:example} shows the projection of samples from P and Q on the first two components.

Consider $\rho: \R^{12} \to [0,1]$ to be a function returning an estimate of the probability that an observation $Z$ belongs to the sample of $Q$, obtained for instance from a learning algorithm trained on independent data. Assume we would like to test whether there is a significant difference between a sample of $P$ and a sample from $Q$ based on the binary classifier $\rho_t$ as defined in \eqref{binarybase}.

Table \ref{tab:example} (left) presents the confusion matrix obtained from a Random Forest classifier $\rho_{t}$ trained on $n=10'000$ samples of $P$ and $Q$ (as defined above) with the usual cutoff of $t=0.5$. Based on this matrix, one can use a permutation approach to test $H_0: P=Q$. The corresponding p-value is $0.65$, showing that, despite the high sample size, the classifier is not able to differentiate the two distributions.
 
 \begin{table}
        \centering
    \begin{tabular}{l|l|c|c|c}
\multicolumn{2}{c}{}&\multicolumn{2}{c}{True Class}&\\
\cline{3-4}
\multicolumn{2}{c|}{}&0&1&\multicolumn{1}{c}{}\\
\cline{2-4} 
& 0 & 5121 & 5150  \\
\cline{2-4}    
& 1 & 4879  & 4850  \\
\cline{2-4}
\end{tabular}
  \begin{tabular}{l|l|c|c|c}
\multicolumn{2}{c}{}&\multicolumn{2}{c}{True Class}&\\
\cline{3-4}
\multicolumn{2}{c|}{}&0&1&\multicolumn{1}{c}{}\\
\cline{2-4} 
& 0 & 9987 & 9925  \\
\cline{2-4}    
& 1 & 13  & 75  \\
\cline{2-4}
\end{tabular}
\caption{Confusion matrices for $2$ different thresholds, $t=0.5$ (left) and $t=0.7$ (right). P-values are given as $0.65$ and $9 \times 10^{-4}$ respectively. \label{tab:example}}
 \end{table}
 
 However this changes if we instead use a cutoff of $t=0.7$. Using the same permutation approach, we obtain a p-value of $9 \times 10^{-4}$. The corresponding confusion matrix is displayed on Table \ref{tab:example} (right).

This observation supports that even though $t=0.5$ links to the optimal Bayes rate, depending on the choice of alternative (how $P$ and $Q$ differ), different cutoffs induce vastly different detection powers. Put differently, $t=0.5$ is not always optimal for detecting a change in distribution. In this work, we will explore why this is the case and how this impacts the construction of HPLBs for $\TV(P,Q)$ and their (asymptotic) statistical performances. As an empirical illustration, we show at the end of Section \ref{TVsearchec} that, for the same simulation setting, the HPLB based on a cutoff of 0.5 will be zero, while the one that adaptively chooses the ``optimal'' cutoff will be positive.

\begin{figure}
    \centering
      \includegraphics[height=7cm]{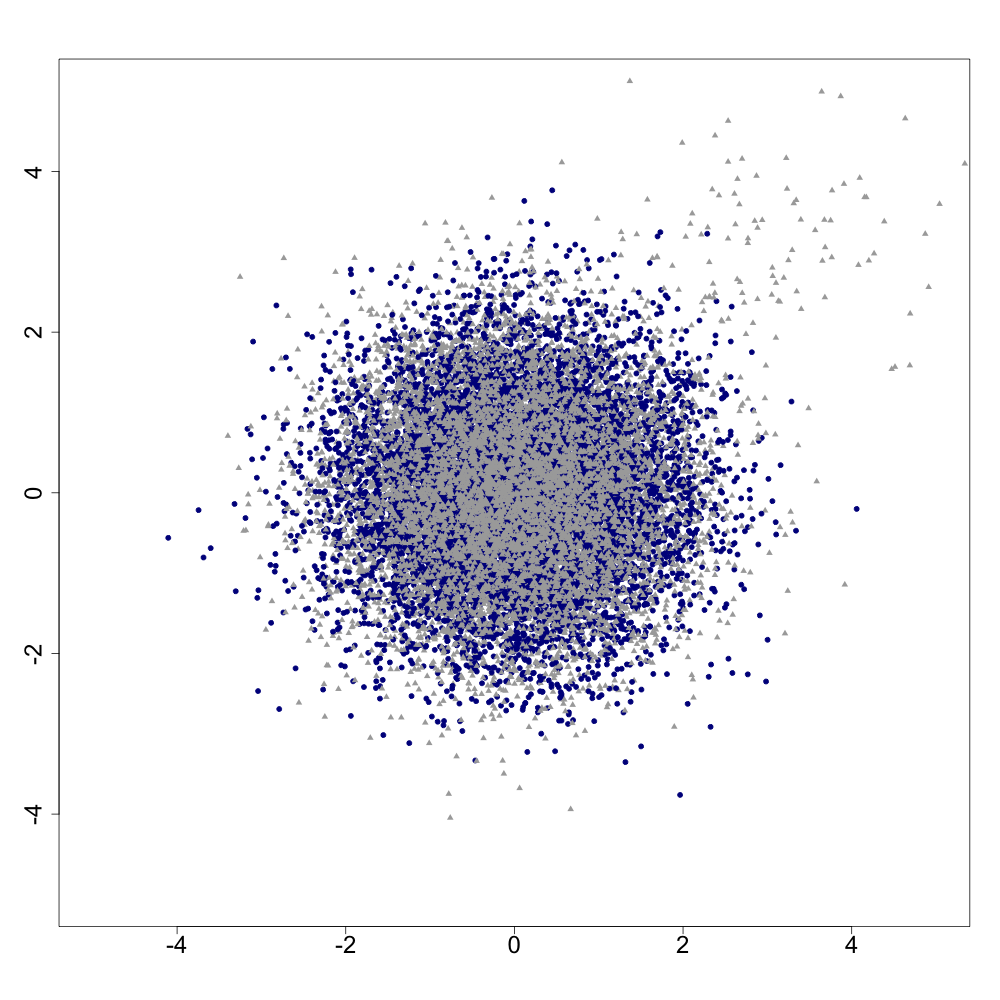}
    \caption{Projections on first two margins of $10'000$ samples from $P$ (blue) and $Q$ (grey).\label{fig:example}}
\end{figure}

\subsection{Contribution and relation to other work}

Direct estimators or bounds for the total variation distance have been studied in previous work when the distributions are assumed to be discrete or to belong to a known given class  (e.g. \cite{ValiantNips}; \cite{sason2015upper}; \cite{Jiao2016}; \cite{devroye2018total}; \cite{kosov2018total}; \cite{Nielsen2018}). Our work aims at constructing lower bounds on the total variation distance based on samples from two unknown distributions. The goal is to provide additional information over the rejection status of classification-based two-sample tests. We summarize our contributions as follows:

\begin{itemize}
    \item[] \textbf{Construction of HPLBs:} We provide a framework for the construction of high probability lower bounds for the total variation distance based on (potentially unbalanced) samples and propose two estimators derived from binary classification. The first estimator $\hat{\lambda}_{bayes}$, assumes the fixed cutoff $1/2$, whereas the second one $\hat{\lambda}_{adapt}$, is cutoff-agnostic. Despite the somewhat complicated nature of the latter estimator, we show that is a valid HPLB.
    \item[] \textbf{Asymptotic detection and power boundaries:} We characterize power and detection rates for the proposed estimators for local alternatives with decaying power rate $\TV(\rho_{\#}P_N,\rho_{\#}Q_N) \propto N^{-\gamma}$, $-1 < \gamma < 0$, for $N=m+n$. We  summarize the main result as follows:  
    Consider the minimal rate $-1 < \gamma < 0$ for which a difference in $P$ and $Q$ could still be detected, if the optimal cutoff $t^*$ in \eqref{binarybase} for a given $P$, $Q$ was known -- this will be referred as the ``oracle rate''. The estimator $\hat{\lambda}_{adapt}$ always attains the oracle rate, whereas $\hat{\lambda}_{bayes}$ only attains the oracle rate if the optimal cutoff is actually $t^*=1/2$. We also obtain the same favorable results for $\hat{\lambda}_{adapt}$ when considering a sequence $\hat{\rho}_N$, estimated on independent data.
    \item[] \textbf{Application:} We show the potentially use and efficacy of HPLBs on the total variation distance in two different types of applications based on a climate reanalysis dataset. 
    
    \item[] \textbf{Software:} We provide implementations of the proposed estimators in the R-package $\texttt{HPLB}$ available on CRAN.
\end{itemize}   

From a technical point of view, the construction of our lower bound estimators relate to the higher-criticism literature (\cite{donoho2004}; \cite{donoho2015}) and is inspired by similar methodological constructions of high-probability lower bounds in different setups (e.g. \cite{Nicolai2005}; \cite{Nicolai2006}; \cite{meinshausen2006}). It has also some similarities with the problem of semi-supervised learning in novelty detection (\cite{Blanchard2010}).

\noindent The paper is structured as follows. Section 2 introduces the classification framework for constructing lower bounds on total variation distance and describes our proposed estimators. Power and detection rates guarantees of our proposed estimators are presented in Section 3. In Section 4 we generalize our framework beyond binary classification and in Section 5 we present three different applications of our estimators in the context of a climate dataset.

\section{Theory and methodology}\label{sec:meth}


Let $P$ and $Q$ be two probability measures on $\X$ and $\rho: \X \to I$ be any measurable function mapping to some subset $I \subset \mathbb{R}$. If not otherwise stated, we consider $I=[0,1]$. The following chain of inequalities for $\TV(P,Q)$ will form the starting point of our approach:
\begin{align}\label{Basicinequality}
    \TV(P,Q) \geq \TV(\rho_{\#}P,\rho_{\#}Q) \geq \sup_{t \in I} \left| \rho_{\#}P \left( (0,t] \right) - \rho_{\#}Q \left( (0,t] \right) \right|. 
\end{align}
For any such $\rho$, one can define a binary classifier $\rho_t: \X  \to \{0,1 \}$ based on the cutoff $t$ by $\rho_t(z):=\Ind{\{ \rho(z) > t \}}$.  Before diving into more details, let us introduce our setup and some necessary notation. 

 \textbf{Setup and notation:} Where not otherwise stated, we assume to observe two independent iid samples $X_1,\ldots, X_m$ from $P$ and $Y_1, \ldots, Y_n$ from $Q$. We define

$$Z_i := \begin{cases}
X_{i} \hspace{0.5cm} \mbox{   \text{if} } 1 \leq i \leq m, \\
Y_{i} \hspace{0.5cm} \mbox{   \text{if} } m+1 \leq i \leq m+n,
\end{cases}$$

\noindent and attach a label $\ell_i:=0$ for $i=1,\ldots, m$ and $\ell_i :=1$, for $i=m+1,\ldots, m+n$. Both $m$ and $n$ are assumed to be non-random with $N=m+n$ such that $m/N \to \pi \in (0,1)$, as $N \to \infty$. For notational convenience, we also assume that $m \leq n$. We denote by $f$ and $g$ the densities of $P$, $Q$ respectively.\footnote{Wlog, we assume that the densities exist with respect to some common dominating measure.} We define $F := \rho_{\#}P$ and $G := \rho_{\#}Q$ and introduce the empirical measures
\begin{align*}
    \hat{F}_m(t) :=\frac{1}{m} \sum_{i=1}^m  \Ind{\{ \rho(X_i) \leq t \}} , \  \ \ \hat{G}_n(t) :=\frac{1}{n} \sum_{j=1}^n \Ind{\{ \rho(Y_j) \leq t \}} , 
\end{align*}
of all observations $\left\{\rho_i\right\}_{i=1,\ldots,N}$ := $\left\{\rho(Z_i)\right\}_{i=1,\ldots,N}$. Denote by $\rho_{(z)}$, $z \in \{1,\ldots, N\}$, the $z^{th}$ order statistic of  $\left( \rho_i \right)_{i=1,\ldots,N}$. Throughout the text, $ q_{\alpha}(p,m)$ is the $\alpha$-quantile of a binomial distribution with success probability $p$ and number of trials $m$ symbolized by $\mbox{Binomial}(p, m)$. Similarly, $q_{\alpha}$ is the $\alpha$-quantile of a standard normal distribution, denoted $\mathcal{N}(0,1)$. Finally, for two functions $h_1,h_2: \mathbb{N} \rightarrow [0,\infty)$, the notation $h_1(N) \asymp h_2(N)$, as $N \to \infty$ means $\limsup_{N \to \infty} h_1(N)/h_2(N) \leq a_1 \in (0, +\infty)$ and $\limsup_{N \to \infty} h_2(N)/h_1(N) \leq a_2 \in (0, +\infty)$. The first part of our theoretical analysis centers around the projection $\rho^{*}: \X \to [0,1]$ given as
\begin{equation}\label{Bayesprobability}
       \rho^{*}(z):=\frac{ g(z)}{f(z)+ g(z)}.
\end{equation}
As a remark, if we put a prior probability $\pi=1/2$ on observing a label $\ell$ of 1, $\rho^*$ is the posterior probability of observing a draw from $Q$, referred to as the Bayes probability. 

We now formally state the definition of a high-probability lower bound for the total variation distance, using the notation $\lambda:=\TV(P,Q)$ from now on:

\begin{definition}\label{HPLBDef}
For a given $\alpha \in (0,1)$, an estimate $\hat{\lambda}=\hat{\lambda}\left( (Z_1, \ell_1), \ldots, (Z_N, \ell_N)  \right)$ satisfying
\begin{equation}
    \Prob(\hat{\lambda} > \lambda) \leq \alpha
\end{equation}
will be called high-probability lower bound (HPLB) at level $\alpha$. If instead only the condition
\begin{equation}
    \limsup_{N \to \infty } \Prob(\hat{\lambda} > \lambda) \leq \alpha
\end{equation}
holds, we will refer to $\hat{\lambda}$ as asymptotic high-probability lower bound (asymptotic HPLB) at level $\alpha$. 
\end{definition}

\noindent Note that an estimator $\hat{\lambda}$ depends on a function $\rho$. When necessary, this will be emphasized with the notation $\hat{\lambda}^{\rho}$ throughout the text. Whenever $\rho$ is not explicitly mentioned it should be understood that we consider $\rho=\rho^*$.

The above definition is very broad and does not entail any informativeness of the (asymptotic) HPLB. For instance, $\hat{\lambda}=0$ is a valid HPLB, according to Definition \ref{HPLBDef}. Consequently, for $\varepsilon \in (0,1]$, we study whether for a given (asymptotic) HPLB $\hat{\lambda}$, 
\begin{equation}\label{overalreq}
    \mathbb{P}(\hat{\lambda} > (1-\varepsilon) \lambda) \to 1, \tag{$\mathcal{C}_{\varepsilon}$}
\end{equation}
as $N \to \infty$. This entails several cases: if $\varepsilon=1$, then \eqref{overalreq} means the (detection) power goes to 1. If \eqref{overalreq} is true for \emph{all} $\varepsilon \in (0,1]$, it corresponds to consistency of $\hat{\lambda}$. One could also be interested in a non-trivial fixed $\varepsilon$,\ i.e. in detecting a fixed proportion of $\lambda$. 
In order to quantify the strength of a given (asymptotic) HPLB, we examine how fast $\lambda$ may decay to zero with $N$, such that $\hat{\lambda}$ still exceeds a fraction of the the true $\lambda$ with high probability. More precisely, we assume that the signal vanishes at a rate $N^{\gamma}$, for some $-1 < \gamma < 0$, i.e. $1 > \lambda:=\lambda_N \asymp N^{\gamma}$, as $N \to \infty$. If for a given estimator $\hat{\lambda}$, $\varepsilon \in (0,1]$ and $-1 \leq \underline{\gamma}(\varepsilon) < 0$, \eqref{overalreq} is true for all $ \gamma > \underline{\gamma}(\varepsilon)$, we write $\hat{\lambda}$ attains the rate $\underline{\gamma}(\varepsilon)$. To quantify the strength of an estimator $\hat{\lambda}$, we will study the smallest such rate $\underline{\gamma}(\varepsilon)$ it can attain for a given $\varepsilon$, denoted as $\underline{\gamma}^{\hat{\lambda}}(\varepsilon)$. Formally,   

\begin{definition}
For a given (asymptotic) HPLB $\hat{\lambda}$ and for $\varepsilon \in (0,1]$, we define  $\underline{\gamma}^{\hat{\lambda}}(\varepsilon) := \inf\{\gamma_0 \in [-1,0): \text{ for all } \gamma > \gamma_0 \text{ and }  \lambda \asymp N^{\gamma},\  \mathbb{P}(\hat{\lambda} > (1-\varepsilon) \lambda) \to 1 \}$.
\end{definition}

Of course, attaining the true $\lambda$, might be unrealistic in general. In such cases it is also possible to regard $\lambda$ as the total variation distance between the two distributions after the projection through $\rho$, $\TV(\rho_{\#}P,\rho_{\#}Q)$, as described in more detail in Section \ref{powersec}.

In the following, we aim to construct informative (asymptotic) HPLBs for $\TV(P,Q)$. To put the previously introduced rates into perspective, we first introduce an ``optimal'' or \emph{oracle rate}. In Section \ref{Binonmethod} we introduce binary classification asymptotic HPLBs focusing on the fixed cutoff $1/2$. Section \ref{TVsearchec} will then introduce a more data adaptive asymptotic HPLB that indeed considers the supremum over all available cutoffs in the sample.

\subsection{Oracle rate}


\noindent In light of \eqref{Basicinequality} and the notation introduced in the last section, for $(t_N)_{N\geq 1} \subset I$ a (nonrandom) sequence of cutoffs, we define the estimator
\begin{equation}\label{oraclestimator_0}
    \hat{\lambda}^{\rho}(t_N) = \hat{F}_m(t_N) - \hat{G}_n(t_N) - q_{1-\alpha} \sigma(t_N),
\end{equation}
where $\sigma(t)$ is the theoretical standard deviation of $\hat{F}_m(t) - \hat{G}_n(t)$,
\begin{equation}\label{sigmat}
\sigma(t)=\sqrt{ \frac{F(t)(1-F(t))}{m} + \frac{G(t)(1-G(t))}{n} }.
\end{equation}

Using \eqref{Basicinequality}, it can be shown that:

\begin{restatable}{proposition}{oraclelevel}\label{oraclelevel}
Let $\lambda > 0$. For any sequence $(t_N)_{N \geq 1} \subset I$ of cutoffs, $\hat{\lambda}^{\rho}(t_N)$ defined in \eqref{oraclestimator_0} is an asymptotic HPLB of $\lambda$ (at level $\alpha$) for any $\rho: \X \to I$.
\end{restatable}

\noindent The condition $\lambda > 0$ arises from a technicality -- for $\lambda=0$, one can construct a sequence $(t_N)_{N \geq 1}$ such that the level cannot be conserved. Since $\hat{\lambda}^{\rho}(t_N)$ only serves as a theoretical tool, this is not an issue. However the same problem will arise later in Section \ref{estimatedrho}.

Naturally, the performance of  $\hat{\lambda}^{\rho}(t_N)$ will differ depending on the choice of the sequence $(t_N)_{N\geq 1}$ and the choice of $\rho$. Ideally we would like to choose the ``optimal sequence'' $(t_N^*)_{N\geq 1}$ to reach the lowest rate $\underline \gamma$ possible. We might even want to attain the smallest possible rate $\underline{\gamma}$ if
we are able to freely choose $(t_N)_N$ for each given $\gamma > \underline{\gamma}$. This rate is technically the rate obtained by a collection of estimators, whereby for each $\gamma  > \underline{\gamma}$ a potentially different estimator $\hat{\lambda}(t_N(\gamma))$ may be used. More formally, given $\varepsilon \in (0,1]$, let for the following the \emph{oracle rate} $\underline \gamma^{\text{oracle}}(\varepsilon)$ be the smallest rate such that for all $\gamma > \underline\gamma^{\text{oracle}}(\varepsilon)$ there exists a sequence $(t_N)_{N\geq 1} \subset I$ such that \eqref{overalreq} is true for $\hat{\lambda}=\hat{\lambda}^{\rho}(t_N)$. If there exists a sequence $(t_N^*)_{N\geq 1} \subset I$ independent of $\gamma > \underline\gamma^{oracle}(\varepsilon)$, we may define the \emph{oracle estimator}
\begin{equation}\label{oraclestimator}
    \hat{\lambda}^{\rho}_{\text{oracle}} = \hat{F}_m(t_N^*) - \hat{G}_n(t_N^*) - q_{1-\alpha} \sigma(t_N^*).
\end{equation}
In this case $\underline\gamma^{oracle}(\varepsilon)$ is the smallest rate attained by $\hat{\lambda}_{oracle}^{\rho^*}$ for a given $\varepsilon$. Clearly, $\underline\gamma^{\text{oracle}}(\varepsilon)$ depends on $\rho$ as well, whenever $\rho=\rho^*$ in \eqref{Bayesprobability}, the dependence on $\rho$ is omitted. 

Since $(t_N^*)_{N \geq 1}$ corresponds to a specific nonrandom sequence, Proposition \ref{oraclelevel} ensures that $\hat{\lambda}^{\rho}_{oracle}$ is an asymptotic HPLB. Clearly, even if $\hat{\lambda}^{\rho}_{oracle}$ is defined, it will not be available in practice, as $(t_N^*)_{N\geq 1}$ is unknown. However, the oracle rate it attains should serve as a point of comparison for other asymptotic HPLBs. 
We close this section by considering an example:

\begin{example}\label{Example0}
Let $P,Q$ be defined by $P= p_N P_0 + (1-p_N) Q_0$ and $Q= (1-p_N) P_0 + p_N Q_0 $, where $p_N \in [0,1]$ and $P_0$, $Q_0$ have a uniform distribution on $[-1,0]$ and $[0,1]$ respectively. In this example, only $p_N$ is allowed to vary with $N$, while $P_0$, $Q_0$ stay fixed. If we assume $p_N > 0.5$, $\lambda_N=2p_N-1$ and $\lambda \asymp N^{\gamma}$ iff $p_N-1/2 \asymp N^{\gamma}$.
\noindent Thus

\begin{restatable}{proposition}{Examplezeroprop}\label{Example0prop}
For the setting of Example \ref{Example0}, assume $p_N > 0.5$ for all $N$, and $p_N-1/2 \asymp N^{\gamma}$. Then $\underline\gamma^{oracle}(\varepsilon)=-1/2$ for all $\varepsilon \in (0,1]$. This rate is attained by the oracle estimator in \eqref{oraclestimator} with $t_N^*=1/2$ for all $N$. 
\end{restatable}

\end{example}


\subsection{Binary classification bound}\label{Binonmethod}

Let us fix a cutoff $t \in [0,1]$. From the binary classifier $\rho_{t}(Z)=\Ind{\{ \rho(Z)  > t \}}$ we can define the \emph{in-class accuracies} as $A^{\rho}_0(t):=P(\rho_{t}(X)=0)$ and $A^{\rho}_1(t):=Q(\rho_{t}(Y)=1)$. From there, relation \eqref{Basicinequality} can be written in a more intuitive form:
\begin{equation}\label{Basicinequality2}
     TV(P,Q) \geq \sup_{t \in [0,1]}  \left[ A^{\rho}_0(t)  + A^{\rho}_1(t)\right] -1. 
\end{equation}
Thus, the (adjusted) maximal sum of in-class accuracies for a given classifier is still a lower bound on $\lambda=TV(P,Q)$. As it can be shown that the inequality in \eqref{Basicinequality2} is an equality for $\rho=\rho^*$ and $t=1/2$, it seems sensible to build an estimator based on $ A^{\rho}_0(1/2)  + A^{\rho}_1(1/2) -1$.
\noindent Define the in-class accuracy estimators $\hat A^{\rho}_0(1/2)=\frac{1}{m}\sum_{i=1}^{m}  (1-\rho_{1/2}(X_i))$ and $\hat A^{\rho}_1(1/2)=\frac{1}{n}\sum_{j=1}^{n}  \rho_{1/2}(Y_j)$. It follows as in Proposition \ref{oraclelevel}, that:

\begin{proposition}\label{Binom2}
$\hat{\lambda}_{bayes}^{\rho}:=\hat A^{\rho}_0(1/2) + \hat A^{\rho}_1(1/2) -1 - q_{1-\alpha}\hat{\sigma}(1/2)$ with
\begin{equation}\label{rhohat}
 \hat{\sigma}(1/2)=  \sqrt{\frac{\hat A^{\rho}_0(1/2) (1 - \hat A^{\rho}_0(1/2))}{m} + \frac{\hat A^{\rho}_1(1/2) (1 -  \hat A^{\rho}_1(1/2)) }{n}}
\end{equation}
is an asymptotic HPLB of $\lambda$  (at level $\alpha$) for any $\rho: \mathcal{X} \rightarrow [0,1]$. 
\end{proposition}

It should be noted that if $\hat{\sigma}(1/2)$ in $\hat{\lambda}_{bayes}^{\rho}$ is replaced by $\sigma(1/2)$ in \eqref{sigmat}, we obtain $\hat{\lambda}^{\rho}(1/2)$. Consequently, it should be the case that if $t_N^*=1/2$, the rate attained by $\hat{\lambda}_{bayes}^{\rho}$ is the oracle rate. We now demonstrate this in an example:

\begin{example}\label{Example_1}
Compare a given distribution $Q$ with the mixture $P=(1-\delta_N)Q + \delta_N C$, where $C$ serves as a ``contamination'' distribution and $\delta_N \in (0,1)$. Then, $\TV(P,Q)=\delta_N \TV(C, Q).$ If we furthermore assume that $Q$ and $C$ are disjoint, then $\TV(C, Q)=1$ and $\lambda_N=\delta_N$. Then the oracle rate $\underline{\gamma}^{oracle}(\varepsilon)$ and $\underline{\gamma}^{\hat{\lambda}_{bayes}}(\varepsilon)$ coincide:

\begin{restatable}{proposition}{basicpowerresultBinomial}\label{basicpowerresultBinomial}
For the setting of Example \ref{Example_1}, $ \underline{\gamma}^{\hat{\lambda}_{bayes}}(\varepsilon)= \underline{\gamma}^{oracle}(\varepsilon)=-1$, for all $\varepsilon \in (0,1]$.
\end{restatable}

\noindent Indeed, it can be shown that here $\hat{\lambda}^{\rho^*}(1/2)$ gives rise to the oracle estimator in \eqref{oraclestimator}. It may therefore not be surprising that $\hat{\lambda}_{bayes}^{\rho^*}$ attains the rate $\underline \gamma^{oracle}(\varepsilon)$. A small simulation study illustrating Proposition \ref{basicpowerresultBinomial} is given in Figure \ref{asymptoticpowerfig1} in \ref{simsec_1}.
\end{example}


\noindent While $\hat{\lambda}_{bayes}^{\rho}$ is able to achieve the oracle rate in some situations, it may be improved: Taking a cutoff of 1/2, while sensible if no prior knowledge is available, is sometimes suboptimal. This is true, even if $\rho^*$ is used, as we demonstrate with the following example:

\begin{example}\label{Example_2}
Define $P$ and $Q$ by $P=p_1 C_1 + (1-p_1) p_2 P_0 + (1-p_1) (1-p_2) Q_0$ and $Q=p_1 C_2 + (1-p_1) p_2 Q_0 + (1-p_1) (1-p_2) P_0$, where $C_1, C_2, P_0, Q_0$ are probability measures with disjoint support and $p_1, p_2 \in (0,1)$.

\begin{proposition}\label{Example2prop}
For the setting of Example \ref{Example_2}, let $p_2 > 0.5$, $p_2 = 0.5 + o(N^{-1})$. Then the oracle rate is $\underline\gamma^{\text{oracle}}(\varepsilon) = -1$, while $\hat{\lambda}_{bayes}^{\rho^*}$ attains the rate $\underline\gamma^{\hat{\lambda}_{bayes}}(\varepsilon)=-1/2$, for all $\varepsilon \in (0,1]$.
\end{proposition}

\noindent It can be shown that choosing $t_N=0$ for all $N$, leads to the oracle rate of $-1$ in this example. This is entirely missed by $\hat{\lambda}_{bayes}^{\rho^*}$.

\end{example}

Importantly, $\hat{\lambda}_{bayes}^{\rho^*}$ could still attain the oracle rate in Example \ref{Example_2}, if the cutoff of $1/2$ was adapted. In particular, using $\hat{\lambda}_{bayes}^{\rho^*}$ with the decision rule $\rho^*_0(z)=\Ind{\{ \rho^*(z) > 0 \}}$ would identify only the examples drawn from $C_1$ as belonging to class $0$. This in turn, would lead to the desired detection rate. Naturally, this cutoff requires prior knowledge about the problem at hand, which is usually not available. In general, if $\rho$ is any measurable function, potentially obtained by training a classifier or regression function on independent data, a cutoff of 1/2 might be strongly suboptimal. We thus turn our attention to an HPLB of the supremum in \eqref{Basicinequality} directly.

\begin{figure}
\begin{subfigure}[b]{0.5\textwidth}
  \centering
\begin{tikzpicture}[scale=0.8,
    declare function={unipdf(\x,\xl,\xu)= (\x>\xl)*(\x<\xu)*1/(\xu-\xl);}
]
\begin{axis}[
    samples=100,
    const plot mark mid,
    ymin=0,ymax=0.6
]
\pgfmathsetmacro{\ptwo}{0.55}
\addplot [ thick, orange] {\ptwo*unipdf(x,-1,0) + (1-\ptwo)*unipdf(x,0,1)};
\addplot [dashed, thick, cyan] {(1-\ptwo)*unipdf(x,-1,0) + \ptwo*unipdf(x,0,1)};
\end{axis}
\end{tikzpicture}

\end{subfigure}
\begin{subfigure}[b]{0.5\textwidth}
  \centering
\begin{tikzpicture}[ scale=0.8,
    declare function={unipdf(\x,\xl,\xu)= (\x>\xl)*(\x<\xu)*1/(\xu-\xl);}
]
\begin{axis}[
    samples=100,
    const plot mark mid,
    ymin=0,ymax=0.6
]
\pgfmathsetmacro{\pone}{0.1}
\pgfmathsetmacro{\ptwo}{0.55}
\addplot [ thick, orange] { \pone*unipdf(x,-3,-2)  + \ptwo*(1-\pone)*unipdf(x,-1,0) + (1-\ptwo)*(1-\pone)*unipdf(x,0,1)};
\addplot [dashed, thick, cyan] {\pone*unipdf(x,2,3) + (1-\ptwo)*(1-\pone)*unipdf(x,-1,0) + \ptwo*(1-\pone)*unipdf(x,0,1)};
\end{axis}
\end{tikzpicture}

\end{subfigure}
\caption{ Illustration of Examples \ref{Example0} and \ref{Example_2}. Left: Illustration of Example \ref{Example0} using $p_N=0.55$. Right: Illustration of Example \ref{Example_2} using $p_1=0.1$, $p_2=0.55$ and uniform distributions for $C_1$, $C_2$, $P_0$ and $Q_0$.}
\label{fig:illustration0}
\end{figure}

\subsection{Adaptive binary classification bound} \label{TVsearchec}

In light of relation \eqref{Basicinequality}, we aim to directly account for the randomness of $\sup_t(\hat{F}_m(t)- \hat{G}_n(t))=\sup_{z} (\hat{F}_m(\rho_{(z)})- \hat{G}_n(\rho_{(z)}))$. We follow \cite{finner2018} and define the counting function $V_{m,z}=m\hat{F}_{m}(\rho_{(z)})$ for each $z \in  J_{m,n}:=\{1,\ldots, m+n-1 \}$. Using $m\hat{F}_{m}(\rho_{(z)}) + n\hat{G}_{n}(\rho_{(z)})=z$, it is possible to write:
\begin{equation}\label{Vfunction}
\hat{F}_{m}(\rho_{(z)}) - \hat{G}_{n}(\rho_{(z)})= \frac{m+n}{mn} \left( V_{m,z} - \frac{m z}{m+n} \right).
\end{equation}
A well-know fact (see e.g., \cite{finner2018}) is that under $H_0: F=G$, $V_{m,z}$ is a hypergeometric random variable, obtained by drawing without replacement $z$ times from an urn that contains $m$ circles and $n$ squares and counting the number of circles drawn. We denote this as $V_{m,z} \sim \mbox{Hypergeometric}(z, m+n, m)$ and simply refer to the resulting process $z \mapsto V_{m,z}$ as the \emph{hypergeometric process}. Though the distribution of $V_{m,z}$ under a general alternative is not known, we will now demonstrate that one can nonetheless control its behavior, at least asymptotically. We start with the following definition, inspired by \cite{Nicolai2005}:

\begin{definition}[Bounding function]\label{Qfuncdef}
A function $J_{m,n} \times [0,1]  \ni (z, \tilde{\lambda})  \mapsto Q_{m,n, \alpha}(z, \tilde{\lambda})$ is called a \emph{bounding function} at level $\alpha$ if
\begin{align}\label{Qfuncdefeq}
       \mathbb{P}(\sup_{z \in J_{m,n}} \left[  V_{m,z} -  Q_{m,n, \alpha}(z,\lambda)\right] > 0) \leq \alpha.
\end{align}
It will be called an \emph{asymptotic bounding function} at level $\alpha$ if instead
\begin{align}\label{Qfuncdefeqa}
       \limsup_{N \to \infty}  \mathbb{P}(\sup_{z \in J_{m,n}} \left[  V_{m,z} -  Q_{m,n, \alpha}(z,\lambda)\right] > 0) \leq \alpha.
\end{align}
\end{definition}

\noindent In other words, for the true value $\lambda$, $Q_{m,n, \alpha}(z,\lambda)$ provides an (asymptotic) type 1 error control for the process $z \mapsto V_{m,z}$ (often the dependence on $\alpha$ will be ommited). For $\lambda=0$ the theory in \cite{finner2018} shows that such an asymptotic bounding function is given by

$$Q_{m,n, \alpha}(z, \tilde{\lambda})=Q_{m,n, \alpha}(z, 0)=\frac{z m}{m+ n} +  \beta_{\alpha,m} w\left(z,m,n\right), $$ 

with
\begin{equation}\label{Hypervariance}
    w\left(z,m,n\right) = \sqrt{\frac{m}{N} \frac{n}{N} \frac{N-z}{N-1}z }.
\end{equation}

Assuming access to a bounding function, we can define the estimator presented in Proposition \ref{TVsearch}. 


\begin{restatable}{proposition}{witsearch}\label{witsearch}\label{TVsearch}
Let $Q_{m,n,\alpha}$ be an (asymptotic) bounding function and define,
 \begin{align}
\hat{\lambda}^{\rho} &=\inf \left \{\tilde{\lambda} \in [0,1]: \sup_{z \in J_{m,n}} \left[  V_{m,z} -  Q_{m,n, \alpha}(z,\tilde{\lambda}) \right] \leq 0 \right \}. \label{TVsearcheq0}
\end{align}
Then $\hat{\lambda}^{\rho}$ is an (asymptotic) HPLB of $\lambda$ (at level $\alpha$) for any $\rho: \X \to I$.
\end{restatable}

\noindent The proof of Proposition \ref{TVsearch} is given in \ref{proofsection1}.

Since we do not know the true $\lambda$, the main challenge in the following is to find bounding functions that would be valid for any potential $\lambda \geq 0$. We now introduce a particular type of such a bounding function. With $\tilde{\lambda} \in [0,1]$, $\alpha \in (0,1)$, $m(\tilde{\lambda})= m - q_{1-\frac{\alpha}{3}}(\tilde{\lambda},m) $, and $n(\tilde{\lambda}) = n - q_{1-\frac{\alpha}{3}}(\tilde{\lambda},n)$,
we define 
\begin{align}\label{TVsearcheqgeneral}
   &Q_{m,n, \alpha}(z,\tilde{\lambda})= \begin{cases}
    z, \text{ if } 1 \leq z \leq q_{1-\frac{\alpha}{3}}(\tilde{\lambda},m)\\
    m, \text{ if } m+n(\tilde{\lambda})\leq z \leq m+n\\ 
  q_{1-\frac{\alpha}{3}}(\tilde{\lambda},m) + \sq_{\alpha/3}\left(\tilde{V}_{m(\tilde{\lambda}),z}, z \in \{1,\ldots, m(\tilde{\lambda})+ n(\tilde{\lambda}) -1\}  \right),  \text{ otherwise}
    \end{cases} 
\end{align}
where $\tilde{V}_{m,z}$ denotes the counting function of a hypergeometric process and $\sq_{\alpha}\left(\tilde{V}_{m,z}, z \in J \right)$ is a \emph{simultaneous} confidence band, such that
\begin{align}\label{newcondition}
    \limsup_{N \to \infty} \Prob \left( \sup_{z \in J} \left[\tilde{V}_{m,z} - \sq_{\alpha}\left(\tilde{V}_{m,z}, z \in J \right)  \right] > 0 \right) \leq \alpha.
\end{align}

\noindent Note that Equation \eqref{newcondition} includes the case $\Prob \left( \sup_{z \in J} \left[\tilde{V}_{m,z} - \sq_{\alpha}\left(\tilde{V}_{m,z}, z \in J \right)  \right] > 0 \right) \leq \alpha$ for all $N$. Depending on which condition is true, we obtain a bounding function or an asymptotic bounding function:

\begin{restatable}{proposition}{Qfunctions}\label{Qfunctions}
$Q_{m,n, \alpha}$ as defined in  \eqref{TVsearcheqgeneral} is an (asymptotic) bounding function.
\end{restatable}

\noindent The proof of Proposition \ref{Qfunctions} is given in ~\ref{proofsection1}.

A valid analytical expression for $\sq_{\alpha/3}\left(\tilde{V}_{m(\tilde{\lambda}),z}, z \in \{1,\ldots, m(\tilde{\lambda})+ n(\tilde{\lambda}) -1 \}  \right)$ in \eqref{TVsearcheqgeneral} based on the theory in \cite{finner2018} is given in Equation~\eqref{finnerconfidenceband} of ~\ref{boundingexpressionsec}. We will denote the asymptotic bounding function when combining \eqref{TVsearcheqgeneral} with \eqref{finnerconfidenceband} by $Q^{A}$. The asymptotic HPLB that arises from \eqref{TVsearcheq0} with projection $\rho$ and bounding function $Q^A$ will be referred to as $\hat{\lambda}^{\rho}_{adapt}$. Alternatively, we may choose $\sq_{\alpha/3}\left(\tilde{V}_{m(\tilde{\lambda}),z}, z \in \{1,\ldots, m(\tilde{\lambda})+ n(\tilde{\lambda}) -1 \}  \right)$ by simply simulating $S$ times from the process $\tilde{V}_{m(\tilde{\lambda}),z}, z=1,\ldots, m(\tilde{\lambda})+ n(\tilde{\lambda})$. For $S \to \infty$, condition \eqref{newcondition} then clearly holds true. This is especially important, for smaller sample sizes, where the (asymptotic) $Q^{A}$ could be a potentially bad approximation.

We close this section by considering once again the introductory example in Section~\ref{toyexamplesec}. Our two proposed estimators applied to this example give $\hat{\lambda}^{\rho}_{bayes}=0$ and $\hat{\lambda}^{\rho}_{adapt}=0.0022$. Thus, as one would expect from the permutation test results, $\hat{\lambda}^{\rho}_{adapt}$ is able to detect a difference, whereas $\hat{\lambda}^{\rho}_{bayes}$ is not. While it is difficult in this case to determine the true $\lambda$, we can show for another example, that $\hat{\lambda}^{\rho^*}_{adapt}$ attains the rate $\hat{\lambda}^{\rho^*}_{bayes}$ could not:

\begin{proposition}\label{Example3prop}
Let $P,Q$ be defined as in Example \ref{Example_2} with $p_2 > 0.5$, $p_2 = 0.5 + o(N^{-1})$. Then $\underline\gamma^{\hat{\lambda}_{adapt}}(\varepsilon)=\underline\gamma^{oracle}(\varepsilon)=-1$, independent of $\varepsilon$.
\end{proposition}

A small simulation study illustrating Propositions \ref{Example2prop} and \ref{Example3prop} is given in Figure \ref{asymptoticpowerfig1} in \ref{simsec_1}. In the next section, we generalize the results in Examples \ref{Example_1} and \ref{Example_2} and show that $\hat{\lambda}_{adapt}^{\rho}$ \emph{always} attains the oracle rate.

\section{Theoretical guarantees} \label{powersec}


This section studies some of the theoretical properties of our proposed lower-bounds. We start in Section \ref{rhostarsec} by assuming access to the ``ideal'' classifier $\rho^*$ and show that in this case, the $\hat{\lambda}^{\rho^*}_{adapt}$ can asymptotically detect a nonzero TV with a better rate than $\hat{\lambda}^{\rho^*}_{bayes}$. More generally, our main results in Proposition \ref{oracleprop} and \ref{newamazingresult} show that $\hat{\lambda}^{\rho^*}_{adapt}$ achieves the same asymptotic performance as $\hat{\lambda}^{\rho^*}$, which is free to ``choose'' a sequence of cutoffs $(t_N)_N$. Though we use $\rho^*$ for simplicity, all of the results in this section also hold true for any arbitrary (fixed) $\rho: \mathcal{X} \to [0,1]$, and also if we replace $\lambda$ by the TV distance on the projection,
\begin{align}\label{rhoonprojectedspace}
    \lambda(\rho) =\sup_{t \in [0,1]} \left[ A_0^{\rho}(t) + A_1^{\rho}(t) \right] -1 := \sup_{t \in [0,1]} \left[ P(\rho(X) \leq t ) - Q(\rho(Y) \leq t )\right]
\end{align}
such that $\lambda:=\lambda(\rho^*)=TV(P,Q)$.

Section \ref{estimatedrho} then extends the main result of Section \ref{rhostarsec} from $\rho^*$ to a sequence $\hat{\rho}=\hat{\rho}_N$, estimated on independent training data, showing that $\hat{\lambda}^{\hat{\rho}}_{adapt}$ and $\hat{\lambda}^{\hat{\rho}}$ have the same asymptotic detection power. Finally, we discuss sufficient conditions for the consistency of $\hat{\lambda}^{\hat{\rho}}_{adapt}$. We restrict to $I=[0,1]$ throughout this section.


\subsection{Using $\rho^*$} \label{rhostarsec}

We start by studying the asymptotic properties of the proposed asymptotic HPLB estimators, assuming access to $\rho^*$ in \eqref{Bayesprobability}. Recall that for a fixed $\varepsilon \in (0,1]$, $\underline \gamma^{oracle}(\varepsilon)$ was defined as the minimal rate such that for all $\gamma > \underline\gamma^{oracle}(\varepsilon)$ there exists a sequence $(t_N)_{N\geq1} \subset I$ such that \eqref{overalreq}, i.e.
\begin{equation*}
    \mathbb{P}(\hat{\lambda} > (1-\varepsilon) \lambda) \to 1,
\end{equation*}
is true for $\hat{\lambda}=\hat{\lambda}^{\rho^*}(t_N)$. Consider for $\varepsilon \in (0,1]$ the following conditions on $(t_N)_{N \geq 1}$:
\begin{align}
     \liminf_{N \to \infty} \frac{\lambda(t_N)}{\lambda_N}  & \geq  1-\varepsilon \label{CondII},
     \end{align}
and
\begin{subequations}\label{CondI}
\begin{align}
  \lim_{N } \frac{\lambda_N}{\sigma(t_N)} &= \infty, \text{ if }  \liminf_{N \to \infty} \frac{\lambda(t_N)}{\lambda_N} >  1-\varepsilon  \label{CondI.1}, \\ 
         \lim_{N } \frac{\lambda_N}{\sigma(t_N)} \left( \frac{\lambda(t_N)}{\lambda_N} -( 1-\varepsilon) \right) &= \infty, \text{ if }  \liminf_{N \to \infty} \frac{\lambda(t_N)}{\lambda_N} =  1-\varepsilon   \label{CondI.2},
       \end{align}
\end{subequations}
where $\lambda(t_N):=F(t_N) - G(t_N)$ and $\sigma(t)$ is defined as in \eqref{sigmat}. We then refer to Condition \eqref{CondI}, iff \eqref{CondI.1} and \eqref{CondI.2} are true:
\begin{align}
    \eqref{CondI.1} \text{ and } \eqref{CondI.2} \text{ hold}.\tag{21}
\end{align}
We now redefine $\underline\gamma^{oracle}(\varepsilon)$ as the smallest element of $[0,1]$ with the property that for all $\gamma > \underline\gamma^{oracle}(\varepsilon)$ there exists a sequence $(t_N)_{N\geq1} \subset I $ such that \eqref{CondII} and \eqref{CondI} are true. Intuitively, this means that a given rate is achieved for $(t_N)_{N\geq1}$ if either $F(t_N) - G(t_N)$ is strictly larger than $(1-\varepsilon)\lambda_N$ and the variance decreases fast relative to $\lambda_N$ (Condition \eqref{CondII} and \eqref{CondI.1}), or $F(t_N) - G(t_N)$ is exactly equal to $(1-\varepsilon)\lambda_N$ in the limit, which needs to be balanced by an even faster decrease in the variance $\sigma(t_N)$ (Condition \eqref{CondII} and \eqref{CondI.2}). As a side remark, \eqref{CondI.2} is problematic for $\varepsilon=1$, if $\sigma(t_N)=0$ for infinitely many $N$. In this case, it should be understood that \eqref{CondI.2} is taken to be false.

The following proposition confirms that the two definitions of $\underline\gamma^{oracle}(\varepsilon)$ coincide:

\begin{restatable}{proposition}{oracleprop}\label{oracleprop}
Let $-1 < \gamma < 0 $ and $\varepsilon \in (0,1]$ fixed. Then there exists a $(t_N)_{N\geq1}$ such that \eqref{overalreq} is true for $\hat{\lambda}^{\rho^*}(t_N)$ iff there exists a $(t_N)_{N\geq1}$ such that \eqref{CondII} and \eqref{CondI} are true. 
\end{restatable}

If we consider a classifier with cutoff $t \in I$, $\rho_t(z)=\Ind\{\rho(z) > t \}$ and, as in Section \ref{Binonmethod}, define in-class accuracies $A_0(t):=A_0^{\rho^*}(t)$, $A_1(t):=A_1^{\rho^*}(t)$, we may rewrite $\sigma(t)$ in a convenient form
\begin{equation}\label{sigmat2}
\sigma(t)=\sqrt{\frac{A_0(t)(1-A_0(t))}{m} + \frac{A_1(t)(1-A_1(t))}{n}  }.
\end{equation}
Since $\sqrt{N} \lambda_N$ does not go to infinity for $\gamma \leq -1/2$, the divergence of the ratio in \eqref{CondI} is only achieved, if both $A_0(t_N)(1-A_0(t_N))$ and $A_1(t_N)(1-A_1(t_N))$ go to zero sufficiently fast. In our context, this is often more convenient to verify directly.

The binary classification estimator $\hat{\lambda}^{\rho^*}_{bayes}$ takes $t_N=1/2$ and, since $ F(1/2)-G_t(1/2)=\lambda_N$, \eqref{CondII} is true for any $\varepsilon$. Thus a given rate $\underline{\gamma}$ is achieved iff \eqref{CondI} is true for $t_N=1/2$. This is stated formally in the following corollary:

 \begin{restatable}{corollary}{nofastratebinomial}\label{nofastratebinomial}
  $\hat{\lambda}^{\rho^*}_{bayes}$ attains the rate $\underline\gamma^{\hat{\lambda}_{bayes}}(\varepsilon)=\underline{\gamma}$ for all $\varepsilon \in (0,1]$, iff \eqref{CondI} is true for $t_N=1/2$ and all $\gamma > \underline{\gamma}$.
 \end{restatable}


The proof is a direct consequence of Proposition \ref{oracleprop} and is given in \ref{proofStoDom}. We thus write $\underline \gamma^{\hat{\lambda}_{bayes}}$ instead of $\underline \gamma^{\hat{\lambda}_{bayes}}(\varepsilon)$. It should be noted \eqref{CondI} is always true for $\gamma > -1/2$. As such, $\underline \gamma^{\hat{\lambda}_{bayes}}\geq -1/2$ and only the case of $\underline{\gamma} < -1/2$ is interesting in Corollary \ref{nofastratebinomial}.

Finally, the adaptive binary classification estimator $\hat{\lambda}^{\rho^*}_{adapt}$ always reaches at least the rate $\underline{\gamma} = -1/2$. In fact, it turns out that it attains the oracle rate:

\begin{restatable}{proposition}{newamazingresult}\label{newamazingresult}
Let $-1 < \gamma < 0$ and $\varepsilon \in (0,1]$ fixed. Then \eqref{overalreq} is true for $\hat{\lambda}^{\rho^*}_{adapt}$ iff there exists a $(t_N)_{N\geq1}$ such  that \eqref{CondII} and \eqref{CondI} are true.
\end{restatable}

\noindent This immediately implies:
\begin{corollary}
For all $\varepsilon \in (0,1]$, $\underline\gamma^{oracle}(\varepsilon)=\underline\gamma^{\hat{\lambda}_{adapt}}(\varepsilon)$.
\end{corollary}

The next section shows that this result can be generalized to $\hat{\rho}$ estimated from independent data.




\subsection{Estimated $\rho$}\label{estimatedrho}

In this section we assume that $\hat{\rho}$ is a ``probability function'' in $[0,1]$, estimated from data. In that case \emph{sample-splitting} should be used, i.e. the function $\hat{\rho}$ is estimated independently on a \emph{training set} using a learning algorithm which is then used to compute an (asymptotic) HPLB based on an independent \emph{test set}. Sample-splitting is important to avoid spurious correlation between $\rho$ and the (asymptotic) HPLB, not supported by our theory. Formally we assume,
\begin{itemize}
    \item[(E1)]  
     $\hat{\rho}=\hat{\rho}_{N_{tr}}$ is trained on a sample of size $N_{tr}$, $(Z_1, \ell_1), \ldots, (Z_{N_{tr}}, \ell_{N_{tr}})$, and evaluated on an independent sample $(Z_1, \ell_1), \ldots, (Z_{N_{te}}, \ell_{N_{te}})$, with $N_{tr} + N_{te}=N$,
    \item[(E2)] $N_{te}, N_{tr} \to \infty$, as $N \to \infty$, with $m_{te}/N_{te} \to \pi \in (0,1)$,
\end{itemize}
where as before $m_{j}$ denotes the number of draws from $P$ (with label $0$) and $n_{j}$ the number of draws from $Q$ (with label $1$), for $j \in \{te, tr \}$.

In practice, most probability estimates try to approximate the Bayes probability (see e.g., \citet{ptpr}):
\begin{align}\label{Bayesprobabilitytrue}
    \rho^B(z) = \frac{(1-\pi) g(z)}{\pi f(z) + (1-\pi) g(z)},
\end{align}
with \emph{Bayes classifier} $\rho^B_{1/2}(z)=\Ind\{ \rho^B(z) > 1/2 \}$. It is the classifier resulting in the maximal overall accuracy, denoted the \emph{Bayes accuracy}: $\pi A_{0}^{\rho^B}(1/2) + (1- \pi) A_{1}^{\rho^B}(1/2)$. Clearly, $\rho^B=\rho^*$ for $\pi=1/2$. More generally, it can be shown that $\rho_{1-\pi}^B(z)=\rho^*_{1/2}(z)$.


Let as before, $\hat{\lambda}^{\hat{\rho}}_{adapt}$  be the estimator obtained when using $\hat{\rho}$ and $\lambda(\hat{\rho})$ be defined as in \eqref{rhoonprojectedspace} for $\rho=\hat{\rho}$. Similarly, for a sequence $(t_N)_{N \geq 1} \subset I$, we define $\hat{\lambda}^{\hat{\rho}}(t_N)$ to be the theoretical estimator \eqref{oraclestimator_0} using $\hat{\rho}$. Conditioning on the training data through $\hat{\rho}$, allows for a generalization of the theory in Section \ref{rhostarsec} to estimated $\rho$. 

The first step, is to extend the theory in Section \ref{rhostarsec} to the case of arbitrary (nonrandom) sequences $(\rho_N)_{N \geq 1}$. While the proofs of the results in Section \ref{rhostarsec} are applicable almost one-to-one in this case, there is one issue arising from the estimator $\hat{\lambda}^{\rho}_{bayes}$. We exemplify this in the following:

\begin{example}\label{levelexample}
Assume both $X_1, \ldots, X_n$, $Y_1, \ldots, Y_n$ uniform on $[0,1]$ and $\rho_N$ such that for some $C > 0$,
\begin{align*}
    \rho_{N,1/2}(Z)= \begin{cases} 0, & \text{if } Z \in [0,C/n]\\
    1, & \text{else} \end{cases}
\end{align*}
Then:
\begin{restatable}{proposition}{Binomialmightnotconservelevel}\label{Binomialmightnotconservelevel}
For the setting of Example \ref{levelexample}, let $\xi_1$, $\xi_2$ be independently Poisson distributed, with mean $C$. Then
\[
\Prob(\hat{\lambda}^{\rho_N}(1/2) > 0 ) \to \Prob(\xi_1-\xi_2 >   q_{1-\alpha} \sqrt{2C} ).
\]
\end{restatable}
It can be shown numerically that $\Prob(\xi_1-\xi_2 >   q_{1-\alpha} \sqrt{2C} )> 1-\alpha$, for some $C$. Thus, $\hat{\lambda}^{\rho_N}(1/2)$ is not a valid asymptotic HPLB.
\end{example}

The case above appears rather exotic, and might not be realistic. What is more, we used $\hat{\lambda}^{\rho_N}(1/2)$ in the above example, with the true variance included, instead of $\hat{\lambda}^{\rho_N}_{bayes}$. In this case, the accuracies $A_0^{\rho_N}(1/2), A_1^{\rho_N}(1/2)$ cannot even be estimated reliably, so it is not clear what exactly will happen if $\sigma(1/2)$ is estimated. However none of these problems are of concern for $\hat{\lambda}_{adapt}^{\hat{\rho}}$, which conserves the level in any case:

\begin{restatable}{proposition}{HPLBwithestimatedrho}\label{HPLBwithestimatedrho}
Assume (E1) and (E2). Then $\hat{\lambda}^{\hat{\rho}}_{adapt}$ is an (asymptotic) HPLB of $\lambda$ (at level $\alpha$).
\end{restatable}


Thus, we will focus in this section only on the adaptive estimator. We first generalize Propositions \ref{oracleprop} and \ref{newamazingresult} to this case.


\begin{proposition}\label{estimatedrhoresult}
Let $-1 < \gamma \leq 0$ and $\varepsilon_1 \in (0,1]$ fixed. Assume that $\lambda_N=N_{te}^{\gamma}$ and that (E1) and (E2) hold. Then the following is equivalent
\begin{itemize}
    \item[(i)] there exists a sequence $(t_{N_{te}})_{N_{te} \geq 1}$ such that \eqref{overalreq} is true for $\hat{\lambda}^{\hat{\rho}}(t_{N_{te}})$,
    \item[(ii)] \eqref{overalreq} is true for $\hat{\lambda}^{\hat{\rho}}_{adapt}$.
\end{itemize}
\end{proposition}

The main message of Proposition \ref{estimatedrhoresult} is that for $\hat{\rho}$ estimated on independent training data, $\hat{\lambda}^{\hat{\rho}}_{adapt}$ still has the same asymptotic performance as an estimator that is free to choose its cutoff for any given sample size. And this holds despite the fact that even for $\lambda=0$, $\hat{\lambda}^{\hat{\rho}}_{adapt}$ is a valid HPLB, which is not clear for $\hat{\lambda}^{\hat{\rho}}(t_N)$, as seen in Example \ref{levelexample}.


In practice, one might be more interested under what conditions $\hat{\lambda}^{\hat{\rho}}_{adapt}$ is consistent for a fixed $\lambda$. To answer this question, we first restate consistency for a sequence of classifiers, assuming $\lambda$ does not change:

\begin{definition}\label{consistencydefinition}
A sequence of classifiers $\hat{\rho}_{N,t_N}=\hat{\rho}_{t_{N}}$, $N \in \N$, is \emph{consistent}, if 
\[
A_0^{\hat{\rho}_N}(t_N) +   A_1^{\hat{\rho}_N}(t_N) \stackrel{p}{\to}   A_0^{\rho^*}(1/2) +    A_1^{\rho^*}(1/2).
\]
\end{definition}

This is the standard definition of consistency, see e.g., \citet[Definition 6.1]{ptpr}, with two small modifications: We consider accuracies instead of classification errors and instead of the Bayes accuracy in the limit, we consider the equally weighted accuracy of $\rho^*$. As such the definition is a special case of the $\Psi$-consistency of \citet{NIPS2014_32b30a25}.

A simple consequence of Proposition \ref{estimatedrhoresult} is that a classifier that is consistent for the equally weighted sum of in-class errors, also leads to a consistent estimate of $\lambda$.

\begin{restatable}{corollary}{estimatedrhoresulttwo}\label{estimatedrhoresult2}
Assume that $\lambda$ is fixed and that there exists a sequence $(t_{N_{tr}})$, such that the sequence of classifiers $\hat{\rho}_{N_{tr},t_{N_{tr}}}$ is consistent. Then \eqref{overalreq} is true for $\hat{\lambda}^{\hat{\rho}}_{adapt}$, for all $\varepsilon > 0$.
\end{restatable}

In essence, for a given sequence of estimated $\hat{\rho}$, it is enough that there exists a sequence of cutoffs leading to a consistent classifier, for $\hat{\lambda}_{adapt}$ to be consistent. As is well-known (see e.g., \citet{ptpr} and \citet{NIPS2014_32b30a25}),

\begin{lemma}
Assume that $m_{tr}/N_{tr} \to \pi \in (0,1)$ and that
\begin{equation}\label{rhocondition}
    \E[ | \hat{\rho}_{N_{tr}} - \rho^B | \mid  \hat{\rho}_{N_{tr}} ] \stackrel{p}{\to} 0.
\end{equation}
Then $\hat{\rho}_{N_{tr},t_{N_{tr}}}$ is consistent for $t_{N_{tr}} = m_{tr}/N_{tr}$.
\end{lemma}


This is a relatively straightforward sufficient condition for the consistency of $\hat{\lambda}^{\hat{\rho}}_{adapt}$. We would like to note however that, as shown in \citet{ptpr} for the Bayes classifier, \eqref{rhocondition} usually is much stronger than consistency of the classifier in Definition \ref{consistencydefinition}.

We close this section with two examples:

\begin{example}
Assume access to the Bayes classifier $\rho^B$ and $\pi \neq 1/2$. In this case, \citet{hediger2020use} showed that no test based on $\rho^B_{1/2}$ has power higher than its level. In our case, this translate to an inconsistent estimate of $\lambda_{bayes}^{\rho^B}$. On the other hand, it is well-known that
\[
A_0^{\rho^B}(1-\pi) +   A_1^{\rho^B}(1-\pi) =  A_0^{\rho^*}(1/2) +    A_1^{\rho^*}(1/2),
\]
so the cutoff of $1-\pi$ fixes the issue and indeed leads to both a consistent estimate and a consistent test. 
\end{example}

\begin{example}
Combining the arguments in \citet[Theorem 3.1]{biauconsistency} and \citet{ptpr}, if $X$, $Y$ are supported on $ [0,1]^d$ and $\hat{\rho}$ is a Random Forest using random splitting, then \eqref{rhocondition} holds. For an adapted version of this Random Forest, the result can also be extended to distributions of $X$, $Y$ on $\R^d$, see \citet[Theorem 3.2]{biauconsistency}.
\end{example}

Of course, as before, even if $\hat{\rho}$ is not consistent, we might still be able to detect a signal, giving us an indication of the strength of difference between two distributions. That is, the result of Proposition \ref{estimatedrhoresult} and Corollary \ref{estimatedrhoresult2} hold more generally, as long as $\lambda(\hat{\rho})$ converges in probability to some $\lambda(\rho) \leq \lambda$, which may again be seen as the total variation distance on the projected space. In the next section, we move on from the question of consistency and study how one might find a $\rho$ in practice in a more general framework.

\section{Practical considerations}\label{appsec}\label{rhochoice}

In this section we put the methodology introduced in Section \ref{sec:meth} in practical perspectives. We first generalize our setting to allow for more flexible projections: Let $\mathcal{P}=\{ P_{t}\mbox{ },\mbox{ } t \in I\}$ be a family of probability measures defined on a measurable space $\left(\mathcal{X}, \A\right)$ indexed by a totally ordered set $\left(I, \preceq \right)$. We further assume to have a probability measure $\mu$ on $I$ and independent observations $\mathcal{T} = \left\{(z_i, t_i)\right\}_{i=1}^{N}$ such that $z_{i} \sim P_{t_i}$ conditionally on $t_i \in I$ drawn from $\mu$, for $1 \leq i \leq N$. Given $s \in I$, and a function $\rho: \mathcal{X} \rightarrow I$, we define two empirical distributions denoted $\hat{F}_{m,s}$ and $\hat{G}_{n,s}$ obtained from ``cutting" the set of observations $\mathcal{T}$ at $s$. Namely if we assume that out of the $N = m + n$ observations, $m$ of them have their index $t_i$ smaller or equal to $s$ and $n$ strictly above, we have for $y \in I$
$$\hat{F}_{m,s}(y) = \frac{1}{m}\sum_{i=1}^{N}\mathbbm{1}\{\rho(z_i) \leq y, t_i \leq s\}, \hspace{0.25cm}\text{and} \hspace{0.25cm} \hat{G}_{n,s}(y) = \frac{1}{n}\sum_{i=1}^{N}\mathbbm{1}\{\rho(z_i) \leq y, t_i > s\}.$$ 

These empirical distributions correspond to the population mixtures $F_s \propto \int_{t \leq s} \rho_{\#}P_t\ d\mu(t)$  and $G_s \propto \int_{t > s} \rho_{\#}P_t\ d\mu(t)$. We will similarly denote the measures associated to $F_s$ and $G_s$ as $P_s$ and $Q_s$ respectively and will use the two notations interchangeably.
Note that we assumed $m$, $n$ deterministic so far, which changes in the above framework, where $m \sim \mbox{Bin}(\pi, N)$, with $\pi:=\Prob(T \leq s)$. Still, with a conditioning argument, one can show that, whenever the level is guaranteed for nonrandom $m,n$, it will also be once $m,n$ are random.



The question remains how to find a good $\rho$ in practice. As our problems are framed as a split in the ordered elements of $I$, it always holds that one sample is associated with higher $t \in I$ than the other. Consequently, we have power as soon as we find a $\rho: \mathcal{X} \to I$ that mirrors the relationship between $T$ and $Z_{T}$. It therefore makes sense to frame the problem of finding $\rho$ as a loss minimization, where we try to minimize the loss of predicting $T \in I$ from $Z \in \mathcal{X}$: For a given split point $s$, consider $\rho_s$ that solves 
\begin{align} \label{overallmin}
    \rho_s := \argmin_{h \in \F}  \E[\mathcal{L}_s(h(Z), T)],
\end{align}
where $\F$ is a collection of functions $h:\mathcal{X} \to I$ and $\mathcal{L}_s: I \times I \to \R_{+}$ is some loss function. As before, we assume to have densities $f_{s}$, $g_s$, for $P_s$, $Q_s$ respectively. For simplicity, we also assume that time is uniform on $I=[0,1]$. As it is well-known, taking $\F$ to be all measurable functions $h:\mathcal{X} \to I$ and $\mathcal{L}(f(z), t)= \left( f(z) - \Ind_{ (s,1]}(t) \right)^2$, we obtain the supremum as 
\begin{equation}\label{Bayesprobability2}
       \rho_{1,s}(z):=\E[\Ind_{ (s,1]}(T)| z]=\frac{(1-s) g_{s}(z)}{s f_{s}(z)+ (1-s)g_{s}(z)},
\end{equation}
which is simply the Bayes probability in \eqref{Bayesprobability}. Taking instead $\mathcal{L}(f(z), t)= \left( f(z) - t \right)^2$, yields $\E[T| Z]$. Some simple algebra shows that if there is only \emph{one} point of change $s^*$, i.e. $T$ is independent of $Z$ conditional on the event $T \leq s^*$ or $T > s^*$, $\E[T \mid z]$ can be expressed as:
\begin{equation}\label{Bayesregression}
        \rho_{2,s^*}(z)= \frac{1}{2} \left( s^* + \rho_{1,s^*}(z) \right),
\end{equation}
which is a shifted version of $\rho_{1,s^*}(z)$. Contrary to $\rho_{1,s}$, the regression version $\rho_{2,s^*}$ does not depend on the actual split point $s$ we are considering.


In Section \ref{empiricssection} we try to approximate \eqref{Bayesprobability2} and \eqref{Bayesregression} by using the Random Forest of \cite{Breiman2001}. That is, the function $\rho$ is fitted on a training set using a learning algorithm which is then used to compute an (asymptotic) HPLB based on an independent test set, as in Section \ref{estimatedrho}.

\section{Numerical examples} \label{empiricssection}


%
%

Distributional change detection in climate is a topic of active research (see e.g. \cite{sebaNico2020} and the references therein). We will demonstrate the estimator $\hat{\lambda}_{adapt}$ in three applications using the \texttt{NCEP$\_$Reanalysis 2} data provided by the NOAA/OAR/ESRL PSD, Boulder, Colorado, USA, from their website at \url{https://www.esrl.noaa.gov/psd/}. The analyses were run using the R-package \texttt{HPLB} (see \url{https://github.com/lorismichel/HPLB}). We mention that the estimator $\hat{\lambda}_{bayes}$ gives comparable results and is ommited here. This dataset is a worldwide reanalysis containing daily observations of the $4$ variables:
\begin{itemize}
    \item[-] \emph{air temperature (air)}: daily average of temperature at 2 meters above ground, measured in degree Kelvin;
    \item[-] \emph{pressure (press)}: daily average of pressure above sea level, measured in Pascal;
    \item[-] \emph{precipitation (prec)}: daily average of precipitation at surface, measured in kg per $m^2 $ per second;
    \item[-] \emph{humidity (hum)}: daily average of specific humidity, measured in proportion by kg of air;
\end{itemize}
over a time span from $1^{st}$ of January $1979$ to $31^{th}$ January $2019$. Each variable is ranging not only over time, but also over $2'592$ locations worldwide, indexed by longitude and latitude coordinates, as (longitude, latitude). All variables are first-differenced to reduce dependency and seasonal effects before running the analyses. Figure \ref{fig:geo} displayed the $4$ time series corresponding to the geo-coordinates (-45,-8) (Brazil).

\begin{figure}
\begin{subfigure}[b]{0.5\textwidth}
  \centering
  \includegraphics[width=8cm]{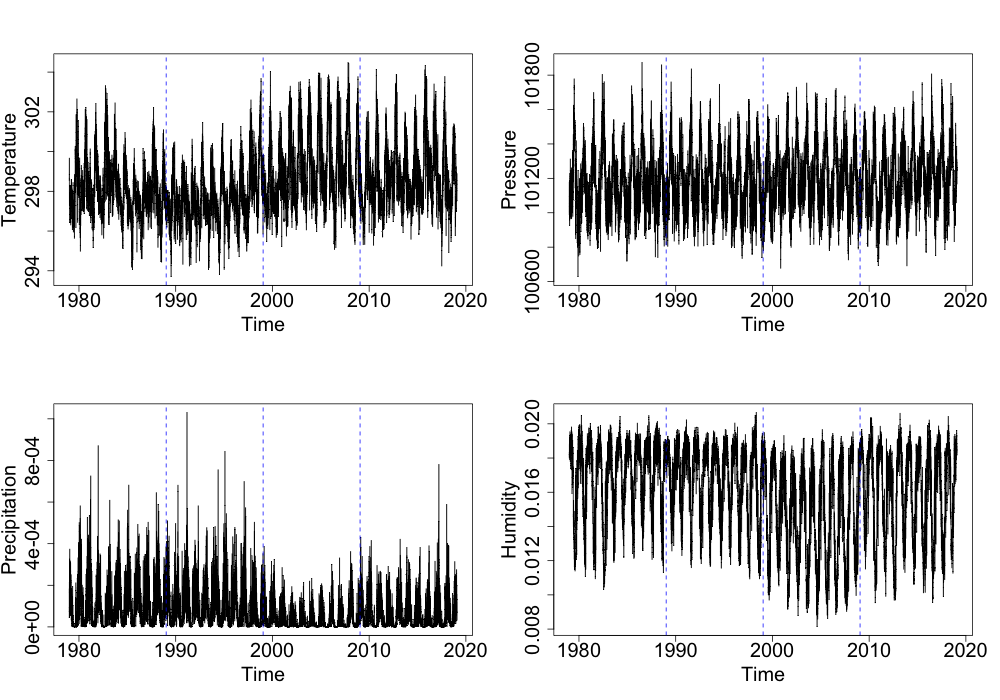}
  \label{fig:1}
\end{subfigure}
\begin{subfigure}[b]{0.5\textwidth}
  \centering
  \includegraphics[width=8cm]{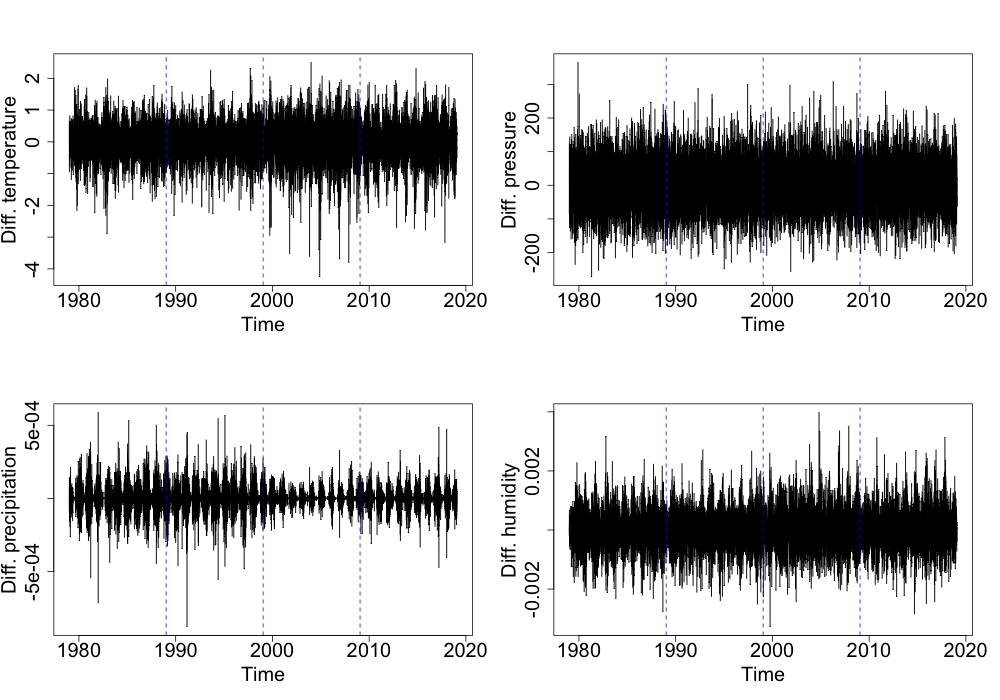}
\end{subfigure}
\caption{ Temperature, pressure, precipitation and humidity at geo-coordinates (-45,-8) (Brazil) over the time period ranging from 1979 to 2019 (on the left).  Corresponding differenced series (on the right). The vertical dashed blue lines are the breaks used in analysis (B).}
\label{fig:geo}
\end{figure}

The potential changes in distribution present in this dataset could require a refined analysis and simple investigation for mean and/or variance shift might not be enough. Moreover, detecting ``small'' changes, as $\hat{\lambda}_{adapt}$ is designed to do, could be of interest. In addition, thanks to the equivalent characterization of TV explained in Section \ref{sec:intro}, $\hat{\lambda}_{adapt}$ represents the minimal percentage of days on which the distribution of the considered variables has changed. We present $2$ types of analyses to illustrate the use of the (asymptotic) HPLBs introduced in this paper:


\begin{itemize}
    \item[(A)] \textbf{\emph{temporal climatic change-map}}: a study of the change of climatic signals between two periods of time ($1^{st}$ of January $1979$ to $15^{th}$ of January $1999$ against $16^{th}$ of January $1999$ to $31^{th}$ January $2019$) across all $2'592$ locations.
    \item[(B)] \textbf{\emph{fixed-location change detection}}: a study of the change of climatic signals over several time points for a fixed location.
\end{itemize}




For analysis (A), compare the first half (years 1979-1999) of the data with the second half (years 1999-2019) over all available locations. That is $P_s$ corresponds to the distribution of the first half of the (differenced) data, while $Q_s$ corresponds to the second. The projection $\rho$ is chosen to be a Random Forest classification. To this end, sample-splitting is applied and the available time-span is equally divided into $4$ consecutive time blocks, of which the middle two are used as a training set, while the remaining two are used as a test set. The goal is thereby, as with differencing, to reduce the dependence between observations due to the time-structure of the series. Figure \ref{fig:climate_map} shows the results as a world heatmap. Interestingly, there is an area of very-high estimated TV values in the pacific ocean off the cost of South America. The water temperature in this area is indicative of El Ni\~no.

\begin{sidewaysfigure}
    \centering
    \includegraphics[width=20cm]{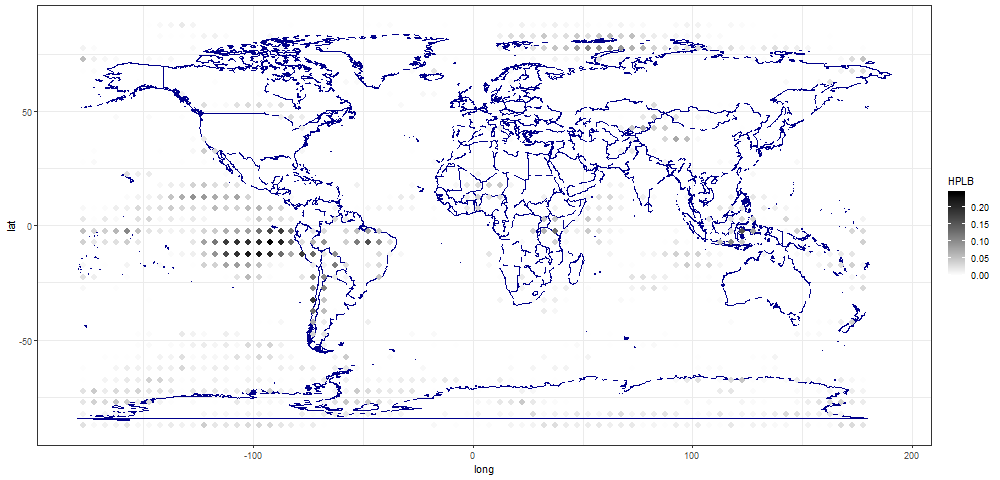}
    \caption{TV estimates by location using $\hat{\lambda}_{adapt}$ when comparing the first half of the available time span with the second half. A Random Forest is fitted using blocks of training and testing on the (differenced) values of the 4 variables. Color changes from white (zero) to black (one).}
    \label{fig:climate_map}
\end{sidewaysfigure}





Analysis (B) illustrates the mixture framework introduced in Section \ref{rhochoice} in a time series context where the ordering is given by time. We analyse the change in distribution for the four climatic variables for $3$ split points chosen uniformly over the time span. The location is thereby fixed to the coordinates (-45,-8) chosen from the analysis in (A). At each split point $s$, the distribution of the observations with time points below $s$ is compared to the future observations. In the context of Section \ref{rhochoice}, a regression model predicting time is an option to quickly evaluate $\hat{\lambda}^{\rho}_{adapt}$ for several different splits. This corresponds to taking the squared error loss in Section \ref{rhochoice}. Here a Random Forest regression is used to predict time from the four variables. Each data point within a period defined by two splits is allocated into two sets (train and test) as follows: the first and last quartiles of the period are allocated to the training set, the rest (i.e. the middle part) is allocated to the testing set. Single splits through time can be then readily analyzed using $\hat{\lambda}_{adapt}$ on the test data. In addition to the analysis with real data here, \ref{simsec_2} shows some simulations results. 

Figure \ref{continuouschangeresults} summarizes the result of the analysis with split points $s$ considered marked in Figure \ref{fig:geo} by blue breaks: While Figure \ref{fig:geo} indicates that some change might be expected even after differencing, this impression is only confirmed for precipitation in Figure \ref{continuouschangeresults}. This hints at the fact that a shift is appearing in precipitation while for the other variables no change can marginally be detected. More interesting change is detectable once all $4$ variables are considered jointly. This is illustrated on the right in Figure \ref{continuouschangeresults}, where the estimated TV climbs to a (relatively) high value of around 0.14 between 1995 and and 2000. This corresponds to the high signal observed in Figure \ref{fig:climate_map} for these coordinates, only that here, the regression approach leads to a slightly lower $\hat{\lambda}_{adapt}$.

\begin{figure}
\begin{subfigure}[b]{.5\textwidth}
  \centering
  \includegraphics[width=7cm]{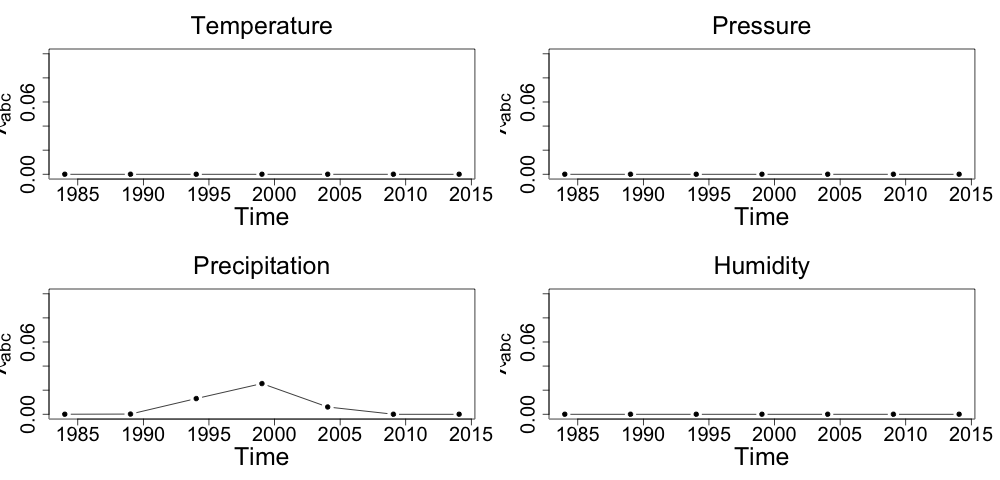}
  \end{subfigure}
  \begin{subfigure}[b]{.5\textwidth}
  \centering
  \includegraphics[width=7cm]{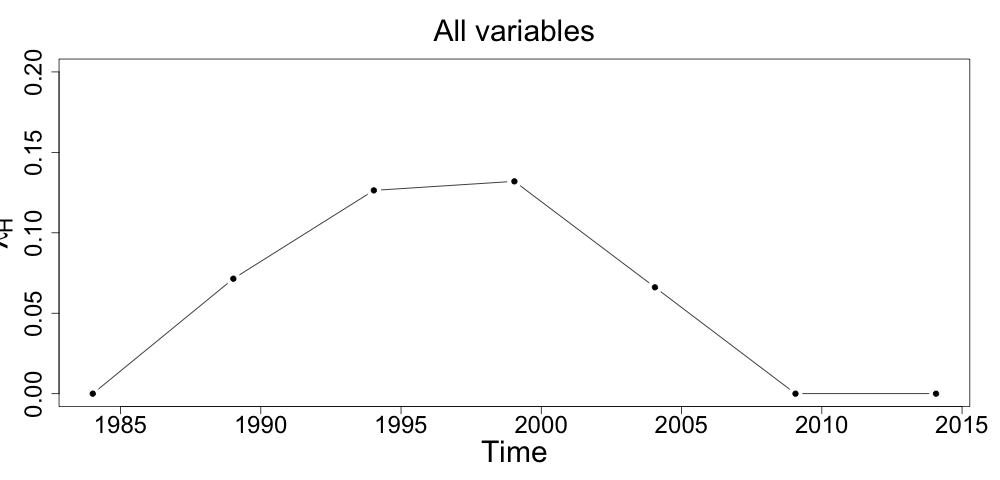}
\end{subfigure}
  \begin{subfigure}[b]{.5\textwidth}
  \centering
  \includegraphics[width=7cm]{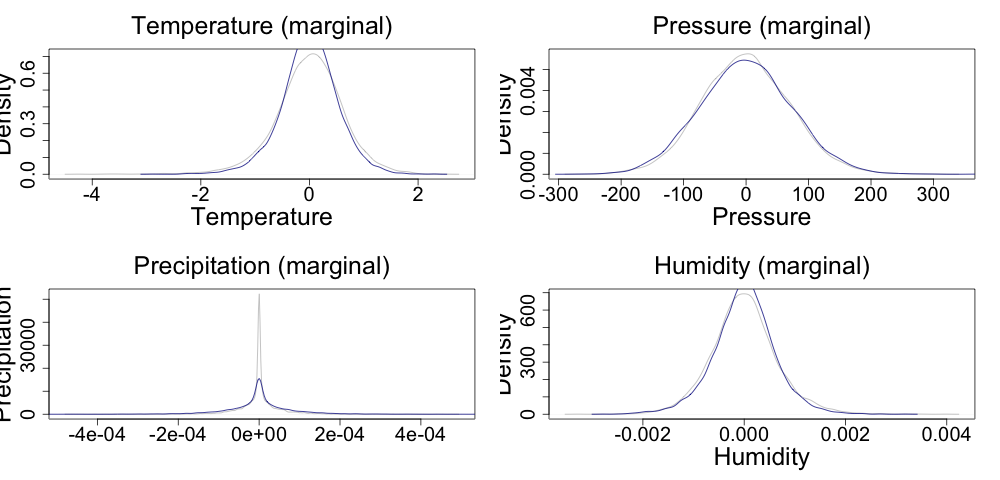}
\end{subfigure}
  \begin{subfigure}[b]{.5\textwidth}
  \centering
  \includegraphics[width=7cm]{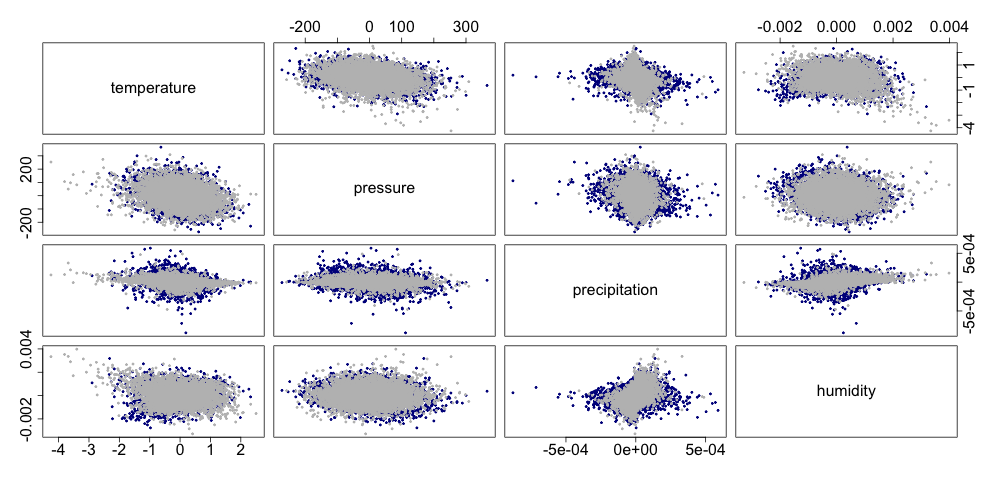}
\end{subfigure}
\caption{Top rows: High-probability lower bounds on total variation corresponding to $8$ breaks for differenced temperature, pressure, precipitation and humidity. On the left the marginal analysis, on the right the joint analysis. Bottom row: Corresponding analysis of marginal density estimates and pair plots.}
\label{continuouschangeresults}
\end{figure}

\section{Discussion}

We proposed in this paper two probabilistic lower bounds on the total variation distance between two distributions based on a one-dimensional projection. We theoretically characterized power rates given a sequence of (potentially random) projections $\rho_N$ and showed that the adaptive estimator always reaches the best possible rate. Application to a climate reanalysis dataset showcased potential use of these estimators in practice.

\clearpage

\bibliography{bib}{}
\bibliographystyle{plainnat}

\appendix

\section{Simulations} \label{simsec}

\subsection{Illustration of Results in Examples \ref{Example_1} and \ref{Example_2}} \label{simsec_1}




\begin{figure}[H]
\centering
\includegraphics[width=10cm]{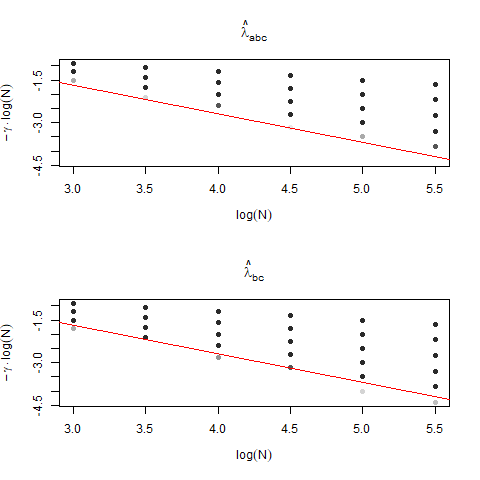}
\caption{Illustration of Proposition \ref{basicpowerresultBinomial} in Example \ref{Example_1}. For a range of different $\gamma$, $-\gamma \cdot \log(N)$ is plotted against $N$. For each $(\gamma, N)$ combination and for 100 repetitions, data was generated from the distribution in Example \ref{Example_1}. The dots indicate the number of times the estimator was strictly larger than zero, with points ranging from white (constituting values smaller 0.05) to black (constituting a value of 1). The red line shows a slope of -1 for comparison.}
\label{asymptoticpowerfig1}
\end{figure}

\begin{figure}[H]
\centering
\includegraphics[width=10cm]{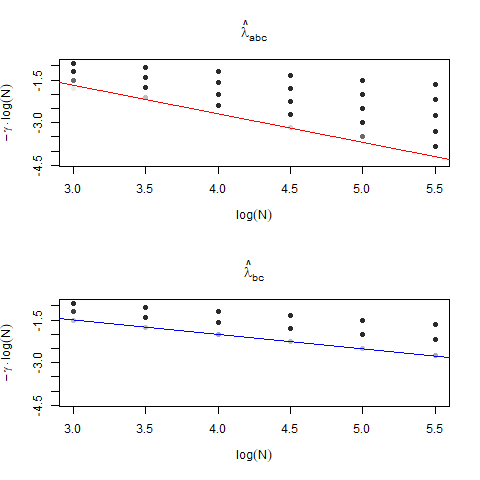}
\caption{Illustration of Propositions \ref{Example2prop} and \ref{Example3prop} of Example \ref{Example_2}. For a range of different $\gamma$, $-\gamma \cdot \log(N)$ is plotted against $N$. For each $(\gamma, N)$ combination and for 100 repetitions, data was generated from the distribution in Example \ref{Example_2}. The dots indicate the number of times the estimator was strictly larger than zero, with points ranging from white (constituting values smaller 0.05) to black (constituting a value of 1). The red and blue lines show slopes of -1 and -1/2 for comparison.}
\label{asymptoticpowerfig2}
\end{figure}

\subsection{Change Detection}\label{simsec_2}

We illustrate the change detection described in Section \ref{empiricssection} in some simple simulation settings. As in Section \ref{rhochoice} we study independent random variables $X_t$, $t \in I$, with each $X_t \sim P_t$ and $\mu$ being the distribution of $T$ on $I$. In all examples, we take $\mu$ to be the uniform distribution on $(0,1)$ and

\begin{itemize}
    \item[1)] simulate independently first $t$ from $\mu$ and then $X_t$ from $P_t$ to obtain a training and test set, each of size $n=10'000$,
    \item[2)] train a Random Forest Regression predicting $t$ from $X_t$ on the training data, resulting in the projection $\rho$,
    \item[3)] given $\rho$, evaluate $\hat{\lambda}_{adapt}$ on the test data for 19 $s$ ranging from 0.05 to 0.95 in steps of 0.05.
\end{itemize}

The first simulation considers 3 settings with univariate random variables $X_t$:
\begin{itemize}
    \item[(a)] A mean-shift, with $X_t \sim \mathcal{N}(0,1)$ for $0 \leq t \leq 1/2$, $X_t \sim \mathcal{N}(1,1)$ for $1/3 < t \leq 2/3$ and $X_t \sim \mathcal{N}(2,1)$ for $2/3 < t \leq 1$.
    \item[(b)] A variance shift, with $X_t \sim \mathcal{N}(0,1)$ for $0 \leq t \leq 1/2$, $X_t \sim \mathcal{N}(0,2)$ for $1/3 < t \leq 2/3$ and $X_t \sim \mathcal{N}(0,3)$ for $2/3 < t \leq 1$.
    \item[(c)] A continuous mean-shift, with $X_t \sim \mathcal{N}(2t, 1)$.
\end{itemize}

Results are given in Figure \ref{Simulation_marginals}.

The second simulation illustrates a covariance change in a bivariate example: For $t\leq 0.5$, $X_t=(X_{t,1}, X_{t,2}) \sim \mathcal{N}(0, \Sigma_0)$, while for $t> 0.5$, $X_t=(X_{t,1}, X_{t,2}) \sim \mathcal{N}(0, \Sigma_1)$, with 
\[
\Sigma_0= \begin{pmatrix} 1 & 0.5 \\ 0.5 & 1 \end{pmatrix} \text{ and } \Sigma_1=\begin{pmatrix} 1 & -0.5 \\ -0.5 & 1 \end{pmatrix}.
\]
The upper and middle part of Figure \ref{Simulation_bivariate} plots the marginal distributions $X_{t,i}$ against $t$. In all $T=1000$ cases, $\lambda_{adapt}^{\rho}$ (correctly) does not identify any changes in the two marginals. The change is however visible when considering the two variables jointly.

\begin{figure}[H]  
  \centering
  \includegraphics[width=12cm]{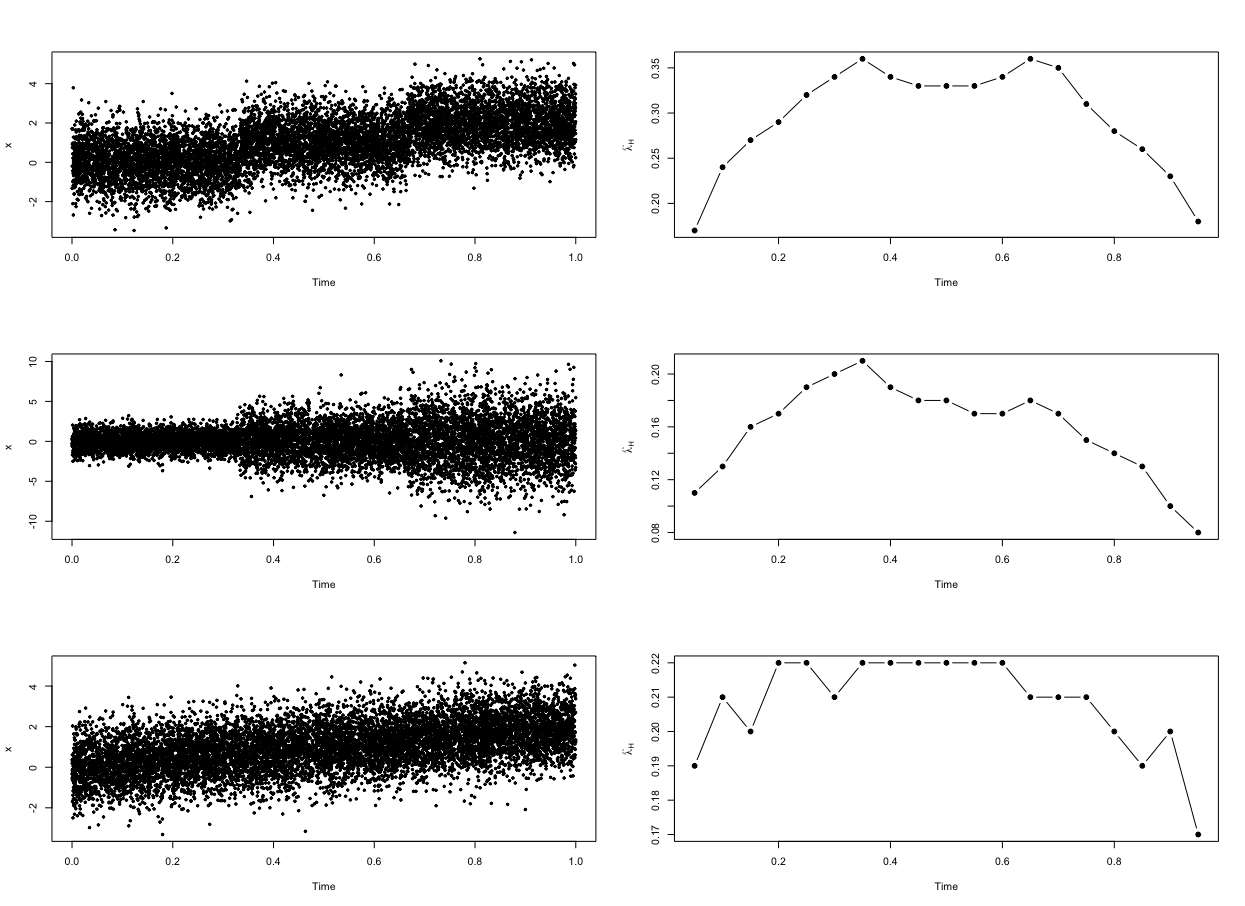}
  \caption{Top row: 3 regimes of mean-shifts. Middle row: 3 regimes of increasing variance, Bottom row: continuous mean-shift}
  \label{Simulation_marginals}
\end{figure}

\begin{figure}  
  \centering
  \includegraphics[width=12cm]{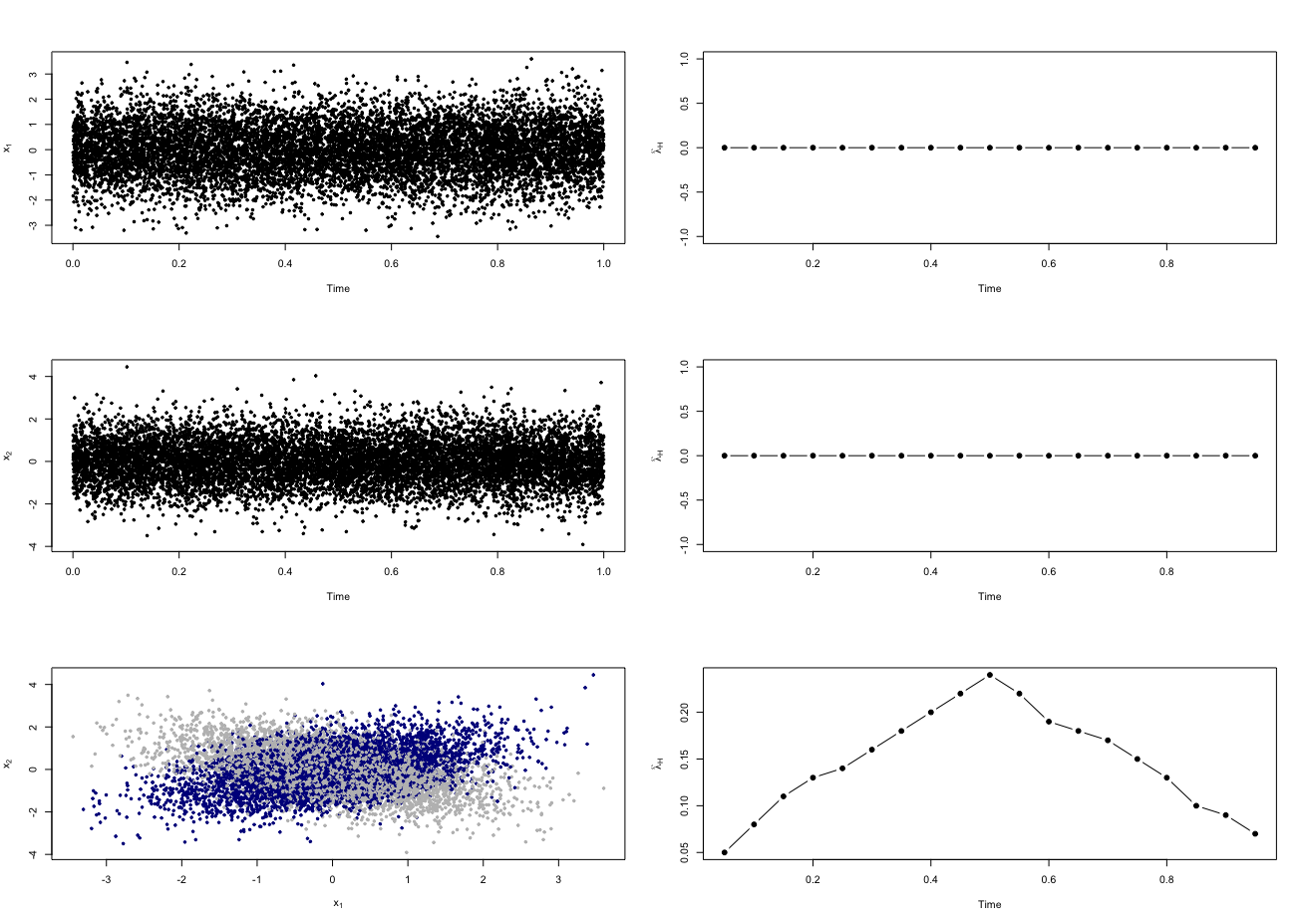}
  \caption{Top and middle row: marginal distributions; Bottom row: joint distribution.  }
  \label{Simulation_bivariate}
\end{figure}

\section{Analytical bounding function}\label{boundingexpressionsec}

Here we give an analytical expression for $\sq_{\alpha/3}\left(\tilde{V}_{m(\tilde{\lambda}),z}, z=1,\ldots, m(\tilde{\lambda})+ n(\tilde{\lambda}) -1  \right)$ in \eqref{TVsearcheqgeneral} based on the theory in \cite{finner2018}:

\begin{corollary}\label{boundingexpressioncor}
The following is a valid simultaneous confidence band in \eqref{TVsearcheqgeneral}:
\begin{align}\label{finnerconfidenceband}
    &\sq_{\alpha/3}\left(\tilde{V}_{m(\tilde{\lambda}),z}, z=1,\ldots, m(\tilde{\lambda})+ n(\tilde{\lambda}) -1  \right)=  (z - q_{1-\frac{\alpha}{3}}(\tilde{\lambda},m)) \frac{m(\tilde{\lambda})}{m(\tilde{\lambda})+n(\tilde{\lambda})} + \nonumber \\
    &\beta_{\alpha/3,m(\tilde{\lambda})} w\left(z - q_{1-\frac{\alpha}{3}}(\tilde{\lambda},m),m(\tilde{\lambda}),m(\tilde{\lambda})\right),
\end{align}

with
\[
\beta_{\alpha,m(\tilde{\lambda})}=\sqrt{2 \log(\log(m(\tilde{\lambda})))} + \frac{\log(\log(\log(m(\tilde{\lambda})))) -  \log(\pi) + 2 x_{\alpha/3} }{2 \sqrt{2 \log(\log(m(\tilde{\lambda})))}},
\]
and $w(z,m,n)$ defined as in \eqref{Hypervariance}. 
\end{corollary}

\begin{proof}
Applying Lemma \ref{inftylemma} with $p=1$ and $p_{\varepsilon}:=\lambda$, $m(\lambda)$, $n(\lambda)$ go to infinity as $m,n \to \infty$. Moreover, since we assume $m/N \to \pi \leq 1/2$, it holds that
\[
\lim_{N \to \infty} \frac{m(\lambda)}{n(\lambda)} = \frac{\pi}{1-\pi} \leq 1.
\]
Thus for all but finitely many $N$, it holds that $m(\lambda) \leq n(\lambda)$. Combining this together with the fact that $ \tilde{V}_{m,z - q_{1-\frac{\alpha}{3}}(\lambda,m)} - Q_{m,n, \alpha}(z,\lambda)$, $z \in \tilde{J}_{m,n,\lambda}$, is just a hypergeometric process adjusted by the correct mean and variance, it follows from the arguments in \cite{finner2018}:
\begin{align*}
\limsup_{N \to \infty}  \mathbb{P}\Big(\sup_{z \in \tilde{J}_{m,n,\lambda}} &\Big[  \tilde{V}_{m,z - (q_{1-\frac{\alpha}{3}}(\lambda,m))}  -  (z - q_{1-\frac{\alpha}{3}}(\lambda,m))\frac{m(\lambda)}{m(\lambda)+n(\lambda)} \\
&-  \beta_{\alpha/3,m(\lambda)} w\left(z - q_{1-\frac{\alpha}{3}}(\lambda,m) ,m(\lambda),n(\lambda)\right)\Big] > 0 \Big)\leq \frac{\alpha}{3}.
\end{align*}
Thus \eqref{Q1demonstration} indeed holds.
\end{proof}

\section{Proofs} \label{proofStoDom}

Here we present the proofs of our main results. We start with a few preliminaries: In Section \ref{sec:meth}, we defined for two functions $h_1,h_2: \mathbb{N} \rightarrow [0,+\infty)$, the notation $h_1(N) \asymp h_2(N)$ to mean that both (1) $\limsup_{N \to \infty} h_1(N)/h_2(N)  \leq a_1$, for some $a_1 \in \R^{+}$ and (2) $\limsup_{N \to \infty} h_2(N)/h_1(N)  \leq a_2$, for some $a_2 \in \R^{+}$. If instead only (1) is known, we write $h_1(N) = O(h_2(N))$ (translated as ``asymptotically larger equal''). If (1) is known to hold for $a_1=0$, we write $h_1(N) = o(h_2(N))$ (translated as ``asymptotically strictly smaller'').

The technical lemmas of Section \ref{techlemmas} should serve as a basis for the results in Section \ref{proofsection1} to \ref{proofsection2}. They ensure that we may focus on the most convenient case, when $(t_N)_{N \geq 1}$ is such that $N \sigma(t_N) \to \infty$ (Lemma \ref{ignorebadcase}) or $m \sigma_F \to \infty$ (Lemma \ref{ignorebadcase2}) holds. For these sequences, Lemma \ref{Gaussianresult} shows that,
\begin{equation}\label{convergenceindist0}
    \frac{\hat{F}_m(t_N) - \hat{G}_n(t_N) - (F(t_N) - G(t_N))}{\sigma(t_N)} \stackrel{D}{\to} \mathcal{N}(0,1), \text{ for } N \to \infty. 
\end{equation}

We will now summarize the main proof ideas for the most important results. For Propositions \ref{oraclelevel} and \ref{oracleprop}, providing the level and power of $\hat{\lambda}(t_N)$ respectively, we use Lemma \ref{Gaussianresult} and \ref{ignorebadcase} to obtain \eqref{convergenceindist0}. From this, Proposition \ref{oraclelevel} directly follows. It moreover implies that Proposition \ref{oracleprop} holds iff
\begin{equation*}
    \frac{ \lambda_N[ (1-\varepsilon) - \lambda(t_N)/\lambda_N]}{\sigma(t_N)}  \to -\infty \iff \eqref{CondII} \text{ and } \eqref{CondI}.
\end{equation*}
This is simple, as both \eqref{CondII} and \eqref{CondI} were designed such that this equivalence holds.


We start in a similar manner to obtain the power result for $\hat{\lambda}_{adapt}$ in Proposition \ref{newamazingresult}. We first restate the bounding function $Q^{A}$, for $z(t) \in \{q_{1-\alpha}(\tilde{\lambda}, m),\ldots, m + n(\tilde{\lambda})  \}$,
\begin{align} \label{Qmndef}
    Q_{m,n, \alpha}(z(t), \tilde{\lambda})  &=q_{1-\frac{\alpha}{3}}(\tilde{\lambda}, m) \frac{n(\tilde{\lambda})}{N(\tilde{\lambda})}  + z(t)\frac{m(\tilde{\lambda})}{N(\tilde{\lambda})} + \beta_{\frac{\alpha}{3}, m(\tilde{\lambda})}\sqrt{ \frac{m(\tilde{\lambda})}{N(\tilde{\lambda})} \frac{n(\tilde{\lambda})}{N(\tilde{\lambda})} \frac{N(\tilde{\lambda}) - z(t)}{N(\tilde{\lambda}) -1 } (z(t) -q_{1-\frac{\alpha}{3}}(\tilde{\lambda}, m)) },
\end{align}
with $N(\tilde{\lambda})=m(\tilde{\lambda}) + n(\tilde{\lambda})$. Lemma \ref{ignorebadcase2} ensures that we may focus on the case $m \sigma_F \to \infty$. This immediately implies $(V_{m,t_N} - mA_0(t_N) )/\sigma_F \stackrel{D}{\to} \mathcal{N}(0,1)$ due to the Lindeberg-Feller CLT (see e.g., \citet[Chapter 2]{vaart_1998}). Using Lemma \ref{simplified} we show that what we would like to prove,
\[
  \Prob(V_{m,t_N}  > Q_{m,n, \alpha}(z(t_N), \tilde{\lambda} ) \ \  \forall \tilde{\lambda} \in [0, \lambda_{\varepsilon}] ) \to 1 \iff \eqref{CondII} \text{ and } \eqref{CondI},
\]
can be replaced by the much simpler
\[
  \Prob(V_{m,t_N}  > \tilde{Q}(\varepsilon) ) \to 1 \iff \eqref{CondII} \text{ and } \eqref{CondI},
\]
where
\begin{align*}
 \tilde{Q}(\tilde{\lambda}) &= m \tilde{\lambda}(1-\pi)+ m [\pi A_0(t_N)  - (1-\pi) A_1(t_N) + (1-\pi)],
\end{align*}
can be seen as the ``limit'' of an appropriately scaled $Q_{m,n, \alpha}(z(t_N), \tilde{\lambda} )$. Using the structure of the problem and the asymptotic normality of $(V_{m,t_N} - m A_0(t_N))/\sigma_F$, we show that the result simplifies to showing that \begin{equation}
    \frac{ \lambda_N[ (1-\varepsilon) - \lambda(t_N)/\lambda_N]}{\sigma(t_N)}  \to -\infty \iff \eqref{CondII} \text{ and } \eqref{CondI},
\end{equation}
which was already done in Proposition \ref{oracleprop}.

On the other hand, to prove that $\hat{\lambda}_{adapt}$ is an asymptotic HPLB, we need to prove Propositions \ref{witsearch} and \ref{Qfunctions}. The former is immediate with an infimum argument, whereas the latter requires some additional concepts. In particular, we use the bounding operation described in Lemma \ref{barVmzproperties} to bound the original $V_{m,z}$ process pointwise for each $z$ by the well behaved $\bar{V}_{m,z}$. The randomness of this process is essentially the one of the hypergeometric process $\tilde{V}_{m(\tilde{\lambda}),z}, z=1,\ldots, m(\tilde{\lambda})+ n(\tilde{\lambda}) -1$, as introduced in Section \ref{TVsearchec}. The assumptions put on the bounding function $Q$, then ensure that we conserve the level.

The next three Section will provide the proofs of the main results, while Section \ref{techlemmas} collects the aforementioned technical lemmas.

\subsection{Proofs for Section \ref{sec:meth}} \label{proofsection1}

In this section, we prove the main results of Section \ref{sec:meth}, except for Propositions \ref{Example0prop}, \ref{basicpowerresultBinomial}, \ref{Example2prop} and \ref{Example3prop} connected to Examples \ref{Example_1} and \ref{Example_2}. Their proofs will be given in Section \ref{proofsection2}.


\oraclelevel*

\begin{proof}
From Lemma \ref{ignorebadcase} (III) in Section \ref{techlemmas}, we may assume that for the given $\left(P, Q, (t_N)_{N \geq 1}, \rho \right)$, $N \sigma(t_N) \to +\infty$, as $N \to \infty$. In this case, we know from Lemma \ref{Gaussianresult} that,
\begin{equation}
    \frac{\hat{F}_m(t_N) - \hat{G}_n(t_N) - (F(t_N) - G(t_N))}{\sigma(t_N)} \stackrel{D}{\to} \mathcal{N}(0,1). 
\end{equation}
Consequently,
\begin{align*}
    \limsup_{N \to \infty} \Prob\left( \hat{\lambda}^{\rho}(t_N) > F(t_N) - G(t_N) \right) &= \lim_N \Prob \left(\frac{\hat{F}_m(t_N) - \hat{G}_n(t_N) - (F(t_N) - G(t_N)) }{\sigma(t_N)} >q_{1-\alpha}  \right)\\
    &= \alpha.
\end{align*}
Since $\lambda \geq F(t_N) - G(t_N)$, the result then follows.
\end{proof}

The exact same proof can also be used to show that Proposition \ref{oraclelevel} holds true, for $\lambda=N^{\gamma}$, $-1 < \gamma \leq 0$, as long as $\lambda > 0$ for all finite $N$. 

Proposition \ref{Binom2} follows directly from Proposition \ref{oraclelevel} by exchanging $\sigma(t_N)$ with the consistent estimator used in $\hat{\lambda}_{bayes}^{\rho}$ and after checking that the case \eqref{NC} cannot happen, for a \emph{fixed} $\rho$.


\witsearch*

\begin{proof}

Let,
\[
G_{m,n}:=\left \{\tilde{\lambda} \in [0,1]: \sup_{z \in J_{m,n}} \left[  V_{m,z} -  Q_{m,n, \alpha}(z,\tilde{\lambda})\right] \leq 0 \right \}.
\]
Then by definition of the infimum,
\begin{align*}
    \mathbb{P}(\hat{\lambda}^{\rho} > \lambda) &\leq \mathbb{P}(\lambda \in G_{m,n}^c) \\
    &=\mathbb{P}(\sup_{z \in J_{m,n}} \left[  V_{m,z} -  Q_{m,n, \alpha}(z,\lambda)\right] > 0).
\end{align*}
The result follows by definition of $Q_{m,n, \alpha}$.
\end{proof}

To prove Proposition \ref{Qfunctions}, we need two technical concepts introduced in Section \ref{techlemmas}. In particular we utilize the concept of Distributional Witnesses in Definition \ref{Witdef} and the bounding operation in Lemma \ref{barVmzproperties}.


\Qfunctions*

\begin{proof}
We aim to prove
\begin{equation} \label{Q1demonstration}
            \limsup_{N \to \infty}  \mathbb{P}(\sup_{z \in J_{m,n}} \left[  V_{m,z} -  Q_{m,n, \alpha}(z,\lambda)\right] > 0) \leq \alpha.
\end{equation}

Let $\Lambda_{P}, \Lambda_{Q}$ be the distributional Witnesses of $P$ and $Q$, as in Definition \ref{Witdef}. Define the events $A_{P}:=\{ \Lambda_{P} \leq q_{1-\frac{\alpha}{3}}(\lambda,m) \}$, $A_{Q}:=\{ \Lambda_{Q} \leq q_{1-\frac{\alpha}{3}}(\lambda,n) \}$ and $A=A_{P} \cap A_{Q}$, such that $\mathbb{P}(A^c)\leq 2\alpha/3$. On $A$, we overestimate the number of witnesses on each side by construction. In this case we are able to use the bounding operation described above with $\bar{\Lambda}_P=q_{1-\frac{\alpha}{3}}(\lambda,m)$ and $\bar{\Lambda}_Q=q_{1-\frac{\alpha}{3}}(\lambda,n)$  to obtain $\bar{V}_{m,z}$ from Lemma \ref{barVmzproperties}. The process $\bar{V}_{m,z}$ has
\begin{align*}
   \bar{V}_{m,z}= \begin{cases}
    z, &\text{ if } 1 \leq z \leq q_{1-\frac{\alpha}{3}}(\lambda,m)\\
    m, &\text{ if } m+n(\lambda)\leq z \leq m+n\\ 
    \tilde{V}_{m,z - q_{1-\frac{\alpha}{3}}(\lambda,m)}+ q_{1-\frac{\alpha}{3}}(\lambda,m) , &\text{ if } q_{1-\frac{\alpha}{3}}(\lambda,m) <  z < m+n(\lambda),
    \end{cases}
\end{align*}
where $m(\lambda)=n-q_{1-\frac{\alpha}{3}}(\lambda,m)$, $n(\lambda)=n-q_{1-\frac{\alpha}{3}}(\lambda,n)$, and $\tilde{V}_{m,z - q_{1-\frac{\alpha}{3}}(\lambda,m)} \sim \text{Hypergeometric}(z-q_{1-\frac{\alpha}{3}}(\lambda,m), m(\lambda), m(\lambda) + n(\lambda))$. Then:
\begin{align*}
    &\mathbb{P}(\sup_{z \in J_{m,n}} \left[  V_{m,z} -  Q_{m,n, \alpha}(z,\lambda)\right] > 0 )\leq \frac{2\alpha}{3} +  \mathbb{P}(\sup_{z \in J_{m,n}} \left[  \bar{V}_{m,z} -  Q_{m,n, \alpha}(z,\lambda)\right] > 0 \cap A),
\end{align*}
Now, $\bar{V}_{m,z} -  Q_{m,n, \alpha}(z,\lambda) > 0$ can only happen for $ z \in \tilde{J}_{m,n,\lambda}:=\{q_{1-\frac{\alpha}{3}}(\lambda,m) + 1, \ldots, m+n(\lambda) -1 \}$, as by construction $\bar{V}_{m,z} -  Q_{m,n, \alpha}(z,\lambda) = 0$, for $ z \notin \tilde{J}_{m,n,\lambda}$. Thus
\begin{align*}
&\limsup_{N \to \infty}  \mathbb{P}(\sup_{z \in J_{m,n}} \left[  \bar{V}_{m,z} -  Q_{m,n, \alpha}(z,\lambda)\right] > 0 \cap A) \leq \frac{2\alpha}{3} +   \\
%
       %
       &\limsup_{N \to \infty}  \mathbb{P}(\sup_{z \in \tilde{J}_{m,n,\lambda}} \left[  \tilde{V}_{m,z - q_{1-\frac{\alpha}{3}}(\lambda,m)}  -  \sq_{\alpha/3}\left(\tilde{V}_{m(\lambda),z - q_{1-\frac{\alpha}{3}}(\lambda,m)}, z \in \tilde{J}_{m,n,\lambda}  \right)\right] > 0)\\
       &\leq \alpha, 
\end{align*}
by definition of $\sq_{\alpha/3}\left(\tilde{V}_{m(\lambda),z}, z=1,\ldots, m(\lambda)+ n(\lambda) -1  \right)$.


\end{proof}

\subsection{Proofs for Section \ref{powersec}}\label{proofsection2}




\oracleprop*

\begin{proof}
According to Lemma \ref{ignorebadcase} we are allowed to focus on sequences $(t_N)_{N \geq 1}$ such that $N\sigma(t_N) \to \infty$. For $(t_N)_{N \geq 1} \subset I$ such a sequence, it holds that
\begin{align*}
    \mathbb{P}(\hat{\lambda}^{\rho^*}(t_N) > (1-\varepsilon) \lambda_N)&= \mathbb{P}(\hat{F}_m(t_N) - \hat{G}_n(t_N) - q_{1-\alpha}\sigma(t_N)   > (1-\varepsilon) \lambda_N)\\
    &=\mathbb{P}\left(\frac{(\hat{F}_m(t_N) - \hat{G}_n(t_N) -  \lambda(t_N))}{\sigma(t_N)}  > q_{1-\alpha} -\frac{ \lambda_N[ (1-\varepsilon) - \lambda(t_N)/\lambda_N]}{\sigma(t_N)}  \right),
\end{align*}
where as in Section \ref{powersec}, $\lambda(t_N)=F(t_N) - G(t_N)$.
With the same arguments as in Proposition \ref{oraclelevel}, $(\hat{F}_m(t_N) - \hat{G}_n(t_N) -  \lambda(t_N))/\sigma(t_N) \stackrel{D}{\to} \mathcal{N}(0,1)$. Thus, $ \mathbb{P}\left(\hat{\lambda}^{\rho^*}(t_N) > (1-\varepsilon)\lambda_N\right) \to 1$, iff 
\begin{equation}\label{anotheroneofthoseconditions}
    \frac{ \lambda_N[ (1-\varepsilon) - \lambda(t_N)/\lambda_N]}{\sigma(t_N)}  \to -\infty.
\end{equation}

For $\gamma > -1/2$, \eqref{CondI}, \eqref{CondII} and $\eqref{anotheroneofthoseconditions}$ are all true for $t_N=1/2$, so there is nothing to prove in this case.

For $\gamma \leq -1/2$, assume \eqref{CondII} and \eqref{CondI} are true for $(t_N)_{N \geq 1}$. Then if $\liminf_{N \to \infty} \lambda(t_N)/\lambda_N > 1-\varepsilon$,
\begin{align*}
    \frac{ \lambda_N[ (1-\varepsilon) - \lambda(t_N)/\lambda_N]}{\sigma(t_N)} \leq  \frac{ \lambda_N \left[ (1-\varepsilon) - \inf_{M \geq N }\lambda(t_M)/\lambda_M \right]}{\sigma(t_N)} \to -\infty,
\end{align*}
as $\left[ (1-\varepsilon) - \inf_{M \geq N }\lambda(t_M)/\lambda_M \right] < 0$ for all but finitely many $N$ and $\lambda_N/\sigma(t_N) \to +\infty$, by \eqref{CondI.1}. If instead $\liminf_{N \to \infty} \lambda(t_N)/\lambda_N = 1-\varepsilon$, the statement follows immediately from \eqref{CondI.2}. This shows one direction.

On the other hand, assume for all $(t_N)_{N \geq 1}$ \eqref{CondII} or \eqref{CondI} is false. We start by assuming the negation of \eqref{CondII}, i.e., $\liminf_{N \to \infty} \lambda(t_N)/\lambda_N < 1-\varepsilon$. Then there exists for \emph{all} $N$ an $M \geq N$ such that $\lambda(t_M)/\lambda_M \leq 1-\varepsilon$, or 
\[
 \frac{ \lambda_M[ (1-\varepsilon) - \lambda(t_M)/\lambda_M]}{\sigma(t_M)}  \geq 0. 
\]
This is a direct contradiction of \eqref{anotheroneofthoseconditions}, which by definition means that for large enough $N$ all elements of the above sequence are below zero. Now assume \eqref{CondI.1} is wrong, i.e. $\liminf_{N \to \infty} \lambda(t_N)/\lambda_N > 1-\varepsilon$, but $\liminf_{N \to \infty} \lambda_N/\sigma(t_N) < \infty $. Since $ \lambda(t_N)/\lambda_N \in [0,1]$ for all $N$, the lower bound will stay bounded away from $-\infty$ in this case. More specifically,
\[
\liminf_{N \to \infty}  \frac{\lambda_N}{\sigma(t_N)} [ (1-\varepsilon) - \lambda(t_N)/\lambda_N] \geq \liminf_{N \to \infty}  \frac{\lambda_N}{\sigma(t_N)} \liminf_N [ (1-\varepsilon) - \lambda(t_N)/\lambda_N] > -\infty.
\]
The negation of \eqref{CondI.2} on the other hand, leads directly to a contradiction with \eqref{anotheroneofthoseconditions}.
Consequently, by contraposition, the existence of a sequence $(t_N)_{N \geq 1}$ such that $\eqref{CondI}$ and $\eqref{CondII}$ are true is necessary.

\end{proof}


\nofastratebinomial*

\begin{proof}
Since $F(t)=A_0(t)$ and $1-G(t)=A_1(t)$, it holds that $\sigma(1/2)/\hat{\sigma}(1/2) \to 1$ almost surely. Thus, the same arguments as in the proof of Proposition \ref{oracleprop} with $t_N=1/2$ give the result. 
\end{proof}


\newamazingresult*

\begin{proof}

Let for the following $\varepsilon \in (0,1]$ be arbitrary. The proof will be done by reducing to the case of $\hat{\lambda}(t_N)$. 
For a sequence $(t_N)_{N \geq 1}$ and a given sample of size $N$ we then define the (random) $z(t_N)$, with
\begin{equation}\label{zdef}
     z(t_N)=\sum_{j=1}^m \Ind\{\rho^*(X_j) \leq  t_N \} + \sum_{i=1}^n \Ind\{\rho^*(X_i) \leq  t_N \}=m \hat{F}(t_N) + n \hat{G}(t_N).
\end{equation}
Since by definition the observations $\rho^*_{(1)}, \ldots \rho^*_{(z(t_N))}$ are smaller $t_N$, the classifier $\tilde{\rho}_{t_N}(z):=\Ind\{\rho^*(z) > t_N \}$ will label all corresponding observations as zero. As such the number of actual observations coming from $P$ in $\rho^*_{(1)}, \ldots \rho^*_{(z(t_N))}$, $V_{m,t_N}$, will have $V_{m,t_N} \sim \mbox{Bin}(A_0(t_N), m)$. Recall that
\[
A_0(t_N)=A_0^{\rho^*}(t_N)=P( \tilde{\rho}_{t_N}(X) = 0), \ \ A_1(t_N)=A_1^{\rho^*}(t_N)=Q( \tilde{\rho}_{t_N}(Y) = 1),
\]
i.e. the true accuracies of the classifier $\tilde{\rho}_{t_N}$.

The goal is to show that we overshoot the quantile $Q_{m,n, \alpha}$:
\begin{equation}\label{whatwewant}
    \Prob(V_{m,t_N}  > Q_{m,n, \alpha}(z(t_N), \tilde{\lambda} ) \ \  \forall \tilde{\lambda} \in [0, \lambda_{\varepsilon}] ) \to 1,
\end{equation}
if and only if there exists a $(t_N)$ such that \eqref{CondI} and \eqref{CondII} hold. For this purpose, Lemma \ref{ignorebadcase2} emulates Lemma \ref{ignorebadcase} to allow us to focus on $(t_N)_{N \geq 1}$ such that $m A_0(t_N)(1-A_0(t_N)) \to \infty$.

A sufficient condition for \eqref{whatwewant} is 
\begin{equation} \label{currentgoalsufficient}
    \Prob \left( \frac{V_{m,t_N}- m A_0(t_N)}{\sqrt{m A_0(t_N)(1-A_0(t_N))}}  >  \frac{\sup_{\tilde{\lambda} }Q_{m,n, \alpha}(z(t_N), \tilde{\lambda}) - m A_0(t_N)}{\sqrt{m A_0(t_N)(1-A_0(t_N))}}   \right) \to 1,
\end{equation}
while a necessary condition is given by 
\begin{equation} \label{currentgoalnecessary}
    \Prob \left( \frac{V_{m,t_N}- m A_0(t_N)}{\sqrt{m A_0(t_N)(1-A_0(t_N))}}  >  \frac{ Q_{m,n, \alpha}(z(t_N), \lambda_{\varepsilon}) - m A_0(t_N)}{\sqrt{m A_0(t_N)(1-A_0(t_N))}}   \right) \to 1.
\end{equation}
We instead work with a simpler bound:
\begin{align}\label{Qtildedef}
 \tilde{Q}(\tilde{\lambda}) &= m \tilde{\lambda}(1-\pi)+ m [\pi A_0(t_N)  - (1-\pi) A_1(t_N) + (1-\pi)]. 
\end{align}
Note that
\begin{equation}
    \sup_{\tilde{\lambda} \in [0, \lambda_{\varepsilon}] } \tilde{Q}(\tilde{\lambda} ) = \tilde{Q}(\lambda_{\varepsilon}).
\end{equation}
and
\begin{align*}
   \tilde{Q}(\lambda_{\varepsilon})- m A_0(t_N) &= m (1-\pi)[\lambda_{\varepsilon} -   (A_0(t_N)(t_N)  + A_1(t_N) - 1) ]\\
   &= m (1-\pi)[\lambda_{\varepsilon} -   \lambda(t_N) ].
\end{align*}

We first show that 
\begin{equation}\label{currentgoal2}
    \Prob \left( \frac{V_{m,t_N}- m A_0(t_N)}{\sqrt{m A_0(t_N)(1-A_0(t_N))}}  >  \frac{\tilde{Q}(\lambda_{\varepsilon}) - m A_0(t_N)}{\sqrt{m A_0(t_N)(1-A_0(t_N))}}   \right) \to 1,
\end{equation}
if and only if there exists a $(t_N)$ such that \eqref{CondI} and \eqref{CondII} hold.

Since again $\frac{V_{m,t_N}- m A_0(t_N)}{\sqrt{m A_0(t_N)(1-A_0(t_N))}} \stackrel{D}{\to} N(0, 1)$, due to the Lindeberg-Feller CLT (\cite{vaart_1998}), \eqref{currentgoal2} holds iff 
\begin{equation}\label{anotheroneofthoseconditions2}
    \frac{ \tilde{Q}(\lambda_{\varepsilon} ) - m A_0(t_N)}{\sqrt{m A_0(t_N)(1-A_0(t_N))}}   \to - \infty.
\end{equation}
To prove this claim, we write
\begin{align}\label{}
    \frac{ \tilde{Q}(\lambda_{\varepsilon} ) - m A_0(t_N)}{\sqrt{m A_0(t_N)(1-A_0(t_N))}} = (1-\pi) \frac{\lambda_N [(1-\varepsilon) -   \lambda(t_N)/\lambda_N ]}{\sqrt{ \frac{A_0(t_N)(1-A_0(t_N))}{m}}}
\end{align}
and show that
\begin{equation}\label{anotheroneofthoseconditions3}
    \frac{\lambda_N}{\sqrt{ \frac{A_0(t_N)(1-A_0(t_N))}{m}}} \to +\infty \iff \frac{\lambda_N}{\sigma(t_N)} \to +\infty.
\end{equation}
In this case, \eqref{anotheroneofthoseconditions2} is equivalent to \eqref{anotheroneofthoseconditions} and it follows from exactly the same arguments as in the proof of Proposition \ref{oracleprop} that \eqref{anotheroneofthoseconditions2} is true iff there exists a $(t_N)$ such that \eqref{CondI} and \eqref{CondII} hold. 

To prove \eqref{anotheroneofthoseconditions3}, first assume that
\[
 \frac{\lambda_N}{\sqrt{ \frac{A_0(t_N)(1-A_0(t_N))}{m}}} \to +\infty.
\]
This implies that $A_0(t_N)(1-A_0(t_N))=o(N^{2\gamma + 1})$, which means that either $A_0(t_N)=o(N^{2\gamma + 1})$ or $(1-A_0(t_N))=o(N^{2\gamma + 1})$. Assume $A_0(t_N)=o(N^{2\gamma + 1})$. Since by definition $A_0(t_N) + A_1(t_N) - 1 =\lambda(t_N) = O(\lambda_N)$, this means that $1-A_1(t_N) = O(N^{\gamma}) + o(N^{2\gamma + 1})=o(N^{2\gamma + 1})$ and thus also $A_1(t_N)(1-A_1(t_N))=o(N^{2\gamma + 1})$. The same applies for $1-A_0(t_N)=o(N^{2\gamma + 1})$. Writing $\sigma(t)$ as in \eqref{sigmat2} this immediately implies $\frac{\lambda_N}{\sigma(t_N)} \to +\infty$. On the other hand, assume $\frac{\lambda_N}{\sigma(t_N)} \to +\infty$. This in turn means
\begin{equation}\label{A0A1condition}
    A_0(t_N)(1-A_0(t_N)) + A_1(t_N)(1-A_1(t_N))=o(N^{2\gamma + 1})
\end{equation}
and thus $A_0(t_N)(1-A_0(t_N))=o(N^{2\gamma + 1})$ and $\lambda_N/\sqrt{ \frac{A_0(t_N)(1-A_0(t_N))}{m}} \to +\infty$. This proves \eqref{anotheroneofthoseconditions3}. Using the arguments of the proof of Proposition \ref{oracleprop} this demonstrates that \eqref{anotheroneofthoseconditions2} is true iff there exists a $(t_N)$ such that \eqref{CondII} and \eqref{CondI} hold. 

It remains to show that $\eqref{currentgoal2}$ implies \eqref{currentgoalsufficient} and is implied by $\eqref{currentgoalnecessary}$. More specifically, as \eqref{CondI} demands that 
\[
A_0(t_N)(1-A_0(t_N)) = o(N^{2\gamma + 1}) \text{ and } A_1(t_N)(1-A_1(t_N)) = o(N^{2\gamma + 1}),
\] 
we may use Lemma \ref{simplified} below to see that for $c \in (0,+\infty)$,
\begin{align*}
   c + o_{\Prob}(1) \leq \frac{ Q_{m,n, \alpha}(z(t_N), \lambda_{\varepsilon}) - m A_0 }{\tilde{Q}(\lambda_\varepsilon) - m A_0  }  \leq   \frac{\sup_{\tilde{\lambda}} Q_{m,n, \alpha}(z(t_N), \tilde{\lambda}) - m A_0 }{\tilde{Q}(\lambda_\varepsilon) - m A_0  } \leq  \frac{1}{c} + o_{\Prob}(1).
\end{align*}

For $Z_N \stackrel{D}{\to} N(0,1)$, $Q_{1,N} \to -\infty$ and $ c + O_N \leq  Q_{2,N}/Q_{1,N}$, with $c > 0$ and $O_N \stackrel{p}{\to} 0$, it holds that
\begin{align*}
    \Prob(Z_N > Q_{2,N} ) &=  \Prob\left(Z_N > \frac{Q_{2,N}}{Q_{1,N}} Q_{1,N}   \right)\\
    & \geq \Prob\left(Z_N > (c+O_N) Q_{1,N}  \right)\\
    &= \Prob\left(Z_N > (c+O_N) Q_{1,N}  \cap  |O_N| \leq  \frac{c}{2} \right) + \Prob\left(Z_N > (c+O_N) Q_{1,N}  \cap  |O_N| >  \frac{c}{2} \right)\\
    & \to 1,
\end{align*}
as $Q_{1,N} < 0$ for all but finitely many $N$ and $(c+O_N) > 0$ on the set $|O_N| \leq  \frac{c}{2}$. Using this argument first with $Q_{1,N}= \tilde{Q}(\lambda_\varepsilon) - m A_0(t_N)$ and $Q_{2,N}= Q_{m,n, \alpha}(z(t_N), \lambda_{\varepsilon}) - m A_0(t_N) $, and repeating it with $Q_{1,N}=Q_{m,n, \alpha}(z(t_N), \lambda_{\varepsilon}) - m A_0(t_N) $ and $Q_{2,N}=  \tilde{Q}(\lambda_\varepsilon) - m A_0(t_N)$, \eqref{QQbounding} shows that $\eqref{currentgoal2}$ implies \eqref{currentgoalsufficient} and is implied by $\eqref{currentgoalnecessary}$.

\end{proof}

We are now able to prove the results in Examples \ref{Example_1} and \ref{Example_2}:


\Examplezeroprop*

\begin{proof}

First note that $\rho(z)= (1-p) \Ind\{ -1 \leq z \leq 0 \} + p \Ind\{ 0 \leq z \leq 1 \}$ and thus 
\begin{equation}\label{A0A1Example1}
    A_0(t_N)=p \Ind\{ 1-p \leq t_N \leq p \} + \Ind\{ t_N > p \}, \ \  A_1(t_N)=p \Ind\{ 1-p \leq t_N \leq p \} + \Ind\{ t_N < 1-p \}  
\end{equation}
Take any $\gamma \leq -1/2$. Then for $\lambda_N/\sigma(t_N) \to \infty$ to be true it is necessary that $A_0(t_N)(1-A_0(t_N))$ and $A_1(t_N)(1-A_1(t_N))$ go to zero. But from \eqref{A0A1Example1} and the fact that $p \to 0.5$, it is clear that this is only possible for $t_N \in [1-p,p]^c $ for all but finitely many $N$. However for such $t_N$, $\lambda(t_N)=A_0(t_N) + A_1(t_N)-1=0$. Similarly, a sequence $(t_N)_{N \geq 1}$ that satisfies \eqref{CondI}, cannot satisfy condition \eqref{CondII}. Thus for $\gamma \leq -1/2$ for any sequence $(t_N)_{N \geq 1}$ most one of the two conditions \eqref{CondII} and \eqref{CondI} can be true and thus $- 1/2 \leq \underline{\gamma}^{oracle}(\varepsilon)$.
On the other hand, for $\gamma > -1/2$, taking $t_N=1/2$ independently of $\gamma$, satisfies conditions \eqref{CondII} and \eqref{CondI}.
\end{proof}


\basicpowerresultBinomial*

\begin{proof}

We show that $\underline{\gamma}^{\hat{\lambda}_{bayes}}=-1$, from which it immediately follows that $\underline \gamma^{oracle}=-1$. Since $A_0^{\rho^*}(1/2)=\lambda_N$ and $A_1^{\rho^*}(1/2)=1$, it follows for any $\gamma > -1$, 
\[
\frac{\lambda_N}{\sigma(1/2)} = \frac{\sqrt{m}\lambda_N}{\sqrt{\lambda_N}}  \to \infty, \text{ for } N \to \infty.
\]
By Proposition \eqref{nofastratebinomial} this implies $\underline{\gamma}^{\hat{\lambda}_{bayes}}=-1$.
 

\end{proof}

\begin{proposition}\label{Example23propproof}
For the setting of Example \ref{Example_2}, let $\varepsilon \in (0,1]$ be arbitrary and $p_2 > 0.5$, $p_2 = 0.5 + o(N^{-1})$. Then $\hat{\lambda}^{\rho^*}_{adapt}$ attains the oracle rate $\underline{\gamma}^{\hat{\lambda}_{adapt}} = \underline{\gamma}^{oracle} = -1$, while $\hat{\lambda}^{\rho^*}_{bayes}$ attains the rate $\underline{\gamma}^{\hat{\lambda}_{bayes}}=-1/2$.
\end{proposition}

\begin{proof}
We first find the expression for $\lambda_N$. Since $p_2  > 0.5$
\begin{align*}
    \lambda_N &= \int (f-g) dx = p_1  + \left[(1-p_1) p_2 - (1-p_1) (1-p_2)\right] \int  f_0 dx \\
    &=p_1  +(1-p_1) \left[2 p_2 -  1\right]
\end{align*}
and, since $p_2 - 1/2 = o(N^{-1})$, it immediately holds that $p_1 \asymp \lambda_N$. Let $\gamma > -1$ be arbitrary and take $t_N=0$ for all $N$. Then $\lambda(t_N)=p_1$ and it holds that
\begin{align*}
    \frac{\lambda_N}{ \lambda(t_N)} =  \frac{p_1  +(1-p_1) \left[2 p_2 -  1\right]}{p_1} = 1 + \frac{(1-p_1) \left[2 p_2 -  1\right]}{p_1} \to 1,
\end{align*}
as $2 p_2 -  1=o(N^{-1})$ by assumption. Combining this with the fact that $A^{\rho^*}_0(0)=p_1\asymp \lambda_N$, and $A^{\rho^*}_1(0)=1$ thus
\[
\frac{\lambda_N}{\sigma(t)} =\frac{\sqrt{m}\lambda_N}{\sqrt{ A^{\rho^*}_0(0)(1-A^{\rho^*}_0(0)) }}\asymp \sqrt{m\lambda_N} \to \infty,
\]
it follows that $\underline{\gamma}^{\text{oracle}}=-1$ and therefore also $\underline{\gamma}^{\hat{\lambda}_{adapt}}=-1$. On the other hand
\begin{align*}
    A^{\rho^*}_0(1/2)=A^{\rho^*}_1(1/2)= p_1 + (1-p_1) p_2  \to 0.5,
\end{align*}
so \eqref{CondI} cannot be true for any $\gamma \leq -1/2$. From Corollary \ref{nofastratebinomial} it follows that $\hat{\lambda}_{bayes}$ only attains a rate $\underline{\gamma}^{\hat{\lambda}_{bayes}}=-1/2$.

\end{proof}

Proposition \ref{Example2prop} and \ref{Example3prop} then immediately follow from Proposition \ref{Example23propproof}.



We continue with the proofs for Section \ref{estimatedrho}, by quickly restating assumptions (E1) and (E2):
\begin{itemize}
    \item[(E1)] $\hat{\rho}=\hat{\rho}_{N_{tr}}$ is trained on a sample of size $N_{tr}$, $(Z_1, \ell_1), \ldots, (Z_{N_{tr}}, \ell_{N_{tr}})$, and evaluated on an independent sample $(Z_1, \ell_1), \ldots, (Z_{N_{te}}, \ell_{N_{te}})$, with $N_{tr} + N_{te}=N$
    \item[(E2)] $N_{te}, N_{tr} \to \infty$, as $N \to \infty$, with $m_{te}/N_{te} \to \pi \in (0,1)$.
\end{itemize}

Let $\lambda(\hat{\rho})$ be defined as in \eqref{rhoonprojectedspace}:
\begin{align*}
    \lambda(\hat{\rho}) =\sup_{t \in [0,1]} \left[ A_0^{\hat{\rho}}(t) + A_1^{\hat{\rho}}(t) \right] -1 := \sup_{t \in [0,1]} \left[ P(\hat{\rho}(X) \leq t|\hat{\rho} ) - Q(\hat{\rho}(Y) \leq t| \hat{\rho} )\right].
\end{align*}

We first establish that $\hat{\lambda}^{\hat{\rho}}_{adapt}$ is still an asymptotic HPLB.


\HPLBwithestimatedrho*

\begin{proof}
We first note that Propositions \ref{witsearch} and \ref{Qfunctions} hold true also for a sequence $(\rho_N)_{N \in \N}$, instead of just a single arbitrary $\rho$. Thus conditioning on $\hat{\rho}$ trained on independent data, we have in particular that pointwise,
\begin{align*}
     \limsup_{N \to \infty}   \mathbb{P}(\hat{\lambda}^{\hat{\rho}} > \lambda | \hat{\rho } ) \leq \alpha.
\end{align*}
Using Fatou's Lemma \citep{dudley}, it holds that
\begin{align*}
    \limsup_{N \to \infty } \mathbb{P}(\hat{\lambda}^{\hat{\rho}} > \lambda ) &=\limsup_{N \to \infty } \E[ \mathbb{P}(\hat{\lambda}^{\hat{\rho}} > \lambda | \hat{\rho } )]\\
    &\leq  \E[ \limsup_{N \to \infty } \mathbb{P}(\hat{\lambda}^{\hat{\rho}} > \lambda | \hat{\rho } )]\\
    & \leq \alpha,
\end{align*}
proving the result.
\end{proof}



However, for $\hat{\lambda}^{\hat{\rho}}_{bayes}$ (or $\hat{\lambda}^{\rho_N}(1/2)$), we encounter a difficulty when $\lambda=0$. 


\Binomialmightnotconservelevel*

\begin{proof}
It holds that $A_0^{\rho_N}(1/2)=C/n$ and $1-A_1^{\rho_N}(1/2)=C/n$ and,
\begin{align*}
    \Prob(\hat{\lambda}^{\rho_N}(1/2) > 0 ) &= \Prob(n \hat{F}_n^{\rho_N}(1/2) - n \hat{G}_n^{\rho_N}(1/2) >   q_{1-\alpha} n\sigma(1/2) ).
\end{align*}
Define $\xi_{01}=n \hat{F}_n^{\rho_N}(1/2) \sim \mbox{Binomial}(C/n, n)$ and $\xi_{02}=n\hat{G}_n^{\rho_N}(1/2) \sim \mbox{Binomial}(C/n,n) $. Then by the Poisson convergence theorem and due to independence, $\xi_{01} - \xi_{02}$ converges in distribution to $\xi_1-\xi_2$. Additionally, 
\[
n\sigma(1/2) = \sqrt{ n A_0^{\rho_N}(1/2)(1-A_0^{\rho_N}(1/2)) + n A_1^{\rho_N}(1/2)(1-A_1^{\rho_N}(1/2))  } \to \sqrt{2 C},
\]
proving the result.
\end{proof}

For $\varepsilon \in (0,1]$ the goal in the following is to establish that for all subsequences, there exists a further subsequence $N(\ell(k)))$ such that 
\begin{equation}\label{overalreq2}
    \liminf_{k \to \infty} \mathbb{P}(\hat{\lambda} > (1-\varepsilon) \lambda_{N(\ell(k))} | \hat{\rho}_{N_{tr}(\ell(k))} ) =  1, \text{ a.s.}
\end{equation}
This suggests that for a given $\varepsilon$ we need to check the following adapted conditions on $(t_{N_{te}})_{N_{te} \geq 1}$: For any subsequence $N(\ell)$, we find a further subsequence $N(\ell(k))$, such that
\begin{align}\label{CondIId}
     \liminf_{k \to \infty} \mathcal{T}_{N(\ell(k))}:= \liminf_{k \to \infty} \frac{\lambda^{\hat{\rho}_{N_{tr}(\ell(k))}}(t_{N_{te}(\ell(k))})}{\lambda_{N(\ell(k))}}  & \geq  1-\varepsilon \text{ a.s.}  ,
\end{align}
and
\begin{subequations}\label{CondId}
\begin{align}
  \lim_{k \to \infty} \frac{\lambda_{N(\ell(k))}}{\sigma^{\hat{\rho}_{N_{tr}(\ell(k))}}(t_{N_{te}(\ell(k)))})} &= \infty \text{ a.s.} , \text{ if }    \liminf_{k \to \infty} \mathcal{T}_{N(\ell(k))} >  1-\varepsilon  \text{ a.s.}  \label{CondI.1d}, \\ 
         \lim_{k \to \infty } \frac{\lambda_{N(\ell(k))}}{\sigma^{\hat{\rho}_{N_{tr}(\ell(k))}}(t_{N_{te}(\ell(k)))})} \left( \mathcal{T}_{N(\ell(k))} -( 1-\varepsilon) \right) &= \infty \text{ a.s.} , \text{ if }  \liminf_{k \to \infty} \mathcal{T}_{N(\ell(k))} =  1-\varepsilon  \text{ a.s.}   \label{CondI.2d},
       \end{align}
\end{subequations}
where for $F^{\hat{\rho}}(t):=\Prob(\hat{\rho}(X) \leq t | \hat{\rho} )$ and $G^{\hat{\rho}}(t):=\Prob(\hat{\rho}(Y) \leq t | \hat{\rho} )$
\begin{align*}
    \lambda^{\hat{\rho}}(t)&=  F^{\hat{\rho}}(t)  -  G^{\hat{\rho}}(t) \\
    \lambda(\hat{\rho})&= \sup_{t \in [0,1]} \lambda^{\hat{\rho}}(t)\\
    \sigma^{\hat{\rho}}(t) &=\left( \frac{F^{\hat{\rho}}(t)(1-F^{\hat{\rho}}(t))}{m_{te}} + \frac{G^{\hat{\rho}}(t)(1-G^{\hat{\rho}}(t) )}{n_{te} }\right)^{1/2}.
\end{align*}

We now generalize Propositions \ref{oracleprop} and \ref{newamazingresult} to this case:



\begin{proposition}\label{newamazingresultestimatedrho}
Let $-1 < \gamma \leq 0$ and $\varepsilon_1 \in (0,1]$ fixed. Assume that $\lambda_N=N_{te}^{\gamma}$ and that (E1) and (E2) hold. Then the following is equivalent
\begin{itemize}
    \item[(i)] there exists a $(t_{N_{te}})_{N_{te} \geq1}$ such  that \eqref{CondIId} and \eqref{CondId} are true for $\varepsilon$,
    \item[(ii)] \eqref{overalreq} is true for $\hat{\lambda}^{\hat{\rho}}(t_N)$,
    \item[(iii)] \eqref{overalreq} is true for $\hat{\lambda}^{\hat{\rho}}_{adapt}$.
\end{itemize}

\end{proposition}

 \begin{proof}
 The same arguments as in Proposition \ref{oracleprop} and \ref{newamazingresult} show that for a (nonrandom) sequence $\rho_N$ and $(t_N)_{N \geq 1}$:
\begin{align}
     \liminf_{N \to \infty} \frac{\lambda^{\rho_N}(t_N)}{\lambda_N}  & \geq  1-\varepsilon \label{CondIIseq},
     \end{align}
and
\begin{subequations}\label{CondIseq}
\begin{align}
  \lim_{N } \frac{\lambda_N}{\sigma^{\rho_N}(t_N)} &= \infty, \text{ if }  \liminf_{N \to \infty} \frac{\lambda^{\rho_N}(t_N)}{\lambda_N} >  1-\varepsilon  \label{CondI.1seq}, \\ 
         \lim_{N } \frac{\lambda_N}{\sigma^{\rho_N}(t_N)} \left( \frac{\lambda^{\rho_N}(t_N)}{\lambda} -( 1-\varepsilon) \right) &= \infty, \text{ if }  \liminf_{N \to \infty} \frac{\lambda^{\rho_N}(t_N)}{\lambda_N} =  1-\varepsilon   \label{CondI.2seq},
       \end{align}
\end{subequations}
 if and only if
\begin{align*}
        \liminf_{k \to \infty} \mathbb{P}(\hat{\lambda}^{\rho_N}(t_N) > (1-\varepsilon) \lambda ) =  1,
\end{align*}
and
\begin{align*}
        \liminf_{k \to \infty} \mathbb{P}(\hat{\lambda}^{\rho_N}_{adapt} > (1-\varepsilon) \lambda ) =  1.
\end{align*}

Through conditioning, we now extend this to $\hat{\rho}$. The arguments are the same for $\hat{\lambda}^{\rho_N}_{adapt}$ and $\hat{\lambda}^{\rho_N}(t_N)$ and thus we will write $\hat{\lambda}$ to mean either of them.

First assume \eqref{CondIId} and \eqref{CondId} are true for an $\varepsilon \in (0,1]$, $\hat{\rho}$ and sequence $(t_N)_{N}$. Considering only the chosen subsequence $N(\ell(k))$ and conditioning on $(\hat{\rho}_{N(\ell(k))})_k$, this gives a sequence $\rho_k=\hat{\rho}_{N(\ell(k))}$ such that \eqref{CondIIseq} and \eqref{CondIseq} are true and by the above this means \eqref{overalreq2} holds. Since $\mathbb{P}(\hat{\lambda} > (1-\varepsilon) \lambda_{N} | \hat{\rho}_{N_{tr}} )$ is bounded, we can use Fatous lemma to obtain, that every subsequence has a further subsequence with 
\[
\liminf_{k \to \infty} \mathbb{P}(\hat{\lambda} > (1-\varepsilon) \lambda_{N(\ell(k))}) =  1.
\]
An argument by contradiction shows that then the liminf of the overall sequence must be 1 as well. Indeed assume that this is not true. Then we can find a subsequence $N(\ell)$ such that 
\[
\lim_{\ell \to \infty} \mathbb{P}(\hat{\lambda} > (1-\varepsilon) \lambda_{N(\ell)} ) =  c < 1.
\]
But then any further subsequence will have limsup strictly below $1$, contradicting the above. 


Now assume $\eqref{overalreq}$ is true. Then, by definition, $\Prob \left( \hat{\lambda} > (1-\varepsilon)  \lambda_N \right) \to 1.$ But this is also true for any subsequence and thus \eqref{overalreq2} must also hold. Indeed this simply follows from the fact that
\begin{equation}\label{usefulfact}
    \lim_{k \to \infty} \int |f_k| d \Prob = 0 \implies \lim_{k \to \infty} |f_k| = 0 \text{ a.s.}, 
\end{equation}
applied to $f_k=1-\mathbb{P}(\hat{\lambda} > (1-\varepsilon) \lambda_{N(\ell(k))} | \hat{\rho}_{N_{tr}(\ell(k))} ) \geq 0$. We quickly prove \eqref{usefulfact} for completeness below.
But with that, by the same arguments as above (connecting to a nonrandom sequence $\rho_k$), \eqref{overalreq2} implies \eqref{CondIId} and \eqref{CondId}.

It remains to prove \eqref{usefulfact}. To do so, assume there exists a set $B$, with $\Prob(B) > 0$, such that $\liminf_{k \to \infty} |f_k| > 0$ on $B$. Then again using Fatou's lemma,
\begin{align*}
     \liminf_{k \to \infty} \int |f_k| d \Prob &\geq  \int \liminf_{k \to \infty} |f_k| d \Prob\\
     & \geq \int_{B} \liminf_{k \to \infty} |f_k| d \Prob + \int_{B^c} \liminf_{k \to \infty} |f_k| d \Prob\\
     & > 0,
\end{align*}
since $g > 0$ implies that $\int g d \Prob > 0$. Thus $\liminf_{k \to \infty} \int |f_k| d \Prob > 0$, proving \eqref{usefulfact} by contraposition.

 \end{proof}

Again, Proposition \ref{newamazingresultestimatedrho} would be still valid, if $\lambda$ was replaced everywhere by $\lambda(\hat{\rho})$, assuming that $\lambda(\hat{\rho})$ converges to a limit $\lambda(\rho) \in [0,1]$ in probability. For instance, if $\lambda(\hat{\rho}) \stackrel{p}{\to} 0$, at a rate $N^{\gamma}$, $-1 < \gamma < 0$.

Proposition \ref{estimatedrhoresult} then follows directly from Proposition \ref{newamazingresultestimatedrho}.


\estimatedrhoresulttwo*

\begin{proof}
Due to consistency, it holds for all $\varepsilon \in (0,1]$ that there exists for each subsequence a further subsequence, such that 
\begin{align*}
     \liminf_{k \to \infty}  \frac{\lambda^{\hat{\rho}_{N_{tr}(\ell(k))}}(t_{N_{te}(\ell(k))})}{\lambda}  & \geq  1-\varepsilon \text{ a.s.}  ,
\end{align*}
for $t_{N_{te}(\ell(k))}:=t_{N_{tr}(\ell(k))}$. Thus for the sequence $t_{N_{tr}}$ and all $\varepsilon \in (0,1]$, \eqref{CondIId} is true. Moreover, since $\lambda$ is fixed here, \eqref{CondId} is clearly also true, proving the result.
\end{proof}

\subsection{Technical Results} \label{techlemmas}

\begin{lemma}\label{inftylemma}
Let $p \in [0,1]$, $\alpha \in (0,1)$ with $1-\alpha > 0.5$ and $p_{\varepsilon}:=(1-\varepsilon)p$. Then $m  p - q_{1-\alpha}(p_{\varepsilon},m) \asymp m p \varepsilon$. 
More generally, if $p=p_m \asymp m^{\gamma}$, $-1 < \gamma < 0$, and $p_{\varepsilon}:=(1-\varepsilon) p_m$, then
$m  p_m - q_{1-\alpha}(p_{\varepsilon},m) \asymp m p_m \varepsilon$.
\end{lemma}

\begin{proof}
Let $p=\delta_m \asymp m^{\gamma}$, for $-1 <  \gamma \leq 0 $, where $\gamma=0$ indicates the fixed $p$ case. Writing $q_{1-\alpha}(p_{\varepsilon},m)=q_{1-\alpha}(\Lambda)$, where $\Lambda \sim \mbox{Binomial}(p_{\varepsilon},m)$, it holds that
\[
\frac{q_{1-\alpha}(\Lambda)-m p_{\varepsilon}}{\sqrt{m p_{\varepsilon} (1-p_{\varepsilon})}} = q_{1-\alpha}(Z_m),
\]
where $Z_m:=(\Lambda -mp_{\varepsilon})/\sqrt{m p_{\varepsilon} (1-p_{\varepsilon}) }$ and $q_{1-\alpha}(Z_m)$ is the $1-\alpha$ quantile of the distribution of $Z_m$. By the Lindenberg-Feller central limit theorem, $Z_m$ converges in distribution to $\mathcal{N}(0,1)$ and is thus uniformly tight, i.e. $Z_m=O_{\mathbb{P}}(1)$. Consequently, it must hold that
\[
0 < \frac{q_{1-\alpha}(\Lambda)-mp_{\varepsilon}}{\sqrt{mp_{\varepsilon} (1-p_{\varepsilon})}} = q_{1-\alpha}(Z_m) \asymp 1,
\]
which means $q_{1-\alpha}(\Lambda) - mp_{\varepsilon} \asymp \sqrt{mp_{\varepsilon} (1-p_{\varepsilon})}$. Writing
\[
\Delta_m := m  p_m - q_{1-\alpha}(p_{\varepsilon},m) =  m  p_m - mp_{\varepsilon} - (q_{1-\alpha}(p_{\varepsilon},m) - mp_{\varepsilon}) = m p_m \varepsilon - (q_{1-\alpha}(\Lambda) - mp_{\varepsilon}),
\]
we see that $\Delta_m \asymp  m p_m \varepsilon$.
\end{proof}

As we do not constrain the possible alternatives $P,Q$ and sequences $(t_N)_{N \geq 1}$, some proofs have several cases to consider. In an effort to increase readability we will summarize these different cases here for reference:
We first introduce a ``nuisance condition''. This condition arises when $(t_N)_{N \leq 1}$ or the sequence of alternatives is such that the variance $\sigma(t_N)$ converges to zero fast, namely if
\begin{align}\label{NC}
    \liminf_{N \to \infty} N \sigma(t_N) < +\infty. \tag{\bf{NC}}
\end{align}

The case in which we are mainly interested is however is the negation of \eqref{NC},
\begin{align}\label{MC}
    \lim_{N \to \infty} N \sigma(t_N) = +\infty \tag{\bf{MC}}
\end{align}
A special case of that is the following 
\begin{align}\label{MCE}
   \text{either } F(t_N)(1-F(t_N))=0 \text{ or } G(t_N)(1-G(t_N))=0 \text{ for infinitely many $N$} \tag{\bf{MCE}}.
\end{align}

We first show an important limiting result, in the case \eqref{MC}, on which much of our results are based:

\begin{lemma} \label{Gaussianresult}
Let $-1 < \gamma \leq 0$, where $\gamma=0$ constitutes the constant case $\lambda_N=\lambda$. Then for any $\rho: \X \to I$ and any sequence $(t_N)_{N\geq 1} \subset I$ such that \eqref{MC} holds,
\begin{equation}\label{convergenceindist}
    \mathcal{Z}_N:=\frac{\hat{F}_m(t_N) - \hat{G}_n(t_N) - (F(t_N) - G(t_N))}{\sigma(t_N)} \stackrel{D}{\to} \mathcal{N}(0,1).
\end{equation}
\end{lemma}

\begin{proof}
Let
\[
\sigma_F:=\sqrt{ \frac{F(t_N)(1-F(t_N))}{m}} \hspace{0.5cm} \text{and} \hspace{0.5cm}
\sigma_G:=\sqrt{ \frac{G(t_N)(1-G(t_N))}{m}},
\]
so that we may write $\sigma(t_N)=\sqrt{ \sigma_F^2 +  \sigma_G^2 }$. From \eqref{MC} we require $m \sigma_F \to \infty$ or $n \sigma_G \to \infty$. 

By the Lindenberg-Feller CLT (see e.g., \cite{vaart_1998}), it holds for $N \to \infty$ (and thus $m,n \to \infty$),
\begin{align*}
&\frac{1}{\sigma_F}(\hat{F}_m(t_N)  - F(t_N)) \stackrel{D}{\to} \mathcal{N}(0,1), \text{ if } m \sigma_F \to \infty \\ 
&\frac{1}{\sigma_G}(\hat{G}_m(t_N)  - G(t_N))  \stackrel{D}{\to} \mathcal{N}(0,1), \text{ if } n \sigma_G \to \infty 
\end{align*}


We write
\begin{align*}
   \mathcal{Z}_N &= \frac{\hat{F}_m(t_N) - F(t_N) - (\hat{G}_n(t_N) - G(t_N))}{\sigma(t_N)} \\
  &= \frac{\hat{F}_m(t_N) - F(t_N)} {\sigma(t_N)}  - \frac{(\hat{G}_n(t_N) - G(t_N))}{\sigma(t_N)} \\
   &= \frac{\hat{F}_m(t_N) - F(t_N)} {\sigma_F} \frac{\sigma_F}{\sigma(t_N)}  - \frac{(\hat{G}_n(t_N) - G(t_N))}{\sigma_G} \frac{\sigma_G}{\sigma(t_N)}.
\end{align*}

\noindent Now as,
\begin{align*}
    \frac{\sigma_F}{\sigma(t_N)}  = \sqrt{\frac{\sigma_F^2}{\sigma_F^2 + \sigma_G^2}}, \text{ and }  \frac{\sigma_G}{\sigma(t_N)}= \sqrt{\frac{\sigma_G^2}{\sigma_F^2 + \sigma_G^2}}, 
\end{align*}
we can define $\omega_N:=\sigma_F/\sigma(t_N)$, so that
\begin{align*}
   \mathcal{Z}_N  &= \frac{\hat{F}_m(t_N) - F(t_N)} {\sigma_F}  \omega_N  - \frac{(\hat{G}_n(t_N) - G(t_N))}{\sigma_G} \sqrt{1 - \omega_N^2}.
\end{align*}
Had $\omega_N$ a limit, say $\lim_N \omega_N:=a \in [0,1]$ and if both $m \sigma_F \to \infty$ and  $m \sigma_G \to \infty$ were true, it would immediately follow from classical results (see e.g., \cite[Chapter 2]{vaart_1998}) that 
$\mathcal{Z}_N \stackrel{D}{\to} \mathcal{N}(0,1)$. This is not the case as the limit of $\omega_N$ might not exist and either $\limsup_{N \to \infty} m \sigma_F < \infty$ or $\limsup_{N \to \infty} n \sigma_G < \infty$. However since $\omega_N \in [0,1]$ for all $N$, it possesses a subsequence with a limit in $[0,1]$. More generally, every subsequence $(\omega_{N(k)})_k$ possesses a further subsequence $(\omega_{N(k(\ell))})_\ell$ that converges to a limit $a \in [0,1]$. This limit depends on the specific subsequence, but for any such converging subsequence it still holds as above that $\mathcal{Z}_{N(k(\ell))} \stackrel{D}{\to} \mathcal{N}(0,1)$. Indeed, if both $m \sigma_F \to \infty$ and $n \sigma_G \to \infty$ this is immediate from the above. If, on the other hand, $\liminf_{N \to \infty} m \sigma_F < \infty$, then $[\hat{F}_m(t_{N(k(\ell))}) - F(t_{N(k(\ell))})]/\sigma_F \stackrel{D}{\to} \mathcal{N}(0,1)$ might not be true. However, if we assume that \eqref{MCE} does not hold, for the chosen subsequence 
\[
\frac{\sigma_F^2}{\sigma(t_{N(k(\ell))})^2} = \omega_{N(k(\ell))}^2  \to a^2 \in [0,1], 
\]
and it either holds that $a=0$ in which case the first part of $\mathcal{Z}_{_{N(k(\ell))}}$ is negligible or $a > 0$, in which case it must hold that
$\sigma_F \asymp \sigma(t_{N(k(\ell))})$ and thus $N(k(\ell)) \sigma_F \to \infty$, allowing for $[\hat{F}_m(t_{N(k(\ell))}) - F(t_{N(k(\ell))})]/\sigma_F \stackrel{D}{\to} \mathcal{N}(0,1)$. The symmetric argument applies if $\liminf_{N \to \infty} m \sigma_G < \infty$. Now assume \eqref{MCE} holds and for a given subsequence $(t_{k})_k$, it is not possible to find a subsequence, such that $ F(t_{k(\ell)})(1-F(t_{k(\ell)}))> 0 $ for all but finitely many $\ell$. Since for all subsequences $k(\ell)\sigma(t_{k(\ell)}) \to \infty$, it must hold that $n(k(\ell))\sigma_G \to \infty$. In particular, we may choose the subsequence such that $ F(t_{k(\ell)})(1-F(t_{k(\ell)}))= 0 $ for all but finitely many $N$ and in this case:
\[
\hat{F}_m(t_{k(\ell)}) - F(t_{k(\ell)}) = 0 \text{ a.s.  and } \frac{\hat{G}(t_{k(\ell)}) - G(t_{k(\ell)})}{\sigma_G} \stackrel{D}{\to} \mathcal{N}(0,1)
\]
both are true, implying \eqref{convergenceindist}. The symmetric argument holds if instead $G(t_N)(1-G(t_N))=0$ for infinitely many $N$, but $N F(t_N)(1-F(t_N))\to \infty$. 

Thus we have shown that for \emph{any} subsequence of $\mathcal{Z}_N$, there exists a further subsequence converging in distribution to $\mathcal{N}(0,1)$. Assume that despite this, \eqref{convergenceindist} is not true. Then, negating convergence in distribution in this particular instance, means there exists $z \in \R$ such that the cumulative distribution function of $\mathcal{Z}_N$, $F_{\mathcal{Z}_N}$, has $\limsup_{N \to \infty} F_{\mathcal{Z}_N}(z) \neq \Phi(z) $. By the properties of the limsup, there exists a subsequence $\lim_{k \to \infty} F_{\mathcal{Z}_{N(k)}}(z) = \limsup_{N \to \infty} F_{\mathcal{Z}_N}(z) \neq \Phi(z)$. But then no further subsequence of $F_{\mathcal{Z}_{N(k)}}(z)$ converges to  $\Phi(z)$, a contradiction.
\end{proof}

The next lemma ensures that we can for all intents and purposes ignore sequences $(t_N)_{N \geq 1}$ for which \eqref{NC} is true.

\begin{lemma}\label{ignorebadcase}
Let $-1 < \gamma < 0 $ and $\varepsilon \in (0,1]$ arbitrary. If for a sequence $(t_N)_{N\geq1}$ and $\rho=\rho^*$, \eqref{NC} holds, then
\begin{itemize}
    \item[(I)] \eqref{CondII} or \eqref{CondI} is not true,
     \item[(II)] \eqref{overalreq} is not true for $\lambda=\hat{\lambda}^{\rho^*}(t_N)$.
\end{itemize}
Furthermore, if for $-1 < \gamma < 0 $, a sequence $(t_N)_{N\geq1}$ and $\rho: \X \to I$, \eqref{NC} holds then, 
\begin{itemize}
    \item[(III)] $\limsup_{N \to \infty}  \Prob \left( \hat{\lambda}^{\rho}(t_N)  > \lambda \right) = 0$.
\end{itemize}
(III) is also true for the constant case, $\gamma=0$, as long as $\lambda > 0$.
\end{lemma}

\begin{proof}
(I): If \eqref{NC} is true, then $N^{\beta} \sigma(t_N) \to 0 $, for any $\beta \in [0,1)$. Indeed, assume there exists $\beta \in [0,1)$ such that $\liminf_{N \to \infty} N^{\beta} \sigma(t_N)  > 0 $. Then
\[
\liminf_{N \to \infty} N \sigma(t_N) \geq  +\infty, 
\]
In particular, it must hold that
\begin{align*}
    F(t_N)(1-F(t_N)) = o(N^{\zeta}) \text{ and } G(t_N)(1-G(t_N)) = o(N^{\zeta}),
\end{align*}
for all $\zeta \in [-1,0)$. There are four possibilities for this to be true:
\begin{itemize}
    \item[(1)]  $ F(t_N) = o(N^{\zeta})$, $  G(t_N) = o(N^{\zeta})$. 
    \item[(2)] $F(t_N) = o(N^{\zeta})$, $1- G(t_N) = o(N^{\zeta})$. 
    \item[(3)] $ (1- F(t_N)) = o(N^{\zeta})$, $(1- G(t_N)) =  o(N^{\zeta})$. 
    \item[(4)] $(1- F(t_N)) = o(N^{\zeta})$, $ G(t_N) = o(N^{\zeta})$. 
\end{itemize}

As (2) and (4) imply that $\lambda(t_N) \to -1$ and $\lambda_N \geq \lambda(t_N) \to 1$ respectively, they are not relevant in our framework.
Thus \eqref{NC} directly implies that either (1) or (3) is true and both of them imply $\lambda(t_N)=o(N^{\zeta})$ for all $\zeta \in (-1,0)$. Consequently, \eqref{CondII} cannot be true for $\varepsilon < 1$.

For $\varepsilon=1$ we slightly strengthen the relevant cases (1) and (3):
\begin{itemize}
    \item[(1')]$ \liminf_{N \to \infty} N A_0(t_N) < \infty$, $ \liminf_{N \to \infty} N (1-A_1(t_N)) < \infty$,
    \item[(3')] $  \liminf_{N \to \infty} N (1- A_0(t_N)) < \infty$, $ \liminf_{N \to \infty} N A_1(t_N) < \infty$.
\end{itemize}
It's clear from the above, that if \eqref{NC} holds, then one of the two has to hold. We will now show that, even though \eqref{CondII} is true in this case, \eqref{CondI.2} is not. Indeed, it was mentioned in Section \ref{powersec} that in case of \eqref{MCE}, \eqref{CondI.2} is defined to be false. Thus, we may assume $\sigma(t_N)$ is bounded away from zero for all finite $N$. Assume that \eqref{CondI.2} is true, i.e.
\[
\frac{\lambda(t_N)}{\sigma(t_N)} \to \infty.
\]
This must then also hold for any subsequence. Now assume (1') is true, and choose a subsequence $(t_k)_{k \in \N}:=(t_{N(k)})_{k \in \N}$ with $\lim_{k \to \infty} k A_0(t_k)=\liminf_{N \to \infty} N A_0(t_N):=a \in [0,\infty)$. Then
\[
\frac{\lambda(t_k)}{\sigma(t_k)} \leq \frac{A_0(t_k)}{\sigma(t_k)} \leq \frac{\sqrt{m(k)} A_0(t_k)}{\sqrt{A_0(t_k)(1-A_0(t_k))} } \asymp \sqrt{k A_0(t_k)}  \to \sqrt{a} < \infty,
\]
a contradiction. Similarly, if (2') is true, we find a subsequence such that $\lim_{k \to \infty} k A_1(t_k):=a \in [0, \infty)$ and bound
\[
\frac{\lambda(t_k)}{\sigma(t_k)} \leq \frac{A_1(t_k)}{\sigma(t_k)} \leq \frac{\sqrt{n(k)} A_1(t_k)}{\sqrt{A_1(t_k)(1-A_1(t_k))} } \asymp \sqrt{k A_1(t_k)}  \to \sqrt{a} < \infty.
\]
Thus \eqref{CondI.2} cannot be true.

(II) and (III): 
Consider first $\varepsilon \in [0,1)$ and $-1 < \gamma \leq 0 $. \eqref{NC} implies for any $\rho$:
\begin{equation}
    \lambda_N^{-1} (\hat{F}(t_N) - \hat{G}(t_N) - \lambda(t_N))  \stackrel{p}{\to} 0.
\end{equation}
Indeed by a simple Markov inequality argument for all $\delta > 0$:
\begin{align*}
    \Prob \left(  \lambda_N^{-1} (\hat{F}(t_N) - \hat{G}(t_N) - \lambda(t_N))  >  \delta  \right) \leq \frac{\lambda_N^{-2} \sigma(t_N)^2}{\delta} \asymp \frac{(N^{-\gamma} \sigma(t_N))^2}{\delta} \to 0,
\end{align*}
since $-\gamma \in [0,1)$. Additionally, from the argument in (I), $\lambda_N^{-1}\sigma(t_N)  \to 0$ and $\frac{\lambda(t_N)}{\lambda_N} \to 0 $. Consequently, for any $\varepsilon \in [0,1)$
\begin{align*}
        &\mathbb{P}(\hat{F}(t_N) - \hat{G}(t_N) - q_{1-\alpha} \sigma(t_N) > (1-\varepsilon) \lambda_N) \\
         &=\mathbb{P}(\hat{F}(t_N) - \hat{G}(t_N) - \lambda(t_N) - q_{1-\alpha} \sigma(t_N) > (1-\varepsilon) \lambda_N  - \lambda(t_N)) \\
         &=\mathbb{P}(\lambda_N^{-1} (\hat{F}(t_N) - \hat{G}(t_N) - \lambda(t_N)) - q_{1-\alpha} \lambda_N^{-1}\sigma(t_N) + \frac{\lambda(t_N)}{\lambda_N}  > (1-\varepsilon)  )\\
         & \to 0.
\end{align*}
Consequently, \eqref{overalreq} is false for any $\varepsilon \in [0,1)$ and (II) and (III) hold. The case $\varepsilon=1$ needs special care: Assume that despite \eqref{NC}, $ \Prob(\hat{\lambda}(t_N)  > 0 ) \to 1$ holds true. We consider the two possible cases (1') and (3') in turn: If (1') is true, we write:
\begin{align}
    \Prob(\hat{\lambda}(t_N)  > 0 ) \leq  \Prob(\hat{F}(t_N)  > 0 ) = \Prob( m \hat{F}(t_N)  > 0 ) = \Prob(V_{m,t_N}  > 0 ),
\end{align}
where $V_{m,t_N}:= m \hat{F}(t_N) \sim \mbox{Binomial}(A_0(t_N), m)$. Since $ \Prob(\hat{\lambda}(t_N)  > 0 ) \to 1$, this is true for any subsequence $\hat{\lambda}(t_{N(k)})$ as well. In particular, we may choose the subsequence $(t_{N(k)})_{k \geq 1}$ with $\lim_{k} N(k) A_0(t_{N(k)})=\liminf_{N \to \infty} N A_0(t_N) := a \in [0, \infty)$. Renaming the subsequence $(k A_0(t_{k}))_{k \geq 1}$ for simplicity, we find $\limsup_{k \to \infty }  k A_0(t_k) \leq a$, or $A_0(t_k) = O(k^{-1})=O(m(k)^{-1})$, since by assumption $m(k)/k \to \pi \in (0,1)$. But then
\begin{align*}
    \Prob(\hat{\lambda}(t_k)  > 0 ) \leq \Prob(V_{m(k),t_k}  > 0 ) = 1- (1-A_0(t_k))^{m(k)}
\end{align*}
and 
\[
\liminf_{k \to \infty} (1-A_0(t_k))^{m(k)} \geq \liminf_{k \to \infty}  (1-\frac{a}{m(k)})^{m(k)} = \exp(-a) > 0. 
\]
Thus, $\limsup_{k \to \infty } \Prob(\hat{\lambda}(t_k)  > 0 ) < 1$, a contradiction.

If (3') is true, then $\liminf_{N \to \infty} N A_1(t_N) < +\infty$ and similar arguments applied to
\begin{align}
    \Prob(\hat{\lambda}(t_N)  > 0 ) \leq  \Prob(1 - \hat{G}(t_N)  > 0 ) = \Prob( n -  n\hat{G}(t_N)  > 0 ) = \Prob(V_{n,t_N}  > 0 ),
\end{align}
where now $V_{n,t_N}:= \sum_{i=1}^n \Ind{\{ \rho(Y_i) > t_N \}} \sim \mbox{Binomial}(A_1(t_N), n)$, give
\begin{align*}
    \limsup_{k \to \infty } \Prob(\hat{\lambda}(t_k)  > 0 ) \leq \limsup_{k \to \infty} \Prob(V_{n(k),t_k}  > 0 ) = 1- \exp(-a) < 1.
\end{align*}
This again contradicts $ \Prob(\hat{\lambda}(t_N)  > 0 ) \to 1$.

\end{proof}

Lemma \ref{ignorebadcase} also immediately implies for $\hat{\lambda}(t_N)$ that if \eqref{overalreq} or \eqref{CondII} are true, then $\lim_{N} N \sigma(t_N) = + \infty$ must hold.

\begin{lemma}\label{simplified}
Let $-1 < \gamma < 0$ be fixed and as above $\lambda_N \asymp N^{\gamma}$. Define for $(t_N)_{N \geq 1}$, $z(t_N)$ as in \eqref{zdef} and $Q_{m,n, \alpha}$, $\tilde{Q}$ as in \eqref{Qmndef} and \eqref{Qtildedef}. Assume that for the given $\gamma$,
\begin{equation} \label{A0A1def}
    A_0(t_N)(1-A_0(t_N)) = o(N^{2\gamma + 1}) \text{ and } A_1(t_N)(1-A_1(t_N)) = o(N^{2\gamma + 1}),
\end{equation}
then for $c \in (0,+\infty)$,
\begin{align}\label{QQbounding}
   c + o_{\Prob}(1) \leq \frac{ Q_{m,n, \alpha}(z(t_N), \lambda_{\varepsilon}) - m A_0 }{\tilde{Q}(\lambda_\varepsilon) - m A_0  }  \leq   \frac{\sup_{\tilde{\lambda}} Q_{m,n, \alpha}(z(t_N), \tilde{\lambda}) - m A_0 }{\tilde{Q}(\lambda_\varepsilon) - m A_0  } \leq  \frac{1}{c} + o_{\Prob}(1).
\end{align}

\end{lemma}

\begin{proof}
Note that $z(t_N)$ is random, while everything else is deterministic.
First,
\begin{align*}
 &q_{1-\frac{\alpha}{3}}(\lambda_{\varepsilon}, m) \frac{n(\lambda_{\varepsilon})}{N(\lambda_{\varepsilon})}  + \frac{z(t_N)}{N(\lambda_{\varepsilon})} m(\lambda_{\varepsilon}) + \beta_{\frac{\alpha}{3}, m(\lambda_{\varepsilon})}\sqrt{ \frac{m(\lambda_{\varepsilon})}{N(\lambda_{\varepsilon})} \frac{n(\lambda_{\varepsilon})}{N(\lambda_{\varepsilon})} \frac{N(\lambda_{\varepsilon}) - z(t_N)}{N(\lambda_{\varepsilon}) -1 } (z(t_N) - q_{1-\frac{\alpha}{3}}(\lambda_{\varepsilon}, m))  } \\
   &\leq \sup_{\tilde{\lambda} \in [0, \lambda_{\varepsilon}]}\  Q_{m,n, \alpha}(z(t_N), \tilde{\lambda}) \leq \\
&q_{1-\frac{\alpha}{3}}(\lambda_{\varepsilon}  , m)\frac{n}{N(\lambda_{\varepsilon})}+ \frac{z(t_N)}{N(\lambda_{\varepsilon})} m  + \beta_{\frac{\alpha}{3}, m}\sqrt{ \frac{m}{N(\lambda_{\varepsilon})} \frac{n}{N(\lambda_{\varepsilon})} \frac{N - z(t_N)}{N(\lambda_{\varepsilon}) -1 } z(t_N)  }.
\end{align*}

Additionally for all $\tilde{\lambda} \in [0, \lambda_{\varepsilon}]$, with $p_N=[\pi A_0(t_N)  - (1-\pi) A_1(t_N) + (1-\pi)]$, 
\begin{align}
     &\frac{m(\tilde{\lambda})}{N(\tilde{\lambda})} \to \pi, \label{m/N} \\
     &\frac{n(\tilde{\lambda})}{N(\tilde{\lambda})} \to 1-\pi, \label{n/N}\\
      &\frac{q_{1-\frac{\alpha}{3}}(\tilde{\lambda}  , m)}{m \tilde{\lambda}} \to 1 \label{q/m}\\
     & \frac{ \frac{z(t_N)}{N(\lambda_{\varepsilon})} - A_0(t_N) }{ p_N - A_0(t_N) }\stackrel{p}{\to} 1. \label{z/m}
\end{align}

The first three assertions follow from Lemma \ref{inftylemma} and the assumption that $m/N \to \pi$, as $N \to \infty$. We quickly verify \eqref{z/m}. Define 
\[
S_N= \frac{ \frac{z(t_N)}{N(\lambda_{\varepsilon})} - A_0(t_N) }{ p_N - A_0(t_N) }.
\]
By Chebyshev's inequality,
\begin{equation}
    \Prob \left( | S_N - \E[S_N] | > \delta \right) \leq \frac{\Var(S_N)}{\delta},
\end{equation}
for all $\delta > 0$. Now, $z(t_N)$ may be written as a sum of independent Bernoulli random variables:
\[
z(t_N) = \sum_{i=1}^N \Ind\{ \rho^*(Z_i) \leq t_N \}  = \sum_{i=1}^m \Ind{ \{ \rho^*(X_i) \leq t_N \}} +  \sum_{j=1}^n \Ind{\{ \rho^*(Y_i) \leq t_N \}},
\]
with $\Ind{ \{ \rho^*(X_i) \leq t_N \}} \sim \mbox{Bernoulli}(A_0(t_N))$ and $\Ind{\{ \rho^*(Y_i) \leq t_N\}} \sim \mbox{Bernoulli}(1-A_1(t_N))$. Then
\begin{align*}
    \Var(S_N) &= \frac{1}{\left(p_N - A_0(t_N) \right)^2} \Var\left( \frac{z(t_N)}{N(\lambda_{\varepsilon})}  \right) \\
    &= \frac{1}{\left(p_N - A_0(t_N) \right)^2 N(\lambda_{\varepsilon})^2} \left[ m A_0(t_N)(1-A_0(t_N)) +  n A_1(t_N)(1-A_1(t_N))\right]\\
    & \asymp \frac{1}{\left(p_N - A_0(t_N) \right)^2 N} \left[  A_0(t_N)(1-A_0(t_N)) +   A_1(t_N)(1-A_1(t_N))\right].
\end{align*}
Now, since (i)  $p_N - A_0(t_N) = - (1-\pi)[ A_0(t_N) +  A_1(t_N) - 1] = - (1-\pi) \lambda(t_N)$ and $\lambda(t_N) \leq \lambda_N \asymp N^{\gamma}$ and (ii) \eqref{A0A1def} holds,
it follows that
\[
\Var(S_N) \asymp \frac{1}{N^{2\gamma + 1}}  o(N^{2\gamma + 1})  \to 0, \text{ for } N \to \infty.
\]
Thus $| S_N - \E[S_N]| \stackrel{p}{\to} 0$. Moreover, it holds that,
\begin{align*}
    \E[S_N] -1 &= \frac{m/N(\lambda_{\varepsilon}) A_0(t_N) - n/N(\lambda_{\varepsilon}) A_1(t_N) +  n/N(\lambda_{\varepsilon}) - A_0(t_N)}{ \pi A_0(t_N)  - (1-\pi) A_1(t_N) + (1-\pi) - A_0(t_N)} -1 \\
    %
&=\frac{[(1- \pi  ) - (1- m/N(\lambda_{\varepsilon})) ] A_0(t_N) - [(1-\pi)  - n/N(\lambda_{\varepsilon})] (1 - A_1(t_N) ) }{  (\pi - 1 ) [A_0(t_N)  - (1- A_1(t_N))  ] }\\
&=\frac{o(1) A_0(t_N) - o(1) (1 - A_1(t_N) ) }{  (\pi - 1 ) [A_0(t_N)  - (1- A_1(t_N))  ] }\\
                 &=\frac{  o(1) [A_0(t_N) -  (1- A_1(t_N))  ]}{ (\pi - 1 ) [A_0(t_N)  - (1- A_1(t_N)) ] } \to 0, \text{ for } N \to \infty.
\end{align*}

Thus, finally $|S_N -1  | \leq |S_N - \E[S_N]| + | \E[S_N] -1 | \stackrel{p}{\to} 0$.

Continuing, let for the following for two random variables index by $N$ $ X_N \preceq Y_N$ mean that
$\Prob(X_N \leq Y_N) \to 1$, as $N \to \infty$. Then,
\begin{align*}
    &N^{-\gamma} m^{-1} \left( \sup_{\tilde{\lambda}} Q_{m,n, \alpha}(z(t_N), \tilde{\lambda}) - m A_0 \right) \leq  \\
    & N^{-\gamma} q_{1-\frac{\alpha}{3}}(\lambda_{\varepsilon}  , m)/m \frac{n}{N(\lambda_{\varepsilon})}+ N^{-\gamma}\frac{z(t_N)}{N(\lambda_{\varepsilon})}   + N^{-\gamma}\frac{\beta_{\frac{\alpha}{3}, m}}{m}\sqrt{ \frac{m}{N(\lambda_{\varepsilon})} \frac{n}{N(\lambda_{\varepsilon})} \frac{N - z(t_N)}{N(\lambda_{\varepsilon}) -1 } \frac{z(t_N)}{m} }  - N^{-\gamma}A_0(t_N)\\
    & = N^{-\gamma} \left[  \frac{q_{1-\frac{\alpha}{3}}(\lambda_{\varepsilon}  , m)/m}{\lambda_{\varepsilon}}  \frac{n}{N(\lambda_{\varepsilon})} \lambda_{\varepsilon} + \frac{z(t_N)}{N(\lambda_{\varepsilon})}  - A_0(t_N)  \right]  + N^{-\gamma}\frac{\beta_{\frac{\alpha}{3}, m}}{m}\sqrt{ \frac{m}{N(\lambda_{\varepsilon})} \frac{n}{N(\lambda_{\varepsilon})} \frac{N - z(t_N)}{N(\lambda_{\varepsilon}) -1 } \frac{z(t_N)}{m} }  \\
    &= N^{-\gamma} \left[  \frac{q_{1-\frac{\alpha}{3}}(\lambda_{\varepsilon}  , m)/m}{\lambda_{\varepsilon}}  \frac{n}{N(\lambda_{\varepsilon})} \lambda_{\varepsilon} + \frac{z(t_N)/N(\lambda_{\varepsilon}) - A_0(t_N)}{p_N - A_0(t_N)} [p_N - A_0(t_N)]  \right]  +\\
    &N^{-\gamma}\frac{\beta_{\frac{\alpha}{3}, m}}{m}\sqrt{ \frac{m}{N(\lambda_{\varepsilon})} \frac{n}{N(\lambda_{\varepsilon})} \frac{N - z(t_N)}{N(\lambda_{\varepsilon}) -1 } \frac{z(t_N)}{m} }  \\
    %
    %
    %
    & \preceq \Big[  \frac{q_{1-\frac{\alpha}{3}}(\lambda_{\varepsilon}  , m)/m}{\lambda_{\varepsilon}}  \frac{n}{N(\lambda_{\varepsilon})} \frac{(1-\varepsilon) \lambda_N}{\lambda_N} - (1-\pi) \frac{z(t_N)/N(\lambda_{\varepsilon}) - A_0(t_N)}{p_N - A_0(t_N)} \inf_{M \geq N }\frac{\lambda(t_M)}{\lambda_M}\Big] \sup_{M \geq N}  (M^{-\gamma} \lambda_M) \\
    & + N^{-\gamma}\frac{\beta_{\frac{\alpha}{3}, m}}{m}\sqrt{ \frac{m}{N(\lambda_{\varepsilon})} \frac{n}{N(\lambda_{\varepsilon})} \frac{N - z(t_N)}{N(\lambda_{\varepsilon}) -1 } \frac{z(t_N)}{m} }\\
    & \stackrel{p}{\to} d_1 (1-\pi)\left[ (1-\varepsilon) -  d_2 \right] ,
\end{align*}
where $d_1 = \limsup_{N \to \infty} N^{-\gamma}\lambda_N \in (0, \infty)$, $d_2=\liminf_{N \to \infty} \frac{\lambda(t_N)}{\lambda_N} \in ((1-\varepsilon) \lambda_N, 1]$. 

Similarly, 
\begin{align*}
    &N^{-\gamma} m^{-1} \left( \sup_{\tilde{\lambda}} Q_{m,n, \alpha}(z(t_N), \tilde{\lambda}) - m A_0 \right) \geq  \\
    &N^{-\gamma}\left( \frac{q_{1-\frac{\alpha}{3}}(\lambda_{\varepsilon}, m)/m}{\lambda_{\varepsilon}} \frac{n(\lambda_{\varepsilon})}{N(\lambda_{\varepsilon})} \lambda_{\varepsilon}  + \frac{z(t_N)}{N(\lambda_{\varepsilon})} \frac{m(\lambda_{\varepsilon})}{m} + \frac{\beta_{\frac{\alpha}{3}, m(\lambda_{\varepsilon})}}{m}\sqrt{ \frac{m(\lambda_{\varepsilon})}{N(\lambda_{\varepsilon})} \frac{n(\lambda_{\varepsilon})}{N(\lambda_{\varepsilon})} \frac{N(\lambda_{\varepsilon}) - z(t_N)}{N(\lambda_{\varepsilon}) -1 } \frac{z(t_N)-q_{1-\frac{\alpha}{3}}(\lambda_{\varepsilon}, m)}{m}  } - A_0\right)\\
    &=N^{-\gamma}\lambda_N\Big[\frac{q_{1-\frac{\alpha}{3}}(\lambda_{\varepsilon}, m)/m}{\lambda_{\varepsilon}} \frac{n(\lambda_{\varepsilon})}{N(\lambda_{\varepsilon})} \frac{\lambda_{\varepsilon}}{\lambda_N}  + \frac{z(t_N)/N(\lambda_{\varepsilon}) (m(\lambda_{\varepsilon})/m) - A_0(t_N)}{p_N- A_0(t_N)} \frac{[ p_N - A_0]}{\lambda_N}  \Big]\\
    & +N^{-\gamma}\frac{\beta_{\frac{\alpha}{3}, m(\lambda_{\varepsilon})}}{m}\sqrt{ \frac{m(\lambda_{\varepsilon})}{N(\lambda_{\varepsilon})} \frac{n(\lambda_{\varepsilon})}{N(\lambda_{\varepsilon})} \frac{N(\lambda_{\varepsilon}) - z(t_N)}{N(\lambda_{\varepsilon}) -1 } \frac{z(t_N)-q_{1-\frac{\alpha}{3}}(\lambda_{\varepsilon}, m)}{m}  }\\
    &\succeq  \Big[\frac{q_{1-\frac{\alpha}{3}}(\lambda_{\varepsilon}, m)/m}{\lambda_{\varepsilon}} \frac{n(\lambda_{\varepsilon})}{N(\lambda_{\varepsilon})} \frac{\lambda_{\varepsilon}}{\lambda_N}  - \frac{z(t_N)/N(\lambda_{\varepsilon}) (m(\lambda_{\varepsilon})/m) - A_0(t_N)}{p_N- A_0(t_N)} (1-\pi) \sup_{M \geq N}\frac{\lambda(t_M)}{\lambda_M} \Big] \inf_{M \geq N} M^{-\gamma}\lambda_M\\
    & +N^{-\gamma}\frac{\beta_{\frac{\alpha}{3}, m(\lambda_{\varepsilon})}}{m}\sqrt{ \frac{m(\lambda_{\varepsilon})}{N(\lambda_{\varepsilon})} \frac{n(\lambda_{\varepsilon})}{N(\lambda_{\varepsilon})} \frac{N(\lambda_{\varepsilon}) - z(t_N)}{N(\lambda_{\varepsilon}) -1 } \frac{z(t_N)-q_{1-\frac{\alpha}{3}}(\lambda_{\varepsilon}, m)}{m}  }\\
    &  \stackrel{p}{\to} d_3 (1-\pi) \left[ (1-\varepsilon) - d_4 \right],
\end{align*}
where $d_3 = \liminf_{N \to \infty} N^{-\gamma}\lambda_N \in (0, \infty)$,  $d_4=\liminf_{N \to \infty} \frac{\lambda(t_N)}{\lambda_N} \in ((1-\varepsilon) \lambda_N, 1] $. The convergence in probability follows because $m(\lambda_{\varepsilon})/m \asymp 1 - \lambda_{\varepsilon} \to 1 $ and thus using the same proof as for \eqref{z/m}, it holds that 
\[
\frac{z(t_N)/N(\lambda_{\varepsilon}) (m(\lambda_{\varepsilon})/m) - A_0(t_N)}{p_N- A_0(t_N)} \stackrel{p}{\to} 1, \text{ for } N \to \infty.
\]
Additionally,
\begin{align*}
&N^{-\gamma} m^{-1}    (\tilde{Q}(\lambda_{\varepsilon}) - m A_0) \leq  (1-\pi)\left[(1-\varepsilon) - \inf_{M \geq N}\frac{\lambda(t_M)}{\lambda_M} \right] \sup_{M \geq N} (M^{-\gamma}\lambda_M) \\
&\to d_1 (1-\pi) \left[ (1-\varepsilon) - d_2 \right], \text{ for } N \to \infty 
\end{align*}
and
\begin{align*}
&N^{-\gamma} m^{-1}    (\tilde{Q}(\lambda_{\varepsilon}) - m A_0) \geq  (1-\pi) \left[(1-\varepsilon) - \sup_{M \geq N}\frac{\lambda(t_M)}{\lambda_M} \right] \inf_{M \geq N} (M^{-\gamma}\lambda_M) \\
&\to d_3 (1-\pi) \left[ (1-\varepsilon) - d_4 \right], \text{ for } N \to \infty. 
\end{align*}

Thus taking, 
\[
c = \frac{d_3 (1-\pi) \left[ (1-\varepsilon) - d_{4} \right] }{d_1 (1-\pi) \left[ (1-\varepsilon) - d_2 \right] }
\]

we obtain \eqref{QQbounding}.

\end{proof}

\begin{lemma}\label{ignorebadcase2}
Let $-1 < \gamma < 0 $ and $\varepsilon \in (0,1]$ arbitrary and define $\sigma_F^2(t_N)=A_0(t_N)(1-A_0(t_N))/m$. 
If for a sequence $(t_N)_{N\geq1}$ and $\rho=\rho^*$, 
\begin{align}\label{NCd}
   \liminf_{N \to \infty} m \sigma_F < +\infty \tag{\bf{NC'}}
\end{align}
then 
\begin{itemize}
    \item[(I)] \eqref{CondII} or \eqref{CondI} is not true,
     \item[(II)] \eqref{overalreq} is not true for $\lambda=\hat{\lambda}_{adapt}^{\rho^*}$.
\end{itemize}
\end{lemma}

\begin{proof}
First note that \eqref{NCd} implies
\begin{equation}
   \sigma_F= o(N^{\zeta}),
\end{equation}
for all $\zeta \in (-1,0]$, as in Lemma \ref{ignorebadcase}.

\begin{itemize}
    \item[(I):] With the same arguments as in Lemma \ref{ignorebadcase}, \eqref{NCd} implies two possible cases
\begin{itemize}
    \item[(1')] $\liminf_{N \to \infty} N A_0(t_N) < \infty$, $\liminf_{N \to \infty} N (1-A_1(t_N)) < \infty$ 
    \item[(2')] $\liminf_{N \to \infty} N (1-A_0(t_N)) < \infty$, $\liminf_{N \to \infty} N A_1(t_N) < \infty$
\end{itemize}
and these in turn imply
\begin{itemize}
    \item[(1)] $A_0(t_N)=o(N^{\zeta})$, $1 - A_1(t_N)=o(N^{\zeta})$
     \item[(2)] $1 - A_0(t_N)=o(N^{\zeta})$, $A_1(t_N)=o(N^{\zeta})$,
\end{itemize}
for all $\zeta \in (-1,0]$. But then (1) and (2) imply, $\lambda(t_N)=o(N^{\zeta})$, for all $\zeta \in (-1,0]$, contradicting \eqref{CondII} for $\varepsilon < 1$.
For $\varepsilon=1$ assuming \eqref{CondI.2} to be true and following the exact same subsequence argument for (1') and (2') in turn as in Lemma \ref{ignorebadcase} (I) results in a contradiction and thus \eqref{CondI.2} cannot be true.

\item[(II):] Let $V_{m,t_N}$ be defined as in Proposition \ref{newamazingresult}. (1), (2) make it clear that in our setting, \eqref{NCd} and \eqref{A0A1condition} are equivalent. Moreover, in the same way as in Lemma \ref{ignorebadcase}, for all $\delta > 0$
\begin{align*}
    \Prob(\lambda_N^{-1}(\hat{F}_{m}(t_N) - A_{0}(t_N)) > \delta) \leq \frac{\lambda_N^{-2} \sigma_F^2}{\delta} \asymp \frac{[ N^{\gamma} \sigma_F]^2}{\delta} \to 0, \text{ for } N \to \infty.
\end{align*}

Now assume that despite \eqref{NCd},

\begin{equation}\label{whatwewant3}
        \Prob(V_{m,t_N}  > Q_{m,n, \alpha}(z(t_N), \tilde{\lambda} ) \ \  \forall \tilde{\lambda} \in [0, \lambda_{\varepsilon}]  ) \to 1, \text{ for } N \to \infty,
\end{equation}
holds true. Then, using Lemma \ref{simplified} and the arguments in Proposition \ref{newamazingresult}, also
\begin{equation} \label{whatwewant4}
    \Prob(V_{m,t_N}  >  \tilde{Q}(\lambda_{\varepsilon} )  ) \to 1, \text{ for } N \to \infty,
\end{equation}
must hold. However, for $\varepsilon < 1$, 
\begin{align*}
    \Prob(V_{m,t_N}  >  \tilde{Q}(\lambda_{\varepsilon} )  ) &=  \Prob(V_{m,t_N} - m A_{0}(t_N)  >  m \lambda_N [ (1-\varepsilon)  (1-\pi) -  \frac{\lambda(t_N)}{\lambda_N}]   )\\
    &=  \Prob(\lambda_N^{-1}(\hat{F}_m(t_N) -  A_{0}(t_N))  >    [ (1-\varepsilon)  (1-\pi) -  \frac{\lambda(t_N)}{\lambda_N}]   )
\end{align*}
and as from (I), $\frac{\lambda(t_N)}{\lambda_N} \to 0$ and $\lambda_N^{-1}(\hat{F}_m(t_N) -  A_{0}(t_N)) \stackrel{p}{\to} 0 $, this probability will converge to zero, contradicting \eqref{whatwewant3} for $\varepsilon \in [0,1)$. For $\varepsilon=1$, note that \eqref{whatwewant3} also implies that 
\begin{equation}\label{whatwewant5}
        \Prob(V_{m,t_N}  > Q_{m,n, \alpha}(z(t_N), 0)   ) \to 1, \text{ for } N \to \infty.
\end{equation}
Now by definition of $Q_{m,n, \alpha}(z(t_N), 0)$,
\begin{align*}
     \Prob(V_{m,t_N}  > Q_{m,n, \alpha}(z(t_N), 0)   )  \leq \Prob(V_{m,t_N} - \frac{m}{N} z(t_N)  >  0 ),
\end{align*}
and since $V_{m,t_N}=m\hat{F}_m(t_N)$ and $z(t_N)=m \hat{F}_{m}(t_N) +  n \hat{G}_{m}(t_N)$,
\begin{align*}
     \Prob(V_{m,t_N}  > Q_{m,n, \alpha}(z(t_N), 0)   )  &\leq \Prob \left(  \frac{n}{N} m\hat{F}_m(t_N) - \frac{m}{N}   n \hat{G}_{m}(t_N)    >  0 \right)\\
     &=\Prob(  \hat{F}_m(t_N) -  \hat{G}_{m}(t_N)    >  0 ).
\end{align*}

We now can use \emph{exactly} the same argument as in Lemma \ref{ignorebadcase}, (II), to obtain that for a correctly chosen subsequences,
\[
\limsup_{k \to \infty} \Prob(  \hat{F}_{m(k)}(t_{N(k)}) -  \hat{G}_{n(k)}(t_{N(k)})    >  0 ) < 1,
\]
contradicting \eqref{whatwewant3}. Thus finally, \eqref{overalreq} cannot be true if \eqref{NCd} is true.




\end{itemize}

\end{proof}

\noindent \textbf{Technical tools for Proposition \ref{Qfunctions}}: We now introduce two concepts that will help greatly in the proof of Proposition \ref{Qfunctions}. The first concept is that of ``Distributional Witnesses''. We assume to observe two iid samples of independent random elements $X,Y$ with values in $(\X,\A)$ with respective probability measures $P$ and $Q$.
Similar as in \cite{YuvalMixing}, let $\mathfrak{C}$ be the set of all random elements $(\tilde{X}, \tilde{Y})$ with values in $(\X^2,\A^2)$, and such that $\tilde{X} \sim P$ and $\tilde{Y} \sim Q$. Following standard convention, we call 
$(\tilde{X}, \tilde{Y}) \in \mathfrak{C}$ a \emph{coupling} of $P$ and $Q$.
Then $\TV(P,Q)$ may be characterized as  
\begin{align}\label{optimalcoupling}
    \TV(P,Q)&= \inf_{\mathfrak{C}} \Prob(\tilde{X} \neq \tilde{Y}).
\end{align}
This is in turn equivalent to saying that we minimize $\Pi(x \neq y)$ over all joint distributions $\Pi$ on $(\X^2, \A^2)$, that have $ X_{\#} \Pi =P$ and $ Y_{\#} \Pi =Q$. Equation \eqref{optimalcoupling} allows for an interesting interpretation, as detailed (for example) in \cite{YuvalMixing}: The optimal value is attained for a coupling $(X^*,Y^*)$ that minimizes the probability of $X^* \neq Y^*$. The probability that they are different is exactly given by $\TV(P,Q)$. It is furthermore not hard to show that the optimal coupling is given by the following scheme: Let $W \sim \mbox{Bernoulli}(\TV(P,Q))$ and denote by $f$ the density of $P$ and $g$ the density of $Q$, both with respect to some measure on $(X,\A)$, e.g. $P+Q$. If $W=0$, draw a random element $Z$ from a distribution with density $\min(f,g)/(1-\TV(P,Q))$ and set $X^*=Y^*=Z$. If $W=1$, draw $X^*$ and $Y^*$ independently from $(f-g)_{+}/\TV(P,Q)$ and $(g-f)_{+}/\TV(P,Q)$ respectively.

Obviously, $X^*$ and $Y^*$ so constructed are dependent and do not directly relate to the observed $X$, $Y$, which are assumed to be independent. However it holds true that marginally, $X \stackrel{D}{=} X^*$ and $Y \stackrel{D}{=} Y^*$. In particular, given that $W=1$, it holds that $X \stackrel{D}{=} X^*= Y^* \stackrel{D}{=} Y$, or $X|\{W=1\} \stackrel{D}{=} Y| \{W=1 \}$. On the other hand, for $W=0$, the support of $X$ and $Y$ is disjoint. This suggests that the distribution of $X$ and $Y$ might be split into a part that is common to both and a part that is unique. Indeed, the probability measures $P$ and $Q$ can be decomposed in terms of three probability measures $H_P$, $H_Q$, $H_{P,Q}$ such that
\begin{equation}\label{fundMixture}
P = \lambda H_P + (1-\lambda)H_{P,Q} \hspace{0.25cm}\text{and} \hspace{0.25cm} Q = \lambda H_Q + (1-\lambda)H_{P,Q},\end{equation}

\noindent where the mixing weight is $\lambda = \TV(P,Q)$. 

Viewed through the lens of random elements, these decompositions allow us to view the generating mechanism of sampling from $P$ and $Q$ respectively as equivalent to sampling from the mixture distributions in (\ref{fundMixture}). Indeed we associate to $X$ (equivalently for $Y$) the latent binary indicator $W^{P}$, which takes value $1$ if the component specific to $P$, $H_P$, is "selected" and zero otherwise. As before, it holds by construction $\mathbb{P}(W^{P}=1) = \TV(P,Q)$. Intuitively an observation $X$ with $W^{P}=1$ reveals the distribution difference of $P$ with respect to $Q$. This fact leads to the following definition:

\begin{definition}[Distributional Witness]\label{Witdef}
An observation $X$ from $P$ with latent realization $W^{P}=1$ in the representation of $P$ given by \eqref{fundMixture} is called a \emph{distributional witness} of the distribution $P$ with respect to $Q$. We denote by $\text{DW}_{m}(P;Q)$ The number of witness observations of $P$ with respect to $Q$ out of $m$ independent observations from $P$.
\end{definition}

The second concept is that of a bounding operation: Let $\bar{\Lambda}_{P} \in \mathbb{N}$, $\bar{\Lambda}_{Q}\in \mathbb{N}$ be numbers \emph{overestimating} the true number of distributional witnesses from $m$ iid samples from $P$ and $n$ iid samples from $Q$, i.e.
\begin{equation}\label{witnessdomination}
    \bar{\Lambda}_{P} \geq \Lambda_{P} := \text{DW}_{m}(P;Q), \ \bar{\Lambda}_{Q}\geq \Lambda_{Q} := \text{DW}_{n}(Q;P).
\end{equation}
Thus, it could be that $\bar{\Lambda}_{P}, \bar{\Lambda}_{Q}$ denote the true number of witnesses, but more generally, they need to be larger or equal. If $\bar{\Lambda}_{P} > \Lambda_{P}$ or $\bar{\Lambda}_{Q}> \Lambda_{Q}$, a \emph{precleaning} is performed: We randomly choose a set of $\bar{\Lambda}_{P} - \Lambda_{P}$ non-witnesses from the sample of $F$ and $\bar{\Lambda}_{Q}- \Lambda_{Q}$ non-witnesses from the sample of $G$ and mark them as witnesses. Thus we artificially increase the number of witnesses left and right to $\bar{\Lambda}_{P}$, $\bar{\Lambda}_{Q}$. Given this sample of witnesses and non-witnesses and starting simultaneously from the first and last order statistics $Z_{(1)}$ and $Z_{(N)}$, for $i \in \{1,\ldots, N \}$ in the combined sample, we do:
\begin{itemize}
    \item[(1)] If $i < \bar{\Lambda}_{P}$ and $Z_{(i)}$ is \emph{not} a witness from $F$, replace it by a witness from $F$, randomly chosen out of all the remaining $F$-witnesses in $\{Z_{(i+1)}, \ldots Z_{(N)}\}$. Similarly, if $i < \bar{\Lambda}_{Q}$ and $Z_{(N-i+1)}$ is \emph{not} a witness from $G$, replace it by a witness from $G$, randomly chosen out of all the remaining $G$-witnesses in $\{Z_{(1)}, \ldots Z_{(N-i)}\}$.
    \item[(2)] Set $i=i+1$.
\end{itemize}
We then repeat (1) and (2) until $i=\max\{\bar{\Lambda}_{P},  \bar{\Lambda}_{Q}\}$.

\begin{figure}
\begin{center}
\begin{tikzpicture}[scale=0.7, every node/.style={scale=0.7}]]
\draw (0.75,0) ellipse (0.25cm and 0.25cm);
\draw[blue, thick] (0.75,-0.25) -- (0.75,0.25);
\draw[blue, thick] (0.5,0) -- (1,0);
\draw (1,0) -- (1.5,0);

\draw (1.75,0) ellipse (0.25cm and 0.25cm);
\draw (2,0) -- (2.5,0);
\draw (2.5,0.25) -- (3,0.25) -- (3,-0.25) -- (2.5,-0.25) -- (2.5,0.25);
\draw (3,0) -- (3.5,0);
\draw (3.75,0) ellipse (0.25cm and 0.25cm);
\draw[blue, thick] (3.75,-0.25) -- (3.75,0.25);
\draw[blue, thick] (3.5,0) -- (4,0);

\draw (4,0) -- (4.5,0);
\draw (4.5,0.25) -- (5,0.25) -- (5,-0.25) -- (4.5,-0.25) -- (4.5,0.25);
\draw (5,0) -- (5.5,0);
\draw (5.5,0.25) -- (6,0.25) -- (6,-0.25) -- (5.5,-0.25) -- (5.5,0.25);
\draw (6,0) -- (6.5,0);
\draw (6.75,0) ellipse (0.25cm and 0.25cm);
\draw (7,0) -- (7.5,0);
\draw (7.75,0) ellipse (0.25cm and 0.25cm);
\draw (8,0) -- (8.5,0);
\draw (8.5,0.25) -- (9,0.25) -- (9,-0.25) -- (8.5,-0.25) -- (8.5,0.25);
\draw (9,0) -- (9.5,0);
\draw (9.5,0.25) -- (10,0.25) -- (10,-0.25) -- (9.5,-0.25) -- (9.5,0.25);
\draw[blue, thick] (9.75,-0.25) -- (9.75,0.25);
\draw[blue, thick] (9.5,0) -- (10,0);

\draw (10,0) -- (10.5,0);
\draw (10.5,0.25) -- (11,0.25) -- (11,-0.25) -- (10.5,-0.25) -- (10.5,0.25);
\draw[blue, thick]  (10.75,-0.25) -- (10.75,0.25);
\draw[blue, thick] (10.5,0) -- (11,0);

\draw (11,0) -- (11.5,0);
\draw (11.75,0) ellipse (0.25cm and 0.25cm);
\draw (12,0) -- (12.5,0);
\draw (12.5,0.25) -- (13,0.25) -- (13,-0.25) -- (12.5,-0.25) -- (12.5,0.25);
\draw[blue, thick] (12.75,-0.25) -- (12.75,0.25);
\draw[blue, thick] (12.5,0) -- (13,0);

\draw (13,0) -- (13.5,0);
\draw (13.75,0) ellipse (0.25cm and 0.25cm);
\draw (14,0) -- (14.5,0);
\draw (14.5,0.25) -- (15,0.25) -- (15,-0.25) -- (14.5,-0.25) -- (14.5,0.25);



\draw (0.75,-2) ellipse (0.25cm and 0.25cm);
\draw[blue, thick] (0.75,-2.25) -- (0.75,-1.75);
\draw[blue, thick] (0.5,-2) -- (1,-2);

\draw (1,-2) -- (1.5,-2);
\draw (1.75,-2) ellipse (0.25cm and 0.25cm);
\draw (2,-2) -- (2.5,-2);
\draw (2.5,-1.75) -- (3,-1.75) -- (3,-2.25) -- (2.5,-2.25) -- (2.5,-1.75);
\draw (3,-2) -- (3.5,-2);
\draw (3.75,-2) ellipse (0.25cm and 0.25cm);
\draw[blue, thick] (3.75,-2.25) -- (3.75,-1.75);
\draw[blue, thick] (3.5,-2) -- (4,-2);

\draw (4,-2) -- (4.5,-2);
\draw (4.5,-1.75) -- (5,-1.75) -- (5,-2.25) -- (4.5,-2.25) -- (4.5,-1.75);
\draw[red, thick] (4.75,-2.25) -- (4.75,-1.75);
\draw[red, thick] (4.5,-2) -- (5,-2);

\draw (5,-2) -- (5.5,-2);
\draw (5.5,-1.75) -- (6,-1.75) -- (6,-2.25) -- (5.5,-2.25) -- (5.5,-1.75);
\draw (6,-2) -- (6.5,-2);
\draw (6.75,-2) ellipse (0.25cm and 0.25cm);
\draw (7,-2) -- (7.5,-2);
\draw (7.75,-2) ellipse (0.25cm and 0.25cm);
\draw[red, thick] (7.75,-2.25) -- (7.75,-1.75);
\draw[red, thick] (7.5,-2) -- (8,-2);

\draw (8,-2) -- (8.5,-2);
\draw (8.5,-1.75) -- (9,-1.75) -- (9,-2.25) -- (8.5,-2.25) -- (8.5,-1.75);
\draw (9,-2) -- (9.5,-2);

\draw (9.75,-2) ellipse (0.25cm and 0.25cm);
\draw[blue, thick] (9.75,-2.25) -- (9.75,-1.75);
\draw[blue, thick] (9.5,-2) -- (10,-2);

\draw (10,-2) -- (10.5,-2);
\draw (10.5,-1.75) -- (11,-1.75) -- (11,-2.25) -- (10.5,-2.25) -- (10.5,-1.75);
\draw[blue, thick] (10.75,-2.25) -- (10.75,-1.75);
\draw[blue, thick] (10.5,-2) -- (11,-2);

\draw (11,-2) -- (11.5,-2);
\draw (11.75,-2) ellipse (0.25cm and 0.25cm);
\draw[red, thick] (11.75,-2.25) -- (11.75,-1.75);
\draw[red, thick] (11.5,-2) -- (12,-2);

\draw (12,-2) -- (12.5,-2);
\draw (12.5,-1.75) -- (13,-1.75) -- (13,-2.25) -- (12.5,-2.25) -- (12.5,-1.75);
\draw[blue, thick] (12.75,-2.25) -- (12.75,-1.75);
\draw[blue, thick] (12.5,-2) -- (13,-2);

\draw (13,-2) -- (13.5,-2);
\draw (13.75,-2) ellipse (0.25cm and 0.25cm);
\draw (14,-2) -- (14.5,-2);
\draw (14.5,-1.75) -- (15,-1.75) -- (15,-2.25) -- (14.5,-2.25) -- (14.5,-1.75);




\draw (0.75,-4) ellipse (0.25cm and 0.25cm);
\draw[blue, thick] (0.75,-3.75) -- (0.75,-4.25);
\draw[blue, thick] (0.5,-4) -- (1,-4);

\draw (1,-4) -- (1.5,-4);
\draw (1.75,-4) ellipse (0.25cm and 0.25cm);
\draw[red, thick] (1.75,-3.75) -- (1.75,-4.25);
\draw[red, thick] (1.5,-4) -- (2,-4);

\draw (2,-4) -- (2.5,-4);

\draw (2.75,-4) ellipse (0.25cm and 0.25cm);
\draw[blue, thick] (2.75,-3.75) -- (2.75,-4.25);
\draw[blue, thick] (2.5,-4) -- (3,-4);

\draw (3,-4) -- (3.5,-4);

\draw (3.75,-4) ellipse (0.25cm and 0.25cm);
\draw[blue, thick] (3.75,-3.75) -- (3.75,-4.25);
\draw[blue, thick] (3.5,-4) -- (4,-4);

\draw (4,-4) -- (4.5,-4);
\draw (4.75,-4) ellipse (0.25cm and 0.25cm);
\draw[red, thick] (4.75,-3.75) -- (4.75,-4.25);
\draw[red, thick] (4.5,-4) -- (5,-4);

\draw (5,-4) -- (5.5,-4);
\draw (5.75,-4) ellipse (0.25cm and 0.25cm);
\draw (6,-4) -- (6.5,-4);
\draw (6.5,-3.75) -- (7,-3.75) -- (7,-4.25) -- (6.5,-4.25) -- (6.5,-3.75);

\draw (7,-4) -- (7.5,-4);
\draw (7.5,-3.75) -- (8,-3.75) -- (8,-4.25) -- (7.5,-4.25) -- (7.5,-3.75);
\draw (8,-4) -- (8.5,-4);
\draw (8.75,-4) ellipse (0.25cm and 0.25cm);
\draw (9,-4) -- (9.5,-4);

\draw (9.5,-3.75) -- (10,-3.75) -- (10,-4.25) -- (9.5,-4.25) -- (9.5,-3.75);

\draw (10,-4) -- (10.5,-4);
\draw (10.5,-3.75) -- (11,-3.75) -- (11,-4.25) -- (10.5,-4.25) -- (10.5,-3.75);
\draw[blue, thick] (10.75,-3.75) -- (10.75,-4.25);
\draw[blue, thick] (10.5,-4) -- (11,-4);

\draw (11,-4) -- (11.5,-4);
\draw (11.5,-3.75) -- (12,-3.75) -- (12,-4.25) -- (11.5,-4.25) -- (11.5,-3.75);
\draw[red, thick] (11.75,-3.75) -- (11.75,-4.25);
\draw[red, thick] (11.5,-4) -- (12,-4);
\draw (12,-4) -- (12.5,-4);
\draw (12.5,-3.75) -- (13,-3.75) -- (13,-4.25) -- (12.5,-4.25) -- (12.5,-3.75);
\draw[blue, thick] (12.75,-3.75) -- (12.75,-4.25);
\draw[blue, thick] (12.5,-4) -- (13,-4);

\draw (13,-4) -- (13.5,-4);
\draw (13.75,-4) ellipse (0.25cm and 0.25cm);
\draw (14,-4) -- (14.5,-4);
\draw (14.5,-3.75) -- (15,-3.75) -- (15,-4.25) -- (14.5,-4.25) -- (14.5,-3.75);



\draw (0.75,-6) ellipse (0.25cm and 0.25cm);
\draw[blue, thick] (0.75,-5.75) -- (0.75,-6.25);
\draw[blue, thick] (0.5,-6) -- (1,-6);

\draw (1,-6) -- (1.5,-6);
\draw (1.75,-6) ellipse (0.25cm and 0.25cm);
\draw[red, thick] (1.75,-5.75) -- (1.75,-6.25);
\draw[red, thick] (1.5,-6) -- (2,-6);

\draw (2,-6) -- (2.5,-6);

\draw (2.75,-6) ellipse (0.25cm and 0.25cm);
\draw[blue, thick] (2.75,-5.75) -- (2.75,-6.25);
\draw[blue, thick] (2.5,-6) -- (3,-6);

\draw (3,-6) -- (3.5,-6);

\draw (3.75,-6) ellipse (0.25cm and 0.25cm);
\draw[blue, thick] (3.75,-5.75) -- (3.75,-6.25);
\draw[blue, thick] (3.5,-6) -- (4,-6);

\draw (4,-6) -- (4.5,-6);
\draw (4.75,-6) ellipse (0.25cm and 0.25cm);
\draw[red, thick] (4.75,-5.75) -- (4.75,-6.25);
\draw[red, thick] (4.5,-6) -- (5,-6);

\draw (5,-6) -- (5.5,-6);
\draw (5.75,-6) ellipse (0.25cm and 0.25cm);
\draw (6,-6) -- (6.5,-6);
\draw (6.5,-5.75) -- (7,-5.75) -- (7,-6.25) -- (6.5,-6.25) -- (6.5,-5.75);

\draw (7,-6)-- (7.5,-6);
\draw (7.5,-5.75) -- (8,-5.75) -- (8,-6.25) -- (7.5,-6.25) -- (7.5,-5.75);
\draw (8,-6) -- (8.5,-6);
\draw (8.75,-6) ellipse (0.25cm and 0.25cm);
\draw (9,-6) -- (9.5,-6);

\draw (9.5,-5.75) -- (10,-5.75) -- (10,-6.25) -- (9.5,-6.25) -- (9.5,-5.75);
\draw (10,-6) -- (10.5,-6);

\draw (10.75,-6) ellipse (0.25cm and 0.25cm);
\draw (11,-6) -- (11.5,-6);
\draw (11.5,-5.75) -- (12,-5.75) -- (12,-6.25) -- (11.5,-6.25) -- (11.5,-5.75);

\draw (11,-6) -- (11.5,-6);
\draw (11.5,-5.75) -- (12,-5.75) -- (12,-6.25) -- (11.5,-6.25) -- (11.5,-5.75);
\draw (12,-6) -- (12.5,-6);
\draw (12.5,-5.75) -- (13,-5.75) -- (13,-6.25) -- (12.5,-6.25) -- (12.5,-5.75);
\draw[red, thick] (12.75,-5.75) -- (12.75,-6.25);
\draw[red, thick] (12.5,-6) -- (13,-6);
\draw (13,-6) -- (13.5,-6);
\draw (13.5,-5.75) -- (14,-5.75) -- (14,-6.25) -- (13.5,-6.25) -- (13.5,-5.75);
\draw[blue, thick] (13.75,-5.75) -- (13.75,-6.25);
\draw[blue, thick] (13.5,-6) -- (14,-6);
\draw (14,-6) -- (14.5,-6);
\draw (14.5,-5.75) -- (15,-5.75) -- (15,-6.25) -- (14.5,-6.25) -- (14.5,-5.75);
\draw[blue, thick] (14.75,-5.75) -- (14.75,-6.25);
\draw[blue, thick] (14.5,-6) -- (15,-6);

\end{tikzpicture}
\caption{Illustration of the bounding operation. The first row from above is the original order statistics shown as circles (coming from $F$) and squares (coming from $G$). Witnesses are indicated by blue crosses. In the second, randomly chosen non-witnesses are added to the list of witnesses left and right, indicated by red, until the number of witnesses is $\bar{\Lambda}_{P}$ and $\bar{\Lambda}_{Q}$. In the final two rows, the witnesses of $F$ and $G$ are pushed to the left and right respectively, such that the original order of the non-witnesses in the second row is kept intact.}
\label{boundingoperationillustration}
\end{center}
\end{figure}

This operation is quite intuitive: we move from the left to the right and exchange points that are not witnesses from $F$ (i.e. either non-witnesses or witnesses from $G$), with witnesses from $F$ that are further to the right. This we do, until all the witnesses from $F$ are aligned in the first $\bar{\Lambda}_{P}$ positions. We also do the same for the witnesses of $G$ in the other direction of the order statistics. Figure \ref{boundingoperationillustration} illustrates this operation. Implementing the same counting process that produced $V_{m,z}$ in the original sample leads to a new counting process $z \mapsto \bar{V}_{m,z}$. Lemma \ref{barVmzproperties} collects some properties of this process, which is now much more well-behaved than the original $V_{m,z}$.

\begin{lemma} \label{barVmzproperties}
$\bar{V}_{m,z}$ obtained from the bounding operation above has the following properties:
\begin{itemize}
    \item[(i)] $\mathbb{P}\left(\forall z \in J_{m,n}: \bar{V}_{m,z} \geq V_{m,z}\right) = 1$, i.e. it stochastically dominates $V_{m,z}$.
        \item[(ii)] It increases linearly with slope 1 for the first $\bar{\Lambda}_{P}$ observations and stays constant for the last $\bar{\Lambda}_{Q}$ observations.
    \item[(iii)] If $\bar{\Lambda}_{P} < m$ and $\bar{\Lambda}_{Q}< n$ and for $z \in \{\bar{\Lambda}_{P}+1, \ldots, N-\bar{\Lambda}_{Q}-1   \}$, it factors into $\bar{\Lambda}_{P}$ and a process $\tilde{V}_{m -\bar{\Lambda}_{P} , z-\bar{\Lambda}_{P}}$, with
\begin{equation}\label{middlepart}
    \tilde{V}_{m -\bar{\Lambda}_{P}, z-\bar{\Lambda}_{P}} \sim \mbox{Hypergeometric}\left( z-\bar{\Lambda}_{P}, m + n -\bar{\Lambda}_{P} -\bar{\Lambda}_{Q}, m -\bar{\Lambda}_{P} \right).
\end{equation}
\end{itemize}
\end{lemma}

\begin{proof}
(i) follows, as $\bar{V}_{m,z}$ only counts observations from $F$ and these counts
can only increase when moving the witnesses to the left. (ii) follows directly from the bounding operation, through \eqref{witnessdomination}.

(iii) According to our assumptions, we deal with the order statistics of two independent iid samples $(X_1, W^{X}_1), \ldots (X_m, W^{X}_m)$ and $(Y_1, W^{Y}_1), \ldots (Y_n, W^{Y}_n)$, with $X|W^{X}=1$ being equal in distribution to $Y|W^{Y}=1$. We consider their order statistics $(Z_{(1)}, W^{Z}_1), \ldots (Z_{(N)}, W^{Z}_{N})$. In the precleaning step, we randomly choose $\bar{\Lambda}_{P} - \Lambda_{P}$ $i$ such that $W^{P}_i=0$ and $\bar{\Lambda}_{Q}- \Lambda_{Q}$ $j$ such that $W^{P}_j=0$ and flip their values such that $W^{P}_i=1$ and $W^{Y}_j=1$. Let $\mathcal{I}(\bar{\Lambda}_{P}, \bar{\Lambda}_{Q})$ denote the index set $ \{i: W^{P}_i=1 \text{ or } W^{Y}_i=1 \}$ and let $\mathcal{I}^c:=\mathcal{I}(\bar{\Lambda}_{P}, \bar{\Lambda}_{Q})^c=\{1,\ldots, N \} \setminus \mathcal{I}(\bar{\Lambda}_{P}, \bar{\Lambda}_{Q})$. ``Deleting'' all observations, we remain with the order statistics $(Z_{(i)})_{i \in \mathcal{I}^c}$. By construction, up to renaming the indices, we obtain an order statistics $Z_{(1)}, \ldots, Z_{(N-\bar{\Lambda}_{P}- \bar{\Lambda}_{Q})}$ drawn from the common distribution $H_{P,Q}$. Therefore the counting process $V_{\mathcal{I}, z}=(m-\bar{\Lambda}_{P}) \hat{F}(Z_{(z)})$ is a hypergeometric process.
\end{proof}

 \end{document}